\documentclass[12pt]{amsart}

\usepackage{subfig}

\usepackage[bbgreekl]{mathbbol}

\usepackage{amssymb,amsfonts,amsmath}

\usepackage{graphicx}
\usepackage{amssymb}
\usepackage{epstopdf}

\usepackage[all,cmtip]{xy}

\newcommand{\iinfty}{{\mathchoice
{\begin{minipage}{.15in}\includegraphics[width=.12in]{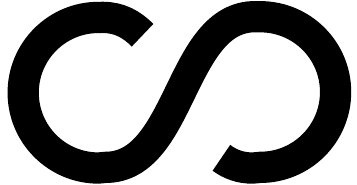}\end{minipage}}
{\begin{minipage}{.10in}\includegraphics[width=.10in]{infty2.pdf}\end{minipage}}
{\begin{minipage}{.08in}\includegraphics[width=.08in]{infty2.pdf}\end{minipage}}
{\begin{minipage}{.08in}\includegraphics[width=.08in]{infty2.pdf}\end{minipage}}
}}

\theoremstyle{plain}
\newtheorem{thm}{Theorem}
\newtheorem{cor}[thm]{Corollary}
\newtheorem{lem}[thm]{Lemma}
\newtheorem{conj}[thm]{Conjecture}
\newtheorem{prop}[thm]{Proposition}
\theoremstyle{definition}
\newtheorem{defn}[thm]{Definition}
\newtheorem{rem}[thm]{Remark}

\renewcommand{\int}{\operatorname{int}}
\newcommand{\sign}{\operatorname{sign}}
\newcommand{\Ker}{\operatorname{Ker}}

\newcommand{\M}{\operatorname{M}}

\newcommand{\Arf}{\operatorname{Arf}}
\newcommand{\Bing}{\operatorname{Bing}}

\newcommand\W{\text{\sf W}}

\newcommand\sL{\text{\sf L}}
\newcommand\sD{\text{\sf D}}
\newcommand\sK{\text{\sf K}}
\newcommand\sY{\text{\sf Y}}

\newcommand{\Z}{\mathbb{Z}}
\newcommand{\N}{\mathbb{N}}

\newcommand{\bM}{\mathbb{M}}
\newcommand{\bW}{\mathbb{W}}
\newcommand{\bL}{\mathbb{L}}

\newcommand{\cT}{\mathcal{T}}
\newcommand{\cV}{\mathcal{V}}
\newcommand{\cW}{\mathcal{W}}

\newcommand{\sra}{\twoheadrightarrow}

\begin{document}

\title{Milnor invariants and twisted Whitney towers}

\begin{abstract}
This paper describes the relationship between the first non-vanishing {\em Milnor invariants} of a classical link and the intersection invariant of a \emph{twisted Whitney tower}. This is a certain 2-complex in the $4$--ball, built from immersed disks bounded by the given link in the $3$--sphere together with finitely many `layers' of Whitney disks.

The \emph{intersection invariant} is a higher-order generalization of the intersection number between two immersed disks in the $4$--ball, well known to give the linking number of the link on the boundary, which measures intersections among the Whitney disks and the disks bounding the given link, together with information that measures the twists (framing obstructions) of the Whitney disks. 

This interpretation of Milnor invariants as higher-order intersection invariants plays a key role in our classifications \cite{CST0,WTCCL} of both the framed and twisted Whitney tower filtrations on link concordance. Here we show how to realize the \emph{higher-order Arf invariants}, 
which also play a role in the classifications, and derive new geometric characterizations of links with vanishing length $\leq 2k$ Milnor invariants.
\end{abstract}

\author[J. Conant]{James Conant}
\email{jconant@math.utk.edu}
\address{Dept. of Mathematics, University of Tennessee, Knoxville, TN}

\author[R. Schneiderman]{Rob Schneiderman}
\email{robert.schneiderman@lehman.cuny.edu}
\address{Dept. of Mathematics and Computer Science, Lehman College, City University of New York, Bronx, NY}

\author[P. Teichner]{Peter Teichner}
 \email{teichner@mac.com}
\address{Dept. of Mathematics, University of California, Berkeley, CA and}
\address{Max-Planck-Institut f\"ur Mathematik, Bonn, Germany}

\keywords{Milnor invariants, Whitney towers, twisted Whitney disk, link concordance, higher-order Arf invariants, trees, gropes, $k$-slice}

\maketitle

\section{Introduction}\label{sec:intro}
In his early work \cite{M1,M2}, John Milnor showed how to extract invariants of classical links from nilpotent quotients of the link group. 
Roughly speaking, Milnor observed that the linking numbers $\mu_L(i,j)\in \Z$ of an oriented link 
$L=L_1\cup L_2\cup\cdots\cup L_m\subset S^3$  
vanish if and only if the link group $\pi_1(S^3 \smallsetminus L)$ is isomorphic {\em modulo 3-fold commutators} to the free group on $m$ generators (the link group of the trivial link). Using this isomorphism, Milnor defined his triple invariants $\mu_L(i,j,k)\in \Z$ which vanish if and only if $\pi_1(S^3 \smallsetminus L)$ is free {\em modulo 4-fold commutators}. Iterating, he obtained a filtration on the set $\bL$ of oriented links in the $3$--sphere:
$$
\dots \subseteq \bM_{3} \subseteq \bM_{2} \subseteq \bM_{1} \subseteq \bL,
$$
where a link $L$ lies in $\bM_n$ if and only if  for all $k\leq n$ the Milnor invariants $\mu_L(i_0,i_1,\dots,i_k)\in\Z$ are defined and vanish. This in turn is equivalent to $\pi_1(S^3 \smallsetminus L)$ being free modulo $(n+2)$-fold commutators, and then the next set of Milnor invariants of {\em length $(n+2)$} are defined (via Magnus expansions of the longitudes, thought of as elements in the free group). We refer to Section~\ref{subsec:intro-Milnor-review} for a precise definition and some history on Milnor invariants. 

In this paper, we will provide a geometric interpretation of Milnor invariants in terms of the intersection invariants of \emph{twisted Whitney towers}. These are certain 2-complexes in the $4$--ball, built on immersed disks bounded by the link $L$ by recursively adding layers of Whitney disks which pair intersections among the previous layers. The intersection invariants measure ``higher-order intersections'' among the Whitney disks, as well as certain relative Euler numbers of their normal bundles, and the relationship with the Milnor invariants provides obstructions to ``raising the order'' of a Whitney tower (see Figures~\ref{fig:Borro-rings},~\ref{fig:Whitehead-122-tower},~\ref{fig:Whitehead-12infty-tower} and Section~\ref{sec:w-towers}). This will show precisely how Milnor invariants are related to the failure of the \emph{Whitney move} \cite{Wh}.

As observed in \cite{KM}, the general failure of the Whitney move in smooth $4$--manifolds goes back to Rohlin's theorem \cite{Ro}. It was dramatically underlined by Donaldson's restrictions on the intersection form of a smooth 4--manifold \cite{Do}. 
Freedman's celebrated recovery \cite{F} of the Whitney move in the settings of surgery and the s-cobordism theorem for topological $4$--manifolds with ``good'' fundamental group was built on a notion of infinitely iterated towers of disks pioneered by Casson \cite{Ca2}. 


In \cite{WTCCL} we studied the {\em twisted Whitney tower filtration}: 
$$
\dots \subseteq \bW^\iinfty_{3} \subseteq \bW^\iinfty_{2} \subseteq \bW^\iinfty_{1} \subseteq \bL,
$$
where a link $L$ lies in $\bW^\iinfty_n$ if and only if $L$ bounds a twisted Whitney tower $\cW$ in the $4$--ball which is of {\em order $n$}, meaning very roughly that $\cW$ consists of $n$ layers of Whitney disks on top of the immersed disks bounding the components of $L$; for details see Section~\ref{sec:w-towers}. 
The ``twist'' symbol $\iinfty$ in our notation stands for the fact that some of the Whitney disks in a twisted Whitney tower are allowed to be \emph{twisted} (i.e.~allowed to have non-zero relative normal Euler number), as opposed to being \emph{framed}.

\begin{thm}\label{thm:pre-Milnor}
There is an inclusion $ \bW^\iinfty_{n}\subseteq \bM_n$ of filtrations. Moreover, all length $(n+2)$ Milnor invariants of $L\in \bW^\iinfty_{n}$ are defined and can be computed from the {\em intersection invariant} $\tau^\iinfty_n(\cW)$ of any order $n$ twisted Whitney tower $\cW$ bounded by $L$. 
\end{thm}
The second statement will be made precise in Theorem~\ref{thm:Milnor invariant}, which describes exactly how Milnor invariants are measured by the higher-order intersections and Whitney disk twistings that determine the invariant $\tau^\iinfty_n$. The theorem also works for an order $0$ twisted Whitney tower 
$\cW$, which by definition is just a union of immersed disks (which are oriented consistently with the orientation of $L$): $\tau_0^\iinfty(\cW)$ counts the transverse intersections between pairs of those disks -- well-known to equal the linking numbers $\mu_L(i,j)$ of the link $L$ -- and also detects the induced framings on the link components, considered as ``self-linking'' numbers $\mu_L(i,i)$ 
(see the proof of Theorem~\ref{thm:Milnor invariant} in Section~\ref{sec:Milnor-thm-proof}). 

Any $L$ in $\bM_{1}$ bounds an order $1$ twisted Whitney tower consisting of immersed disks $D_i$ together with Whitney disks pairing all intersections among the $D_i$. Matsumoto \cite{Ma} showed
that $\mu_L(i,j,k)$ can be computed from the interior intersections between the Whitney disks and the $D_i$ (see Section~\ref{subsec:main-proof-order-1} and Figure~\ref{order-one-longitude-A1-fig}). 
Figure~\ref{fig:Borro-rings} shows the easiest case where one can explicitly see why a Whitney move cannot be used to find 3 disjointly embedded disks in $B^4$ whose boundaries form the Borromean rings.
\begin{figure}[h]
\centerline{\includegraphics[scale=.43]{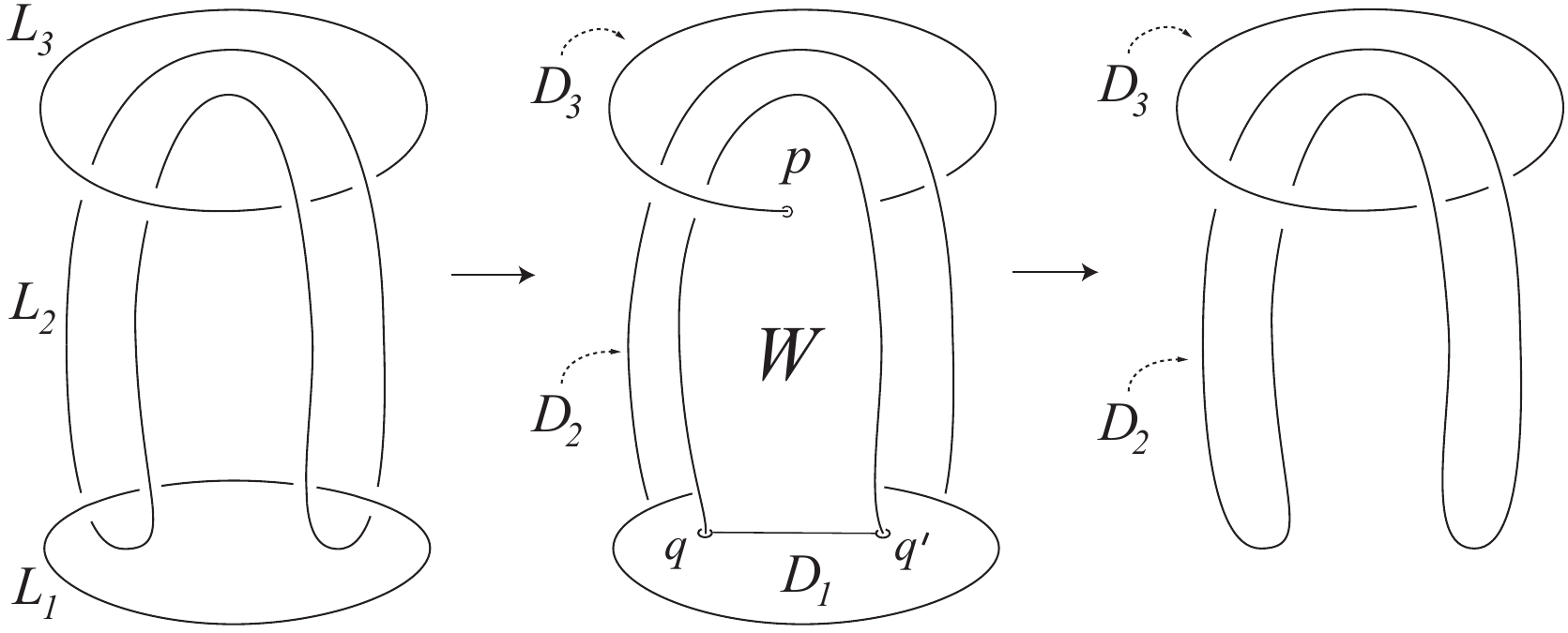}}
\caption{Moving radially into $B^4$ from left to right, this sequence of pictures shows the Borromean Rings 
$L=L_1\cup L_2\cup L_3\subset S^3=\partial B^4$ bounding an order $1$ twisted Whitney tower $\cW\subset B^4$. It consists of embedded
disks $D_i$ with $\partial D_i=L_i$ together with a Whitney disk $W$ that pairs the two intersection points 
$q$ and $q'$ between $D_1$ and $D_2$. 
The disk $D_1$ consists of the `horizontal' opaque disk in the lower part of the middle picture extended by an annular collar back to $L_1$ in the left picture. The disks $D_2$ and $D_3$ consist of the embedded annuli which are the products of $L_2$ and $L_3$ with the radial coordinate into $B^4$ together with embedded disks (not shown) extending further into $B^4$ that bound the unlink in the right picture. The embedded Whitney disk $W$ is completely contained in the middle picture and has a single intersection point $p$ with $D_3$. This unpaired `higher-order' intersection point $p$ shows that $\mu_L(1,2,3)=\pm 1$ and prevents a Whitney move that would promote $\cW$ to a collection of slice disks for $L$.
}
\label{fig:Borro-rings}
\end{figure}

Our results generalize the correspondence between Milnor invariants and higher-order intersection invariants of Whitney towers to all lengths and all orders.
We find some surprising subtleties related to Whitney disk twistings,
with the first occurence in order $n=2$, as illustrated in Section~\ref{subsec:intro-classification-sketch} (see Figure~\ref{fig:Whitehead-12infty-tower}).

To express more precisely the difference between the two filtrations $\bM_n$ and $\bW^\iinfty_n$, we work with the associated graded {\em groups}. More precisely, in Section~\ref{subsec:intro-Milnor-review} we present a {\em universal} order $n$ Milnor invariant $\mu_n:\bM_n\to \sD_n$ with values in a free abelian group of known rank. It carries exactly the information of all Milnor invariants of length $(n+2)$ and can be expressed in the language of trivialent trees that are closely related to the intersection invariant $\tau_n^\iinfty$. Moreover, it has the additivity properties  
\[
\mu_n(L\#_b L') = \mu_n(L) + \mu_n(L')  \quad \text{ and } \quad \mu_n(-L) = - \mu_n(L)
\]
where $L\#_b L'$ is any choice of band connected sum of $m$-component links $L,L'\in \bM_n$ (that are separated by an embedded 2--sphere) and $-L$ denotes the link $L$ mirror reflected and with orientations flipped. 
This additivity follows most easily from the translation of Milnor invariants into Massey products for the link complement, shown independently by
Turaev \cite{Tu} and Porter \cite{Po}.

Since $\bM_{n+1}$ consists exactly of those links $L\in \bM_n$ with $\mu_n(L)=0$, we think of the group $\sD_n$ as being in some sense the ``quotient group'' of $\bM_n$ by $\bM_{n+1}$.  The same band connected sum operation on links makes $\bW^\iinfty_n$ into a finitely generated abelian group $\W_n^\iinfty$ after taking the quotient by the equivalence relation of \emph{order $n+1$ twisted Whitney tower concordance} \cite{WTCCL},
and the following theorem gives ``three quarters'' of our main classification result:
\begin{thm} \label{thm:twisted-three-quarters-classification1} 
The universal Milnor invariants induce surjections $\mu_n\colon \W^\iinfty_n\sra \sD_n$ for all $n$, and isomorphisms for all $n\equiv 0,1,3\,\mod 4$.
\end{thm}
This result follows from Theorem~\ref{thm:pre-Milnor} via \cite{WTCCL,CST3}, as will be explained in Section~\ref{subsec:intro-classification-sketch} (see Corollary~\ref{cor:Milnor invariants} and Theorem~\ref{thm:twisted-three-quarters-classification}). 
For the remaining ``quarter'' of orders, we need certain \emph{higher-order Arf invariants} $\Arf_k$, as described in \cite{CST0,WTCCL} and sketched below in Section~\ref{subsec:intro-higher-order-arf}.
The $\Arf_k$ are link concordance invariants which represent obstructions to \emph{framing} a twisted Whitney tower bounded by a link, with $\Arf_1$ corresponding to the classical Arf invariants of the link components.  
Although the precise values of $\Arf_k$ are not known for $k>1$ (they form a quotient of a known finite $2$-torsion group -- see Definition~\ref{def:Arf-k}), we show in this paper that these invariants measure the difference between two natural notions
of ``nilpotent approximation'' of slice disks for a link: \emph{$k$-slice} and \emph{geometrically $k$-slice} (Theorem~\ref{thm:geo-k-slice-equals-odd-W}). The construction of boundary links realizing the range of $\Arf_k$ (Lemma~\ref{lem:Bing}) yields two new geometric characterizations
of links with vanishing length $\leq 2k$ Milnor invariants, as described in Theorem~\ref{thm:geo-k-slice-after-sums} and Theorem~\ref{thm:mega-k-slice}.

The rest of this introduction develops enough material to give more refined statements of our results.

\subsection{Order $n$ Milnor invariants.}\label{subsec:intro-Milnor-review} 
For a group $\Gamma$, denote by $\Gamma_n$ the \emph{lower central series} of commutator subgroups of $\Gamma$, defined inductively by $\Gamma_1:=\Gamma$ and
$\Gamma_{n+1}:=[\Gamma,\Gamma_n]$. If $L\subset S^3$ is an $m$-component link such that all of its
longitudes lie in the $(n+1)$-th term of the
lower central series of the link group $\pi_1(S^3\setminus L)_{n+1}$, then a choice of meridians induces an isomorphism
$$
\frac{\pi_1(S^3\setminus L)_{n+1}}{\pi_1(S^3\setminus L)_{n+2}}\cong\frac{F_{n+1}}{F_{n+2}}
$$
where $F=F(m)$ is the free group on $\{x_1,x_2,\ldots,x_m\}$.

Let $\sL=\sL(m)$ denote the free Lie algebra (over the ground ring $\Z$) on generators $\{X_1,X_2,\ldots,X_m\}$. It is $\N$-graded, $\sL=\oplus_n \sL_n$, where the degree~$n$ part $\sL_n$ is the
additive abelian group of length $n$ brackets, modulo Jacobi
identities and self-annihilation relations $[X,X]=0$.
The multiplicative abelian group $\frac{F_{n+1}}{F_{n+2}}$ of
length $n+1$ commutators is isomorphic to
$\sL_{n+1}$, with $x_i$ mapping to  $X_i$ and commutators mapping to Lie brackets.

In this setting, denote by $\mu_n^i(L)$ the image of the $i$-th longitude in  $\sL_{n+1}$ under the above isomorphisms and
define the \emph{order $n$ Milnor invariant}
$\mu_n(L)$ by
$$
\mu_n(L):=\sum_{i=1}^m X_i \otimes \mu_n^i(L) \in \sL_1 \otimes
\sL_{n+1}
$$
Then $\mu_n(L)$ is the first non-vanishing \emph{universal} Milnor invariant, in the sense that it determines {\em all} Milnor invariants of \emph{length} $n+2$ (with repeating indices allowed) \cite{M1,M2}. 
The original $\bar{\mu}$-invariants are computed from the longitudes via the Magnus expansion as integers modulo indeterminacies coming from invariants of shorter length. Since in this paper we are only concerned with first non-vanishing $\mu$-invariants, we do not use the ``bar'' notation $\bar{\mu}$.

It turns out that  $\mu_n(L)$ actually lies in the kernel $\sD_n=\sD_n(m)$
of the bracket map $\sL_1 \otimes \sL_{n+1}\rightarrow \sL_{n+2}$ 
(e.g.~by ``cyclic symmetry'' \cite{FT2}). We observe that $\sL_n$ and $\sD_n$ are free abelian groups of known ranks:
The rank $r_n=r_n(m)$ of $\sL_n(m)$ is given by $r_n=\frac{1}{n}\sum_{d|n}\M(d)m^{n/d}$,
with $\M$ denoting the M\"{o}bius function \cite[Thm.5.11]{MKS}; and the rank of $\sD_n(m)$ is equal to 
$mr_{n+1}-r_{n+2}$, first identified by Orr as the number of independent (integer) ${\mu}$-invariants
of length $n+2$ in \cite{O}.

Milnor's \emph{$\bar{\mu}$-invariants} have inspired a significant amount of research over the past 50-plus years.
Work of Stallings \cite{St} implied that Milnor invariants are concordance invariants \cite{Ca}. Realization results of Cochran \cite{C,C1} and Orr \cite{O} provided geometric interpretations of $\bar{\mu}$-invariants \cite{C,IO} and
supported the development of more ``universal'' approaches, including Habbeger and Lin's classification of link homotopy \cite{HL1}  
via an Artin representation characterization of Milnor invariants for string links (see also \cite{HL2}) as well as a growing number of interpretations related to quantum invariants (e.g.~\cite{BN,CM,HM,HO,HabMe,MY}). 
There are even connections with algebraic number theory \cite{Morishita} and molecular biology \cite{Gad}.
See e.g.~chapter 10 of \cite{Hil} for more regarding Milnor invariants.

\subsection{Intersection invariants for (twisted) Whitney towers}\label{subsec:intro-W-tower-int-trees}

By \cite{WTCCL,ST2}, an order $n$ (twisted) Whitney tower $\cW$ built on properly immersed disks in the $4$--ball has an {\em intersection invariant} $\tau_n(\cW)$ (resp.~$\tau_n^\iinfty(\cW)$) which is defined by associating unitrivalent trees to the unpaired higher-order intersection points (and twisted Whitney disks) in $\cW$ (e.g.~Figures~\ref{fig:Whitehead-122-tower} and \ref{fig:Whitehead-12infty-tower}). For the convenience of the reader we briefly describe next the target groups of these invariants. Relevant details on (twisted) Whitney towers and their intersection invariants are presented in Section~\ref{sec:w-towers} below.

\begin{defn}\label{def:Tau}
 In this paper, a {\em tree} will always refer to a finite oriented unitrivalent tree, where the {\em orientation} of a tree is given by cyclic orderings of the adjacent edges around each trivalent vertex. The \emph{order} of a tree is the number of trivalent vertices.
Univalent vertices will usually be labeled from the set $\{1,2,3,\ldots,m\}$ indexing the link components, and we consider trees up to isomorphisms preserving these labelings.
Define $\cT=\cT(m)$ to be the free abelian group on such trees, modulo the antisymmetry (AS) and Jacobi (IHX) relations: 
\begin{figure}[h]
\centerline{\includegraphics[scale=.65]{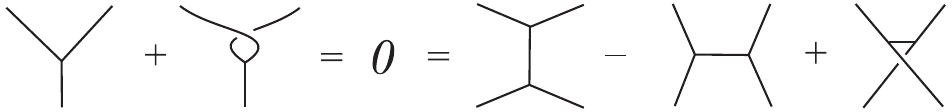}}
         \label{fig:ASandIHXtree-relations}
\end{figure}
\end{defn}

Since the AS and IHX relations are homogeneous with respect to order, $\cT$ inherits a grading $\cT=\oplus_n\cT_n$, where $\cT_n=\cT_n(m)$ is the free abelian group on order $n$ trees, modulo AS and IHX relations.

In the Whitney tower obstruction theory of \cite{ST2}, the order $n$ intersection invariant 
$\tau_n(\cW)\in\cT_n$  assigned to each order~$n$ (framed) Whitney tower $\cW$ is defined by summing the trees associated to unpaired intersections 
in $\cW$
(see Figure~\ref{fig:Whitehead-122-tower} for an example).  
The tree orientations are induced by Whitney disk orientations via a convention that corresponds to the AS relations (Section~\ref{subsec:w-tower-orientations}), and  
the IHX relations can be realized geometrically by controlled maneuvers on Whitney towers as described in \cite{CST,S2}. 
It follows from the obstruction theory that a link bounds an order~$n$ Whitney tower $\cW$ with $\tau_n(\cW)=0$ 
if and only if it bounds an order $n+1$ Whitney tower \cite[Thm.2]{ST2}.

For twisted Whitney towers of order $n$, the intersection invariant $\tau_n^\iinfty$ introduced in \cite{WTCCL}
also assigns certain \emph{twisted trees} ($\iinfty$-trees) to Whitney disks which are not framed (see Section~\ref{subsec:twisted-w-disks} below), and takes values in the following graded groups:

\begin{defn}[\cite{WTCCL}]\label{def:T-infty}
In odd orders, the group $\cT^{\iinfty}_{2k-1}$ is the quotient of $\cT_{2k-1}$ by the \emph{boundary-twist relations}: 
\[
 i\,-\!\!\!\!\!-\!\!\!<^{\,J}_{\,J}\,\,=\,0
\] 
where $J$ ranges over all order $k-1$ subtrees. 

A \emph{rooted} tree has one unlabeled univalent vertex, called the \emph{root}. 
For any rooted tree $J$ we define the corresponding {\em $\iinfty$-tree}, denoted by $J^\iinfty$, by labeling the root univalent vertex with the twist-symbol ``$\iinfty$'':
$$
J^\iinfty := \iinfty\!-\!\!\!- J 
$$ 
In even orders, the group $\cT^{\iinfty}_{2k}$ is the free abelian group on trees of order $2k$ 
and $\iinfty$-trees of order $k$, modulo the following four types of relations:
\begin{enumerate}
     \item AS and IHX relations on order $2k$ trees
   \item \emph{symmetry} relations: $(-J)^\iinfty = J^\iinfty$
  \item \emph{twisted IHX} relations: $I^\iinfty=H^\iinfty+X^\iinfty- \langle H,X\rangle $
   \item {\em interior-twist} relations: $2\cdot J^\iinfty=\langle J,J\rangle $
\end{enumerate}
Here the \emph{inner product} $\langle T_1,T_2\rangle $ of two order $k$ rooted trees $T_1$ and $T_2$ is defined by gluing the roots together to get an unrooted tree of order $2k$. The AS and IHX relations are as pictured above, but they only apply to ordinary trees (not to $\iinfty$-trees). 
\end{defn}
These relations have the following geometric origins:
The \emph{symmetry relation} corresponds to the fact that the relative 
Euler number of a Whitney disk is independent of its orientation, with the minus sign denoting that the cyclic edge-orderings at the trivalent vertices of $-J$ differ from those of $J$ at an odd number of vertices.
As explained in \cite{WTCCL}, the \emph{twisted IHX relation} corresponds to the effect of performing a Whitney move in the presence of a twisted Whitney disk, and the \emph{interior-twist relation} corresponds to the fact that creating a $\pm1$ self-intersection in a Whitney disk changes its twisting by $\mp 2$.

The main reason why this is a good definition comes from the following result, which is Theorem~2.10 of \cite{WTCCL}: A link $L\subset S^3$ bounds an order $n$ twisted Whitney tower $\cW\subset B^4$ with $\tau_n^\iinfty(\cW)=0\in\cT^\iinfty_n$ if and only if $L$ bounds an order $n+1$ twisted Whitney tower in $B^4$ .

\subsection{The summation maps $\eta_n$}\label{subsec:intro-eta-map}
The connection between $\tau^\iinfty_n(\cW)$ and $\mu_n(L)$ is via a homomorphism $\eta_n : \cT^\iinfty_n \to \sD_n$, which is best explained when we regard rooted trees of order $n$ as elements in $\sL_{n+1}$ in the usual way:
For $v$ a univalent vertex of an order $n$ tree $t$ as in Definition~\ref{def:Tau}, denote by $B_v(t)\in\sL_{n+1}$ the Lie bracket of generators $X_1,X_2,\ldots,X_m$ determined by the formal bracketing of indices
which is gotten by considering $v$ to be a root of $t$.

\begin{defn}\label{def:eta}
Denoting the label of a univalent vertex $v$ by $\ell(v)\in\{1,2,\ldots,m\}$, the
map $\eta_n:\cT^\iinfty_n\rightarrow \sL_1 \otimes \sL_{n+1}$
is defined on generators by
$$
\eta_n(t):=\sum_{v\in t} X_{\ell(v)}\otimes B_v(t)
\quad \, \,
\mbox{and}
\quad \, \,
\eta_n(J^\iinfty):= \frac{1}{2}\,\eta_n(\langle J,J \rangle)
$$
The first sum is over all univalent vertices $v$ of $t$, and the second expression lies in $\sL_1 \otimes \sL_{n+1}$ 
because the coefficients of $\eta_n(\langle J,J \rangle)$ are even. Here $J$ is a rooted tree of order $k$ for $n=2k$.
\end{defn}
For example, 
\[
 \begin{array}{lll}
\eta_1\!\left( {\scriptstyle 1}-\!\!\!\!\!-\!\!\!<^{\,\,3}_{\,\,2}\,\right)  &= \quad  X_1\otimes\,-\!\!\!\!\!-\!\!\!<^{\,\,3}_{\,\,2}\,\,\ + \quad X_2\otimes \,{\scriptstyle 1}\!-\!\!\!\!\!-\!\!\!<^{\,\,3}_{\,\,}\,\,\  + \quad X_3\otimes \, {\scriptstyle 1}\!-\!\!\!\!\!-\!\!\!<^{\,\,}_{\,\,2}\,  
\\
&= \quad  X_1\otimes\,[X_2,X_3]+ X_2\otimes [X_3,X_1]  +  X_3\otimes [X_1,X_2].
\end{array}
\]
And similarly,
 \[
\begin{array}{llc}
\eta_2\!\left( {\scriptstyle \iinfty}-\!\!\!\!\!-\!\!\!<^{\,\,2}_{\,\,1}\,\right)  &= \frac{1}{2}\,\eta_2\!\left(^{\,\,1}_{\,\,2}>\!\!\!-\!\!\!\!\!-\!\!\!<^{\,\,2}_{\,\,1}\right) \\ 
&= X_1\,\otimes _{\,\,2}\!>\!\!\!-\!\!\!\!\!-\!\!\!<^{\,\,2}_{\,\,1} \,+\,  X_2\,\otimes  ^{\,\,1}\!>\!\!\!-\!\!\!\!\!-\!\!\!<^{\,\,2}_{\,\,1}   
\\
&= X_1\otimes\,[X_2,[X_1,X_2]]+ X_2\otimes [[X_1,X_2],X_1]. 
\end{array}
\]
In Section~\ref{subsec:lemma-eta-well-defined-proof} we check that $\eta_n$ is well-defined and maps $\cT_n^\iinfty$ onto $\sD_n$.

We can now make Theorem~\ref{thm:pre-Milnor} precise as follows:
\begin{thm}\label{thm:Milnor invariant}
If $L$ bounds a twisted Whitney tower $\cW$ of order $n$, then the order $k$ Milnor invariants $\mu_k(L)$ vanish for $k<n$ and
\[
\mu_n(L) = \eta_n \circ\tau^\iinfty_n(\cW) \in \sD_n
\]
\end{thm}
The proof of Theorem~\ref{thm:Milnor invariant} given in Section~\ref{sec:Milnor-thm-proof} uses a geometric interpretation of the maps  $\eta_n$
which shows precisely how higher-order intersections and Whitney disk twistings correspond to sums of iterated commutators.
Closely related maps over the rationals appear in Habegger and Masbaum's work on the Kontsevich invariant showing that the Milnor invariants 
are the universal finite type (rational) concordance invariants \cite{HM}. Levine's work on homology cylinders \cite{L1,L2,L3} led him to study analogous maps over the integers, and our resolution in \cite{CST3} of his main conjecture accomplished an important step in the classification of both the twisted and framed Whitney tower filtrations on link concordance, as discussed in the next two subsections.
It should also be noted that the relationship between Milnor invariants and trees originally goes back to Cochran's method
of constructing links realizing given (integer) Milnor invariants \cite{C,C1}.

\subsection{Computing the graded groups associated to the filtrations} \label{subsec:intro-classification-sketch}
As a prelude to the description of our results on the higher-order Arf invariants, we briefly recall from \cite{WTCCL} the computation of the
groups $\W_n^\iinfty$ and $\W_n$ associated to the twisted and framed Whitney tower filtrations on link concordance. 
Here the groups $\W_n$ in the framed setting
are defined analogously to $\W_n^\iinfty$: The equivalence relation of  (framed) {\em Whitney tower concordance of order $n+1$} on the set $\bW_n$ of $m$-component links that bound (framed) order~$n$ Whitney towers in the $4$--ball defines $\W_n$ as the associated graded quotient, which also turns out to be a finitely generated abelian group under the band connected sum operation for all $n$.

In \cite{WTCCL} we constructed framed and twisted \emph{realization} epimorphisms
$$
R_n : \cT_n \sra\W_n \quad \textrm{and} \quad R^\iinfty_n : \cT^\iinfty_n \sra\W^\iinfty_n
$$
which send $g\in\cT^{(\iinfty)}_n$ to the equivalence class of links bounding an order $n$ (twisted) Whitney tower $\cW$ with $\tau^{(\iinfty)}_n(\cW)=g$. These maps are defined similarly to Cochran's construction for realizing Milnor invariants (\cite[Sec.7]{C} and \cite[Thm.3.3]{C1}) by ``Bing-doubling along trees'' and taking internal band sums if indices repeat: The Hopf link realizes the order zero tree ${\scriptstyle 1}-\!\!\!-\!\!\!- {\scriptstyle 2}$ corresponding to a transverse intersection between disks in $B^4$ bounded by the components. To realize higher-order generators, iterated (untwisted) Bing-doublings are performed according to the branching of the tree, until we obtain the correct tree but with non-repeating indices labeling the univalent vertices. For example, a single Bing-doubling on the Hopf link yields the Borromean rings realizing ${\scriptstyle 1}-\!\!\!\!\!-\!\!\!<^{\,\,3}_{\,\,2}$. Finally, we take internal band sums according to which indices repeat. For example, one internal band sum may take the Borromean rings to the Whitehead link defining $R_1({\scriptstyle 1}-\!\!\!\!\!-\!\!\!<^{\,\,2}_{\,\,2})$ (see Figure~\ref{fig:Whitehead-122-tower}).

\begin{figure}[h]
\centerline{\includegraphics[scale=.44]{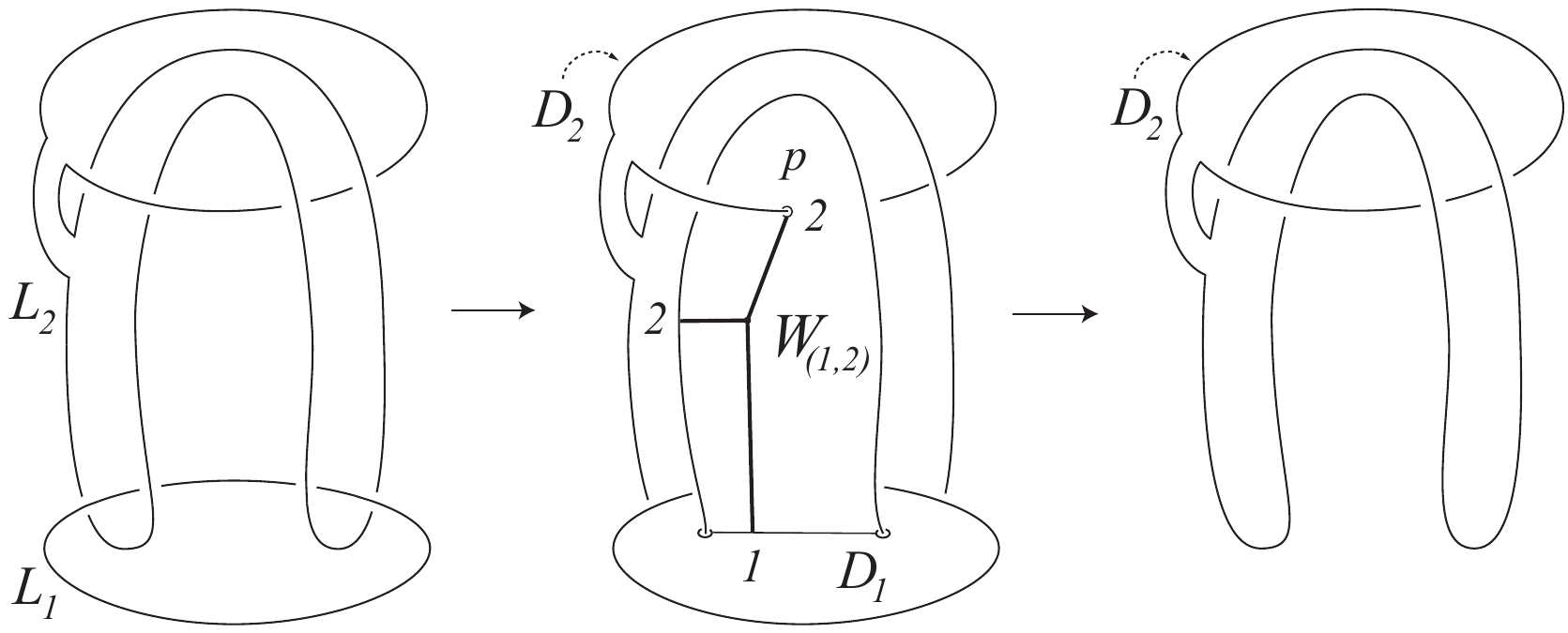}}
\caption{Moving radially into $B^4$ from left to right, this sequence of pictures shows a Whitehead link $L\subset S^3$ bounding an order $1$ Whitney tower $\cW\subset B^4$ with $\tau_1(\cW)={\scriptstyle 1}-\!\!\!\!\!-\!\!\!<^{\,\,2}_{\,\,2}$. 
The left picture shows $L=L_1\cup L_2\subset S^3$ formed by an internal band sum on the Borromean rings. 
Moving into $B^4$ in the middle and right pictures, $L_1$ and $L_2$ bound embedded disks $D_1$ and $D_2$ with an embedded Whitney disk $W_{(1,2)}$ pairing $D_1\cap D_2$. The disk $D_1$ consists of the `horizontal' opaque disk in the lower part of the middle picture extended by an annular collar back to $L_1$ in the left picture. The embedded Whitney disk $W_{(1,2)}$ is completely contained in the middle picture, which also contains the tree $t_p={\scriptstyle 1}-\!\!\!\!\!-\!\!\!<^{\,\,2}_{\,\,2}$ associated to the unpaired intersection point $p$ between $D_2$ and the interior of $W_{(1,2)}$.
The disk $D_2$ is the union of the embedded annulus (visible in all three pictures) which is the product of $L_2$ with the radial coordinate into $B^4$, together with an embedded disk (not shown) extending further into $B^4$ that bounds the parallel of $L_2$ in the right picture.}
\label{fig:Whitehead-122-tower}
\end{figure}

\begin{figure}[h]
\centerline{\includegraphics[scale=.44]{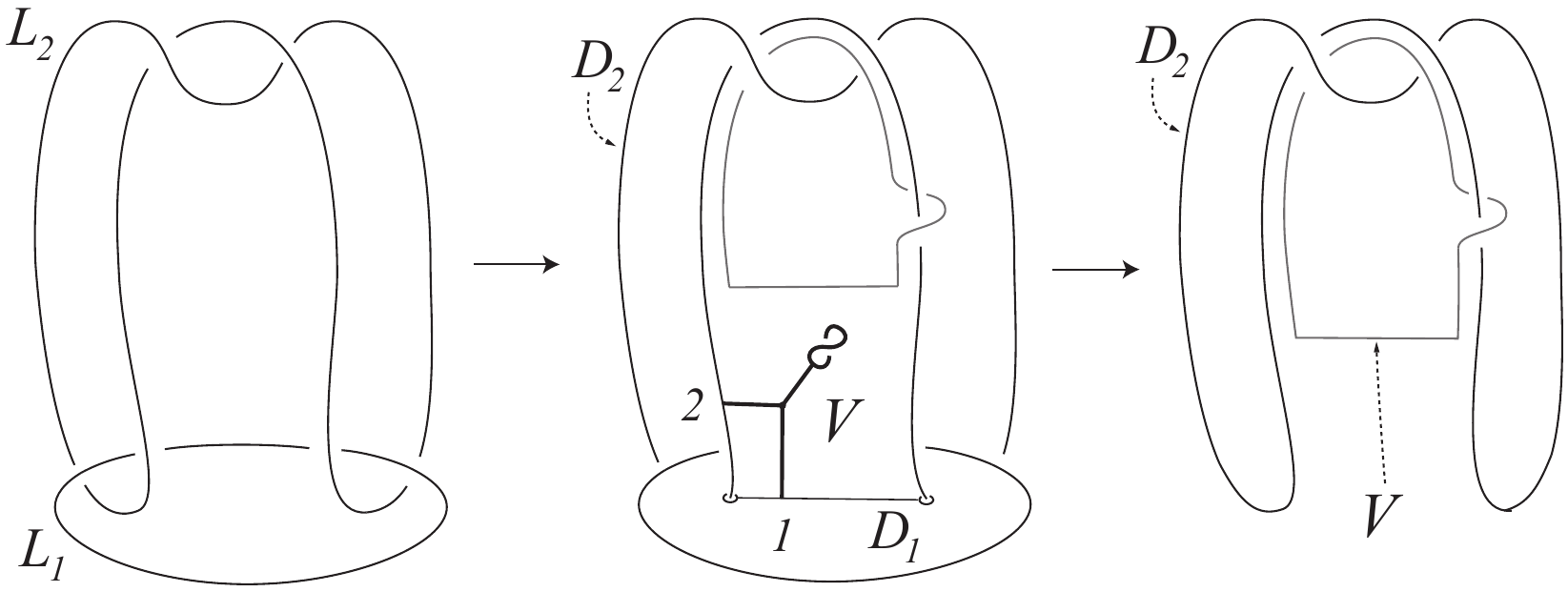}}
\caption{The Whitehead link $L$ of Figure~\ref{fig:Whitehead-122-tower} also
bounds an order $2$ twisted Whitney tower $\cV$ consisting of embedded disks $D_1$ and $D_2$ bounded by $L_1$ and $L_2$, together with an embedded twisted Whitney disk $V$ pairing the intersections $D_1\cap D_2$. Moving into $B^4$ from left to right, the left picture again shows 
$L=L_1\cup L_2\subset S^3$, and the disks $D_1$ and $D_2$ are described just as in Figure~\ref{fig:Whitehead-122-tower}, with $D_1$ contained in the left and middle pictures, while circle slices of $D_2$ are visible in all three pictures with the rest of $D_2$ (not shown) extending further into $B^4$ as an embedded disk bounding the indicated unlink component in the right hand picture. 
The embedded twisted Whitney disk $V$ containing its associated twisted tree is partly visible in the middle picture which shows an opaque annular region of $V$ that contains the boundary of $V$. The rest of $V$ extends further into $B^4$ and is visible as the indicated component of the unlink in the right picture together with an embedded disk (not shown) bounding this component which is disjoint from the part of $D_2$ that bounds the other unlink component. That $V$ is twisted will be shown in Section~\ref{sec:Bing-Arf-proof}.} 
\label{fig:Whitehead-12infty-tower}
\end{figure}

For $\iinfty$-trees, the starting point is the 1-framed unknot as $R_0^\iinfty({\scriptstyle \iinfty}\!-\!\!\!- {\scriptstyle 1})$. The first Bing-doubling has to be a twisted one, giving a Whitehead link as $R_2^\iinfty({\scriptstyle \iinfty}-\!\!\!\!\!-\!\!\!<^{\,\,1}_{\,\,2})$. Notice that this means that the Whitehead link bounds two different Whitney towers, one framed of order~1 and the other twisted of order~2 (Figure~\ref{fig:Whitehead-12infty-tower}). 
This is the easiest example illustrating how the Milnor invariants interact differently with the framed and twisted Whitney tower filtrations:
The Whitehead link $L$ bounds a twisted Whitney tower $\cV$ of order $2$, but not one of order 3, as detected by 
$\tau_2^\iinfty(\cV)={\iinfty}-\!\!\!\!\!-\!\!\!<^{\,\,1}_{\,\,2}\,\,\neq 0\in\cT^\iinfty_2$, corresponding to the non-triviality of $\mu_2(L)$. However, even though the longitudes of $L$ are length three commutators (so that $\mu_2(L)$ is defined), $L$ does not bound an order 2 \emph{framed} Whitney tower; as detected by $\tau_1(\cW)={\scriptstyle 1}-\!\!\!\!\!\!-\!\!\!<^{\,\,2}_{\,\,2} \,\neq 0\in\cT_1$, corresponding to the non-trivial \emph{Sato-Levine invariant} \cite{Sa} which is the projection of $\mu_2(L)$. This phenomenon
occurs in all odd orders of the framed Whitney tower filtration, as described by the \emph{higher-order Sato-Levine invariants} defined in \cite{WTCCL}.

The maps $R^\iinfty_n$ bound the size of the abelian groups $\W^\iinfty_n$ from above, and the following corollary of Theorem~\ref{thm:Milnor invariant} shows that Milnor invariants give a lower bound: 
\begin{cor} \label{cor:Milnor invariants} There is a commutative diagram of epimorphisms
\[
\xymatrix{
\cT^\iinfty_n \ar@{->>}[r]^{R_n^\iinfty} \ar@{->>}[rd]_{\eta_n} & \W^\iinfty_n \ar@{->>}[d]^{\mu_n}\\
& \sD_{n}
}\]
\end{cor}

As a consequence of our proof \cite{CST3} of 
the combinatorial conjecture of Levine formulated in \cite{L2}, we have the following partial classification of $\W^\iinfty_n$:
\begin{thm}[\cite{WTCCL}] \label{thm:twisted-three-quarters-classification} The maps $\eta_n:\cT^\iinfty_n \to \sD_n$ are isomorphisms for $n\equiv 0,1,3\,\mod 4$.
As a consequence, both the Milnor invariants $\mu_n\colon \W^\iinfty_n\to \sD_n$  and the twisted realization maps $R^\iinfty_n : \cT^\iinfty_n \to\W^\iinfty_n$ are isomorphisms for these orders.
\end{thm}
The remaining cases to complete the classification are more complicated, as can already be seen for $n=2$: In the case $m=1$ of knots, Lemma~\ref{lem:Arf} below shows that the Arf invariant induces an isomorphism $\W^\iinfty_2(1) \cong\Z_2$, whereas all Milnor invariants vanish for knots.

Unlike for $n\equiv 0,1,3\mod 4$, where $\Ker(\eta_n)=0$,  there are some obvious elements in $\Ker(\eta_{4k-2})$, namely those of the form ${\displaystyle\iinfty}\!\!-\!\!\!\!\!\!-\!\!\!\!<^{\,\,J}_{\,\,J}$ for an order $k-1$ rooted tree $J$: These are 2-torsion by the interior-twist and IHX relations in $\cT^\iinfty_{4k-2}$ and hence must map to zero in the (torsion-free) group $\sD_{4k-2}$. In \cite{WTCCL} we also deduce the following result from the affirmation of Levine's conjecture:

\begin{prop}[\cite{WTCCL}]\label{prop:kerEta4k-2}
The map sending $1\otimes J $ to $ \iinfty\,-\!\!\!\!\!-\!\!\!\!<^{\,J}_{\,J}\,\,\in\cT^\iinfty_{4k-2}$ for rooted trees $J$ of order $k-1$ defines an isomorphism:
\[
\Z_2 \otimes \sL_k \cong\Ker(\eta_{4k-2})
\]
\end{prop}
Here the identification of rooted order $k-1$ trees with degree $k$ Lie brackets is as in Section~\ref{subsec:intro-eta-map} above (see the examples following Definition~\ref{def:eta}).
It follows that $\Z_2 \otimes \sL_k$ is also an upper bound on
the kernels of the epimorphisms $R^\iinfty_{4k-2}: \cT^\iinfty_{4k-2}\sra \W^\iinfty_{4k-2}$ and $\mu_{4k-2}: \W^\iinfty_{4k-2} \sra \sD_{4k-2}$, and the calculation of $\W^\iinfty_{4k-2}$ is completed by invariants defined on the kernel of 
$\mu_{4k-2}$ which are the above-mentioned higher-order Arf invariants, as we describe next.

\subsection{Higher-order Arf invariants}\label{subsec:intro-higher-order-arf}
Let us first discuss the situation for order $n=2$. Observe that ${\displaystyle \iinfty}\!\!-\!\!\!\!\!-\!\!\!\!<^{\,\,1}_{\,\,1}$ is not zero in $\cT^\iinfty_2(1)$
but that
$^{\,\,1}_{\,\,1}>\!\!\!-\!\!\!\!-\!\!\!\!\!-\!\!\!<^{\,\,1}_{\,\,1}\,=0$ by the IHX relation; so $\cT^\iinfty_2(1)$ is generated by ${\displaystyle \iinfty}\!\!-\!\!\!\!\!-\!\!\!\!<^{\,\,1}_{\,\,1}$, which is $2$-torsion, and $\tau^\iinfty_2(\cW)$ counts (modulo $2$) the framing obstructions
on the Whitney disks in an order $2$ twisted Whitney tower $\cW$. This is explained in Section~\ref{sec:proof-lem-Arf}, which gives a proof of the following result:
\begin{lem} \label{lem:Arf}
Any knot $K$ bounds a twisted Whitney tower $\cW$ of order $2$ and the classical Arf invariant of $K$ can be identified with the intersection invariant
\[
 \tau^\iinfty_2(\cW) \in \cT^\iinfty_2(1) \cong \Z_2
\]
More generally, the classical Arf invariants of the components of an $m$-component link give an isomorphism
\[
\Arf: \Ker(\mu_2:\W^\iinfty_2\sra \sD_2) \overset{\cong}{\to} (\Z_2\otimes \sL_1) \cong (\Z_2)^m
\]
\end{lem}
This lemma verifies our conjecture $ \W^\iinfty_n \cong\cT^\iinfty_n$ from Conjecture~\ref{conj:Arf-k} below in the case $n=2$, with $\Ker(\eta_2)\cong\Ker(\mu_2) \cong 
(\Z_2)^m$.

Following and expanding upon \cite{WTCCL}, we will now describe a similarly satisfying picture for all orders of the form $n=4k-2$ that takes both the Milnor and Arf invariants into account. 

Let $\sK^\iinfty_{4k-2}$ denote the kernel of $\mu_{4k-2}$. It follows from Corollary~\ref{cor:Milnor invariants} and Proposition~\ref{prop:kerEta4k-2} above that mapping $1\otimes J$ to 
$R^\iinfty_{4k-2}( \iinfty\,-\!\!\!\!\!-\!\!\!<^{\,J}_{\,J}\,\,)$ induces a surjection $\alpha^\iinfty_k: \Z_2 \otimes \sL_k
\sra \sK^\iinfty_{4k-2}$, for all $k\geq 1$.
Denote by $\overline{\alpha^\iinfty_{k}}$ the induced isomorphism on $(\mathbb Z_2\otimes {\sf L}_{k})/\Ker \alpha^\iinfty_{k}$.

\begin{defn}[\cite{WTCCL}]\label{def:Arf-k}  
The \emph{higher-order Arf invariants} are defined by
$$
\Arf_{k}:=(\overline{\alpha^\iinfty_{k}})^{-1}:\sK^\iinfty_{4k-2}\to(\mathbb Z_2\otimes {\sf L}_{k})/\Ker \alpha^\iinfty_{k}
$$
\end{defn}
From Theorem~\ref{thm:twisted-three-quarters-classification}, Proposition~\ref{prop:kerEta4k-2} and Definition~\ref{def:Arf-k} we see that the groups $\W^\iinfty_n$ are computed by the Milnor and higher-order Arf invariants.

We conjectured in \cite{CST0,WTCCL} that $\alpha^\iinfty_k$ is an isomorphism, which would mean that the 
$\Arf_k$ are very interesting new concordance invariants:
\begin{conj}[\cite{CST0,WTCCL}]\label{conj:Arf-k}
 $\Arf_{k}:\sK^\iinfty_{4k-2}\to\mathbb Z_2\otimes {\sf L}_{k}$ is an isomorphism for all $k$.
\end{conj}

Conjecture~\ref{conj:Arf-k} would imply that  
$$
\W_{4k-2}^\iinfty\cong \cT^\iinfty_{4k-2} \cong (\Z_2 \otimes \sL_k) \oplus \sD_{4k-2}
$$
where the second isomorphism (is non-canonical and)
already follows from Corollary~\ref{cor:Milnor invariants} and Proposition~\ref{prop:kerEta4k-2} above \cite[Cor.1.12, Prop.1.14]{WTCCL}.
The statement of Conjecture~\ref{conj:Arf-k} is true for $k=1$, by Lemma~\ref{lem:Arf} above, with $\Arf_1=\Arf$ the classical Arf invariant. It remains an open problem whether $\Arf_k$ is non-trivial for any $k>1$.

We have the following specialization of the Bing-doubling construction discussed above Corollary~\ref{cor:Milnor invariants} which realizes symmetric $\iinfty$-trees of the form ${\displaystyle \iinfty}\!\!-\!\!\!\!\!-\!\!\!<^{\,\,J}_{\,\,J}$.  It starts with the fact that any knot with non-trivial Arf invariant represents $R_2^\iinfty({\displaystyle \iinfty}\!\!-\!\!\!\!\!-\!\!\!<^{\,\,1}_{\,\,1})$ by Lemma~\ref{lem:Arf}, then proceeds by applying untwisted Bing-doublings and internal band sums. This has the effect of symmetrically extending both $1$-labeled branches of the original tree into a higher-order twisted Whitney tower, so that the resulting link can be constructed to realize any ${\displaystyle \iinfty}\!\!-\!\!\!\!\!-\!\!\!<^{\,\,J}_{\,\,J}$ (see Section~\ref{subsec:proof-lemma-Bing}). This idea will be used to derive the geometric interpretations of Milnor invariants given in Theorem~\ref{thm:geo-k-slice-after-sums} and Theorem~\ref{thm:mega-k-slice} below.
\begin{lem} \label{lem:Bing}
Let $J$ be any rooted tree of order $k-1$.
By performing iterated untwisted Bing-doublings and interior band sums on the figure-eight knot $K$, a boundary link $K^J$ can be constructed which bounds a twisted Whitney tower $\cW$ of order $4k-2$ such that
\[
\tau^\iinfty_{4k-2}(\cW)= \iinfty\!\!-\!\!\!\!\!-\!\!\!<^{\,\,J}_{\,\,J}  
\]
\end{lem}
The proof of Lemma~\ref{lem:Bing} given in Section~\ref{subsec:proof-lemma-Bing} can be easily modified to show that this result holds for any knot $K$ with non-trivial classical Arf invariant. It is thus already interesting to ask whether our proposed higher-order Arf invariants $\Arf_k$ can be defined on the cobordism group of boundary links. 
The links $K^J$ of Lemma~\ref{lem:Bing} are known not to be slice by work of J.C. Cha \cite{Cha}, providing evidence supporting our conjecture that $\Arf_k$ is indeed a non-trivial link concordance invariant which represents an obstruction to bounding an order $4k-1$ twisted Whitney tower. 
The following result emphasizes the importance of the first open case $k=2$:
\begin{prop}\label{prop:Arf-2-and-greater}
If $\Arf_2$ is trivial, then $\Arf_k$ is trivial for all $k\geq 2$.
\end{prop}
As explained in Section~\ref{subsec:proof-prop-Arf-2-and-greater}, which contains a proof of Proposition~\ref{prop:Arf-2-and-greater}, the statement that $\Arf_2$ is trivial is equivalent to the existence of an order $7$ twisted Whitney tower $\cW$ bounded by the Bing-double of a figure-eight knot.

\subsection{Geometrically $k$-slice links}\label{subsec:intro-k-slice}
We conclude this introduction with some new geometric characterizations of Milnor invariants and the higher-order Arf invariants.
See Section~\ref{sec:Milnor-geo-k-slice} for proofs of these results.

Recall (e.g.~from \cite{T2}) that the $2$--complexes known as \emph{gropes}, are ``geometric embodiments'' of iterated commutators in the sense that a loop in a topological space represents a $k$-fold commutator in the fundamental group if and only if it extends to a continuous map of a grope of class $k$ (see Section~\ref{section:twisted-towers-and-gropes} and e.g.~\cite{CT1,CT2,CST,FQ,FT2,Kr,KT,S1,T1}). Since Milnor invariants measure how deeply the link longitudes extend into the lower central series of the link group, Milnor invariants obstruct bounding \emph{immersed} gropes in $S^3$ essentially by definition. On the other hand, extracting information on bounding \emph{embedded} gropes in the $4$--ball from the vanishing of Milnor invariants is much more difficult.
Embedded framed gropes have usefully served as ``approximations'' to embedded disks in many topological settings (see e.g.~\cite{T2}).

Perhaps the most notable geometric ``if and only if'' characterization of Milnor invariants to date is the \emph{$k$-slice Theorem}, due to Igusa and Orr:
Expressed in the language of gropes, a link $L\subset S^3$ is said to be \emph{$k$-slice} if the link components $L_i$ bound disjointly embedded (oriented) surfaces $\Sigma_i\subset B^4$ such that a symplectic basis of curves on each $\Sigma_i$ 
bound class $k$ gropes immersed in the complement of $\Sigma := \cup_i\Sigma_i$. 
Via a very careful analysis of the third homology of the nilpotent quotients $F/F_k$ of the (rank $m$) free group $F$, Igusa and Orr \cite{IO} proved the following result.
\begin{thm}[\cite{IO}]
A link $L$ is $k$-slice if and only if $\mu_n(L)=0$ for all $n\leq 2k-2$ (equivalently, all Milnor invariants of length $\leq 2k$ vanish). 
\end{thm}
The $k$-slice condition says that the link components bound certain immersed gropes in $B^4$ whose embedded bottom stage surfaces are ``algebraic approximations'' of slice disks modulo the $k$th term of the lower central series of the link group.

This leads to the very natural notion of \emph{geometrically $k$-slice} links: These are links for which there is a symplectic basis of curves on 
the embedded bounding surfaces $\Sigma\subset B^4$ that bound \emph{disjointly embedded framed} gropes of class $k$ in $B^4\setminus \Sigma$. 
In Section~\ref{sec:Milnor-geo-k-slice} we describe how the techniques of \cite{S1} 
together with the classification of the twisted Whitney tower filtration in \cite{WTCCL}
can be used to give the following result, which shows that 
the higher-order Arf invariants $\operatorname{Arf}_k$ detect the difference between $k$-sliceness and geometric $k$-sliceness:
\begin{thm}\label{thm:geo-k-slice-equals-odd-W} 
$L$ is geometrically $k$-slice if and only if $\mu_n(L)=0$ for all $n\leq 2k-2$ and $\Arf_n(L)=0$ for all $n\leq \frac{1}{2}k$. 
\end{thm}
Combining Theorem~\ref{thm:geo-k-slice-equals-odd-W} together with Corollary~\ref{cor:Milnor invariants}, Proposition~\ref{prop:kerEta4k-2} and Lemma~\ref{lem:Bing} we immediately get:
\begin{thm}\label{thm:geo-k-slice-after-sums}
 A link $L$ has $\mu_n(L)=0$ for all $n\leq 2k-2$ if and only if $L$ is geometrically $k$-slice after a finite number of band sums with boundary links.
 $\hfill\square$
 \end{thm}

It turns out that the operation of taking band sums with boundary links is equivalent to
a certain approximation of being geometrically $k$-slice, as described by the following theorem. The basic observation here is that any curve on a surface in $S^3$ bounds an immersed disk in $B^4$, leading to the surfaces of type $\Sigma''_i$ below, associated to the boundary links in Theorem~\ref{thm:geo-k-slice-after-sums}. 
 We note that the ``only if'' part of the following theorem uses a mild generalization of Theorem~\ref{thm:Milnor invariant}, described in Proposition~\ref{prop:mu=tau-on-immersed-surfaces}.
\begin{thm}\label{thm:mega-k-slice}
A link $L=\cup_i L_i$ has $\mu_n(L)=0$ for all $n\leq 2k-2$
if and only if the link components $L_i$ bound
disjointly embedded surfaces $\Sigma_i$ in the $4$--ball, with each surface a
connected sum of two surfaces $\Sigma'_i$ and $\Sigma''_i$ such that
\begin{enumerate}
\item
 a symplectic basis of curves on $\Sigma'_i$
 bound disjointly embedded framed gropes $G_{i,j}$ of class $k$ in the complement of 
 $\Sigma := \cup_i\Sigma_i$, and 
\item
 a symplectic basis of curves on $\Sigma''_i$ bound immersed disks in the
 complement of 
 $\Sigma\cup G$, where $G$ is the union of all $G_{i,j}$.
\end{enumerate}
\end{thm}
Theorem~\ref{thm:mega-k-slice} is a considerable strengthening of the above Igusa--Orr $k$-slice Theorem: Since the geometric conditions 
in both theorems are equivalent to the vanishing of Milnor's invariants
through order $2k-2$ (length $2k$), one can read our result as saying that the {\em immersed gropes} of class $k$ found by Igusa and Orr can be cleaned up to immersed {\em disks} (these are immersed gropes of arbitrarily high class) or {\em disjointly embedded} gropes of class $k$.
As explained in Section~\ref{sec:Milnor-geo-k-slice}, the higher-order Arf invariants are exactly the obstructions to eliminating the need for the $\Sigma''_i$ and these immersed disks.



{\bf Acknowledgments:}
This paper was partially written while the first two authors were visiting the third author at the Max-Planck-Institut f\"ur Mathematik in Bonn. They all thank MPIM for its stimulating research environment and generous support. The first author was also supported by NSF grant DMS-0604351, and the last author was also supported by NSF grants DMS-0806052 and DMS-0757312. The second author was partially supported by PSC-CUNY research grant PSCREG-41-386 and a grant (\#208938) from the Simons Foundations. Thanks also to the referee for helpful comments regarding the exposition.

\section{Whitney towers}\label{sec:w-towers}
This section recalls the relevant theory of (twisted) Whitney towers as developed in \cite{CST,WTCCL,S1,ST2}. We work in the \emph{smooth oriented} category (with orientations usually suppressed from notation), even though all our results hold in the locally flat topological category by the basic results on topological immersions in Freedman--Quinn \cite{FQ}. In fact, it can be shown that the filtrations  $\mathbb W_n$ and $\mathbb W^\iinfty_n$ are identical in the smooth and locally flat settings. This is because a topologically flat surface can be promoted to a smooth surface at the cost of only creating unpaired intersections of arbitrarily high order (see Remark~2.1 of \cite{WTCCL}).

\subsection{Operations on trees}\label{subsec:trees}
To describe Whitney towers it is convenient to use the bijective correspondence
between formal non-associative bracketings of elements from
the index set $\{1,2,3,\ldots,m\}$ and
rooted trees, trivalent and oriented as in Definition~\ref{def:Tau},
with each univalent vertex labeled by an element from the index set, except
for the \emph{root} univalent vertex which is left unlabeled. 

\begin{defn}\label{def:trees}
Let $I$ and $J$ be two rooted trees.
\begin{enumerate} 
\item The \emph{rooted product} $(I,J)$ is the rooted tree gotten
by identifying the root vertices of $I$ and $J$ to a single vertex $v$ and sprouting a new rooted edge at $v$.
This operation corresponds to the formal bracket (Figure~\ref{inner-product-trees-fig} upper right). The orientation of $(I,J)$ is inherited from those of $I$ and $J$ as well as the order in which they are glued.

\item The \emph{inner product}  $\langle I,J \rangle $ is the
unrooted tree gotten by identifying the roots of $I$ and $J$ to a single non-vertex point.
Note that $\langle I,J \rangle $ inherits an orientation from $I$ and $J$, and that
all the univalent vertices of $\langle I,J \rangle $ are labeled.
(Figure~\ref{inner-product-trees-fig} lower right.)

\item The \emph{order} of a tree, rooted or unrooted, is defined to be the number of trivalent vertices.
\end{enumerate}
\end{defn}
The notation of this paper will not distinguish between a bracketing and its corresponding rooted tree
(as opposed to the notation $I$ and $t(I)$ used in \cite{S1,ST2}).
In \cite{S1,ST2} the inner product is written as a dot-product, and the rooted product
is denoted by $*$.

\begin{figure}[ht!]
        \centerline{\includegraphics[scale=.40]{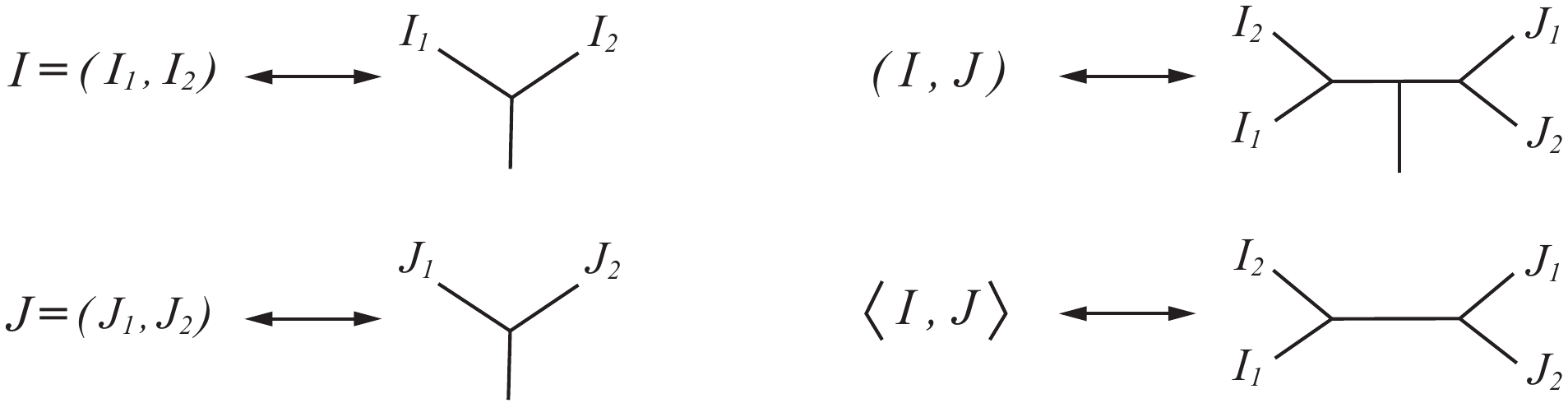}}
        \caption{The \emph{rooted product} $(I,J)$ and \emph{inner product} $\langle I,J \rangle$ of $I=(I_1,I_2)$ and $J=(J_1,J_2)$. All trivalent orientations correspond to a clockwise orientation of the plane.}
        \label{inner-product-trees-fig}
\end{figure}

\subsection{Whitney disks and higher-order intersections}\label{subsec:order-zero-w-towers-and-ints}
\begin{defn}\label{def:framed-order-zero-tower}
A collection $A_1,\ldots,A_m\looparrowright (M,\partial M)$ of 
connected surfaces in a $4$--manifold $M$ is a \emph{Whitney tower of order zero} if the $A_i$ are \emph{properly immersed} in the sense that the boundary is embedded in $\partial M$ and the interior is generically immersed in $M \smallsetminus \partial M$. 
The $A_i$ are also required to be \emph{framed} as discussed in Section~\ref{subsec:twisted-w-disks} below.
\end{defn}

To each \emph{order zero surface} $A_i$ is associated
the order zero rooted tree consisting of an edge with one vertex labeled by $i$, and
to each transverse intersection $p\in A_i\cap A_j$ is associated the order zero
tree $t_p:=\langle i,j \rangle$ consisting of an edge with vertices labelled by $i$ and $j$. Note that
for singleton brackets (rooted edges) we drop the bracket from notation, writing $i$ for $(i)$.

The order 1 rooted $\sY$-tree $(i,j)$, with a single trivalent vertex and two univalent labels $i$ and $j$,
is associated to any Whitney disk $W_{(i,j)}$ pairing intersections between $A_i$ and $A_j$. This rooted tree
can be thought of as being embedded in $M$, with its trivalent vertex and rooted
edge sitting in $W_{(i,j)}$, and its two other edges descending into $A_i$ and $A_j$ as sheet-changing paths. (The cyclic orientation at the trivalent vertex of the bracket  $(i,j)$  corresponds to an orientation of $W_{(i,j)}$ via a convention described below in \ref{subsec:w-tower-orientations}.)

Recursively, the rooted tree $(I,J)$ is associated to any Whitney disk $W_{(I,J)}$ pairing intersections
between $W_I$ and $W_J$ (see left-hand side of Figure~\ref{WdiskIJandIJKint-fig}); with the understanding that if, say, $I$ is just a singleton $i$, then $W_I$ denotes the order zero surface $A_i$.

To any transverse intersection $p\in W_{(I,J)}\cap W_K$ between $W_{(I,J)}$ and any
$W_K$ is associated the tree $t_p:=\langle (I,J),K \rangle$  (see right-hand side of Figure~\ref{WdiskIJandIJKint-fig}).

\begin{figure}[ht!]
        \centerline{\includegraphics[width=120mm]{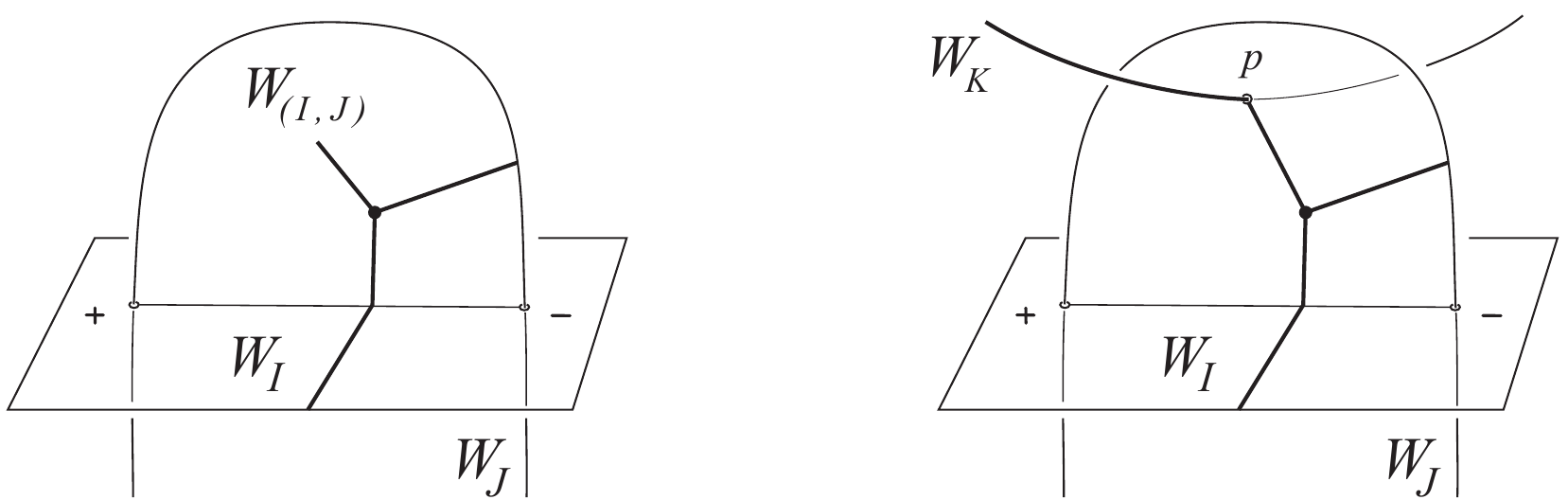}}
        \caption{On the left, (part of) the rooted tree $(I,J)$ associated to a Whitney disk $W_{(I,J)}$. On the right, (part of) the tree $t_p=\langle (I,J),K \rangle$ associated to an intersection $p\in W_{(I,J)}\cap W_K$. Note that $p$ corresponds to where the roots of $(I,J)$ and $K$ are identified to a (non-vertex) point in $\langle (I,J),K \rangle$.}
        \label{WdiskIJandIJKint-fig}
\end{figure}

\begin{defn}\label{def:int-and-Wdisk-order}
The \emph{order of a Whitney disk} $W_I$ is defined to be the order of the rooted tree $I$, and the \emph{order of a transverse intersection} $p$ is defined to be the order of the tree $t_p$.
\end{defn}

\begin{defn}\label{def:framed-tower}
A collection $\cW$ of properly immersed surfaces together with higher-order
Whitney disks is an \emph{order $n$ Whitney tower}
if $\cW$ contains no unpaired intersections of order less than $n$.  
\end{defn}
The Whitney disks in $\cW$ must have disjointly embedded boundaries, and generically immersed interiors.  All Whitney disks and order zero surfaces must also be \emph{framed}, as discussed next.

\subsection{Twisted Whitney disks and framings}\label{subsec:twisted-w-disks}
The normal disk-bundle of a Whitney disk $W$ in $M$ is isomorphic to $D^2\times D^2$,
and comes equipped with a canonical nowhere-vanishing \emph{Whitney section} over the boundary that can be described in the following way: The Whitney disk boundary circle $\partial W$ is the union of two arcs, each lying in a local sheet of the surfaces paired by $W$. The Whitney section is given by pushing $\partial W$  tangentially along one sheet, and normally off of the other sheet (while avoiding the tangential direction of $W$).
See Figure~\ref{Framing-of-Wdisk-fig}, and for more details e.g.~1.7 of \cite{Sc}.
Pulling back the orientation of $M$ with the requirement that the normal disks
have $+1$ intersection with $W$ means the Whitney section determines
a well-defined (independent of the orientation of $W$)
relative Euler number $\omega(W)\in\Z$ which represents the obstruction to extending
the Whitney section across $W$. Following traditional terminology, when $\omega(W)$ vanishes $W$ is said to be \emph{framed}. (Since $D^2\times D^2$ has a unique trivialization up to homotopy, this terminology is only mildly abusive.)
In general when $\omega(W)=k$, we say that $W$ is
$k$-\emph{twisted}, or just \emph{twisted} if the value of $\omega(W)$ is not specified.
So a $0$-twisted Whitney disk is a framed Whitney disk.

\begin{figure}[ht!]
        \centerline{\includegraphics[width=120mm]{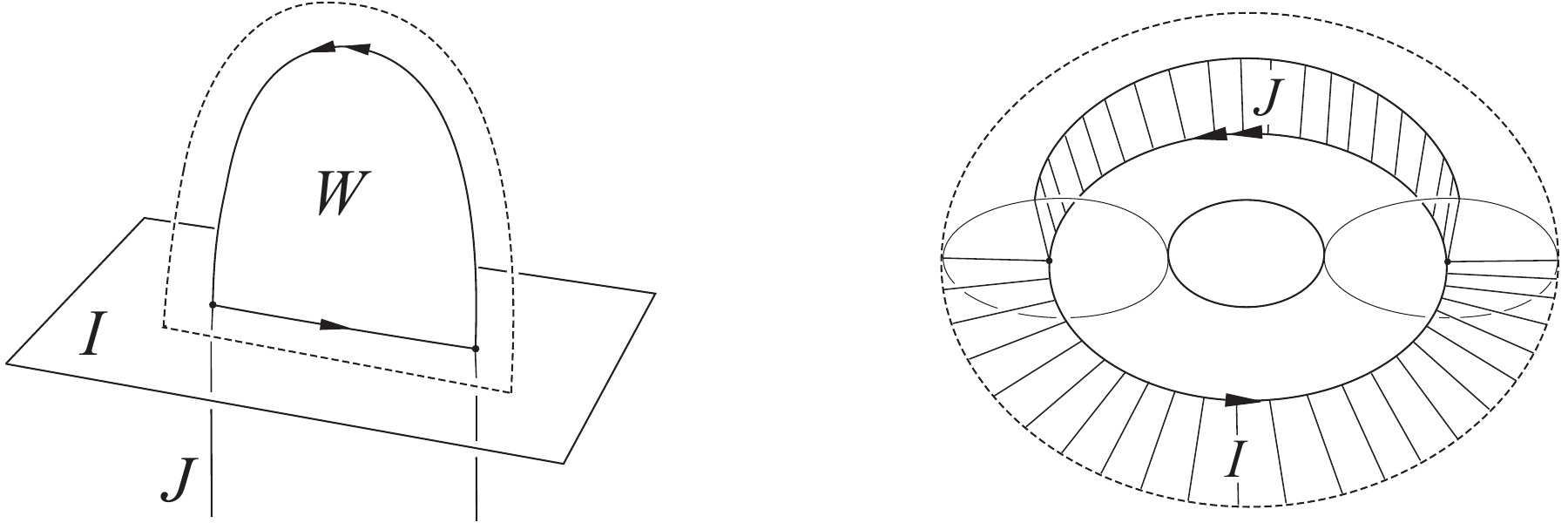}}
        \caption{In a 3-dimensional slice of $4$--space (left),
         the Whitney section over the boundary of a framed Whitney disk $W$ is
         indicated by the dotted loop, where $\partial W$ has been pushed
         tangentially along the $I$-sheet and normally off of the $J$-sheet. On the right is shown an embedding into $3$--space of the normal
         disk-bundle to $W$ over $\partial W$ (a solid torus, shown mostly transparent), 
         with the dotted loop again indicating the Whitney section, 
         and the `thatched' lines indicating parts of the surface sheets. 
         The $I$-labeled thatches indicate the tangential push of $\partial W$ along the $I$-sheet, and the
         $J$-labeled thatches indicate `one side' of the $J$-sheet along $\partial W$.}
        \label{Framing-of-Wdisk-fig}
\end{figure}

Note that for order zero surfaces a {\em framing} of $\partial A_i$ (respectively $A_i$) is by definition a trivialization of the normal bundle of the immersion. If the ambient 
$4$--manifold $M$ is oriented, this is equivalent to an orientation and a nonvanishing normal vector field on $\partial A_i$ (respectively $A_i$).
The twisting $\omega(A_i)\in\Z$ of an order zero surface is also defined when a framing of $\partial A_i$ is given, and the order zero surface $A_i$ is said to be \emph{framed} 
when $\omega(A_i)=0$.


\subsection{Twisted Whitney towers}\label{subsec:intro-twisted-w-towers}
In the definition of an order $n$ Whitney tower given just above (following \cite{CST,S1,S2,ST2})
all Whitney disks and order zero surfaces are required to be framed. It turns out that the natural generalization to twisted Whitney towers involves allowing twisted Whitney disks only in at least ``half the order'' as follows:

\begin{defn}[\cite{WTCCL}]\label{def:twisted-W-towers}
A \emph{twisted Whitney tower of order $0$} is a Whitney tower of order $0$ without any framing requirement 
(a collection of properly immersed surfaces
in a $4$--manifold).

For $k>0$, a \emph{twisted Whitney tower of order $2k-1$} is just a (framed) Whitney
tower of order $2k-1$ as in Definition~\ref{def:framed-tower} above.

For $k>0$, a \emph{twisted Whitney tower of order $2k$} is a Whitney
tower having all intersections of order less than $2k$ paired by
Whitney disks, with all Whitney disks of order less than $k$ required to be framed, but Whitney disks of order at least $k$ allowed to be twisted.
\end{defn}

\begin{rem}\label{rem:framed-is-twisted}
Note that, for any $n$, an order $n$ (framed) Whitney tower
is also an order $n$ twisted Whitney tower. We may
sometimes refer to a Whitney tower as a \emph{framed} Whitney tower to emphasize
the distinction, and will always use the adjective ``twisted'' in the setting of
Definition~\ref{def:twisted-W-towers}.
\end{rem}

\begin{rem}\label{rem:motivate-twisted-towers}
The convention of allowing only order $\geq k$ twisted Whitney disks in order $2k$ twisted Whitney towers will be explained in
Section~\ref{sec:Milnor-thm-proof} where it will be seen that twisted Whitney disks contribute to the link longitudes just as described by the 
definition of the $\eta$-map on $\iinfty$-trees.

In any event, an order $2k$ twisted Whitney tower can always be modified
so that all its Whitney disks of order $>k$ are framed, so the twisted Whitney disks of order equal to $k$ are the ones relevant to the obstruction theory \cite[Sec.4.1]{WTCCL}.  
\end{rem}


 \subsection{Whitney tower orientations}\label{subsec:w-tower-orientations}

Orientations on order zero surfaces in a Whitney tower $\cW$ are fixed, and required to induce the orientations
on their boundaries.
After choosing and fixing orientations on all the Whitney disks in
$\cW$, the associated trees 
are embedded in $\cW$ so that the vertex orientations are induced from
the Whitney disk orientations, with the descending edges of each
trivalent vertex enclosing the \emph{negative intersection point} of the corresponding Whitney disk, as in Figure~\ref{WdiskIJandIJKint-fig}.
(In fact, if a tree $t$ has more than one trivalent vertex corresponding to the same Whitney disk, then
$t$ will only be immersed in $\cW$, but this immersion can be taken to be a local embedding around each trivalent vertex of $t$
as in Figure~\ref{WdiskIJandIJKint-fig}.)

This ``negative corner'' convention, which differs from the
positive corner convention in the earlier papers \cite{CST,ST2} but agrees with all more recent papers on Whitney towers, will turn out to be
compatible with commutator conventions for use in Section~\ref{sec:Milnor-thm-proof}.

With these conventions, different choices of orientations on Whitney disks in $\cW$ correspond to anti-symmetry relations 
(as explained in Section~3.4 of \cite{ST2}).

\subsection{Links bounding (twisted) Whitney towers}\label{subsec:links-bounding-W-towers}
Throughout this paper the statement that a link $L\subset S^3$ bounds an order $n$ (twisted) Whitney tower $\cW\subset B^4$
means that the components of $L$ bound properly immersed disks which are the order $0$ surfaces of $\cW$ as in Definition~\ref{def:framed-tower}
(Definition~\ref{def:twisted-W-towers}), with all conventions as described above.


%


\subsection{Intersection invariants for twisted Whitney towers}\label{subsec:twisted-w-tower-int-invariants}

The intersection invariants for 
twisted Whitney towers take values in the groups $\cT^\iinfty_n$ defined in the introduction (Definition~\ref{def:T-infty}).

Recall from 
Definition~\ref{def:twisted-W-towers} (and Remark~\ref{rem:motivate-twisted-towers}) that twisted Whitney disks
only occur in even order twisted Whitney towers, and only those of half-order are
relevant to the obstruction theory. 
\begin{defn}[\cite{WTCCL}]\label{def:tau-infty}
The \emph{order $n$ intersection intersection invariant}
$\tau_{n}^{\iinfty}(\cW)$ of an order
$n$ twisted Whitney tower $\cW$ is defined to be
$$
\tau_{n}^{\iinfty}(\cW):=\sum \epsilon_p\cdot t_p + \sum \omega(W_J)\cdot J^\iinfty\in\cT^{\iinfty}_{n}
$$
where the first sum is over all order $n$ intersections $p$ and the second sum is over all order $n/2$
Whitney disks $W_J$ with twisting $\omega(W_J)\in\Z$. For $n=0$, recall from \ref{subsec:order-zero-w-towers-and-ints} above our notational convention that $W_j$ denotes $A_j$; in this case $\omega(A_j)\in\Z$ is the relative Euler number of the normal bundle of $A_j$ with respect to the given framing of $\partial A_j$ as in \ref{subsec:twisted-w-disks}.

\end{defn}

By splitting the twisted Whitney disks, as explained in Section~\ref{subsec:split-w-towers} below, 
for $n>0$ we may actually assume that all non-zero $\omega(W_J)\in\{\pm 1\}$, just like the signs $\epsilon_p$.

The vanishing of $\tau_{n}^{\iinfty}$ is sufficient for the existence of a 
twisted Whitney tower of order $(n+1)$: 
\begin{thm}[\cite{WTCCL}]\label{thm:twisted-order-raising-on-A}
If a collection $A$ of properly immersed surfaces in a simply connected $4$--manifold supports an order $n$ twisted Whitney tower $\cW$ with $\tau_n^\iinfty(\cW)=0\in\cT^\iinfty_n$, then $A$ is regularly homotopic (rel $\partial$) to 
$A'$ which supports an order $n+1$ twisted Whitney tower.
\end{thm}

\subsection{Split twisted Whitney towers}\label{subsec:split-w-towers}
A twisted Whitney tower is \emph{split} if each Whitney disk is embedded, and the set of singularities in the interior
of any framed Whitney disk consists of either a single transverse intersection point, or a single boundary arc of a higher-order Whitney disk, or is empty; and if each non-trivially twisted
Whitney disk has no singularities in its interior, and has twisting equal to $\pm 1$.
This can always be arranged by performing (twisted) finger moves along Whitney disks guided by arcs
connecting the Whitney disk boundary arcs (see Section~2.5 of \cite{WTCCL}). 

Splitting simplifies the combinatorics of Whitney tower constructions and will be assumed, often without mention, in subsequent sections. Splitting an order $n$ (twisted) Whitney tower $\cW$ does not change $\tau_n^\iinfty(\cW)\in\cT^\iinfty_n$ (Lemma~2.12 of \cite{WTCCL}).

\subsection{Intersection forests for split twisted Whitney towers}\label{subsec:t(cW)}\label{subsec:int-forest}
Recall from \cite[Def.2.11]{WTCCL} that the disjoint union of signed 
trees and $\iinfty$-trees associated to the unpaired intersections and
$\pm 1$-twisted Whitney disks in a split twisted Whitney tower $\cW$ is denoted by $t(\cW)$, and called the \emph{intersection forest} of $\cW$. Here each tree $t_p$ associated to
an unpaired intersection $p$ is equipped with the sign of $p$, and
each $\iinfty$-tree $J^\iinfty$ associated to a clean $\pm 1$-twisted Whitney disk is given the corresponding 
sign $\pm 1$.

In any split $\cW$, the intersection forest can be thought of as an embedding of the disjoint union of trees $t(\cW)$ into $\cW$ which embodies both the geometric and algebraic data associated to $\cW$: If we think of the trees as subsets of $\cW$, then all singularities of $\cW$ are contained in a neighborhood of $t(\cW)$; and if we think of the trees as generators, then $t(\cW)$ is an ``abelian word'' representing
$\tau^\iinfty_n(\cW)$. 
(In any $\cW$ of order $n$, it is always possible to eliminate all intersections of order strictly greater than $n$, for instance by performing finger moves (``pushing down'') to create algebraically canceling pairs of order $n$ intersections, see discussion in Section~4 of \cite{WTCCL}). 

\begin{rem}
In the older papers \cite{CST,S1,ST2} we referred to $t(\cW)$ as the ``geometric intersection tree'' (and to the group element
$\tau_n(\cW)$ as the order $n$ intersection ``tree'', rather than ``invariant''), but the term ``forest'' better describes
the disjoint union of (signed) trees $t(\cW)$.
\end{rem}

\section{Twisted Whitney towers and gropes}\label{section:twisted-towers-and-gropes}
For use in subsequent sections, this section recalls the correspondence
between (split) Whitney towers and (dyadic) capped gropes \cite{CST,S1} in the $4$--ball, and extends this
relationship to the twisted setting. 
The main goal is to describe how this correspondence
preserves the associated disjoint unions of signed trees (intersection forests). In particular, Lemma~\ref{lem:twisted-tower-to-grope} below will be used 
in Section~\ref{sec:Milnor-thm-proof} to exhibit the relationship between twisted Whitney towers and Milnor invariants in the proof of Theorem~\ref{thm:Milnor invariant}.
A detailed understanding of the material in this section relies heavily on having digested (the proof of) Theorem~5 in \cite{S1}. 
An illustration of the tree-preserving Whitney tower-grope correspondence can be seen in  
Figure~\ref{bing-hopf-example-tower-and-grope-with-tree} below.

\subsection{Dyadic gropes and their associated trees}\label{subsec:gropes}
This subsection reviews and fixes some basic grope terminology.
It will suffice to work with \emph{dyadic} gropes, i.e.~gropes whose higher stages are all genus one; these correspond to split Whitney towers 
(Section~\ref{subsec:split-w-towers}),
and gropes in $4$--manifolds can always be modified to be dyadic by Krushkal's ``grope splitting'' operation \cite{Kr}.

A {\em dyadic grope} $G$ is a $2$-complex constructed by the following method:
\begin{enumerate}
\item Start with a compact orientable connected surface of any positive
genus, called the {\em bottom stage} of $G$, and choose a symplectic
basis of circles on this bottom stage surface.
\item Attach punctured tori to any number of the basis circles and
choose hyperbolic pairs of circles on each attached torus.
\item Iterate the second step a finite number of times, i.e.~attach punctured tori
to any number of previously chosen basis circles that don't already have a torus attached to them, and choose 
hyperbolic circle-pairs for these new tori.
\end{enumerate}
The attached tori are the {\em higher stages} of $G$, and at each iteration
in the construction tori can be attached to circles in any stage.
The basis circles in all stages of $G$ that do not have a torus
attached to them are called the {\em tips} of $G$.

Our requirement that the bottom stage of $G$ has positive genus serves only to simplify
notation and terminology, as the genus zero case will not be needed in our constructions. 

Attaching
$2$--disks along all the tips of $G$ yields a {\em capped} (dyadic) grope, denoted $G^c$,
and the uncapped grope $G$ is called the \emph{body} of $G^c$.

Cutting the bottom stage of $G$ into genus one pieces decomposes $G$ (and $G^c$) into \emph{branches},
and our notion of dyadic grope (following \cite{CST,S1}) is more precisely called
a ``grope with dyadic branches'' in \cite{Kr}.
 
With an eye towards defining intersection forests for capped gropes in $B^4$, we start by 
associating to an abstract capped grope $G^c$
the following disjoint union $t(G^c)$ of unlabeled and unoriented unitrivalent trees:
Assume first that the bottom stage of $G^c$ is a genus one surface
with boundary.  Then define $t(G^c)$ to be the unitrivalent
tree which is dual to the $2$--complex $G^c$. Specifically, the tree $t(G^c)$
can be embedded in $G^c$ in the following way. Choose a 
vertex in the interior of each surface stage and each cap of $G^c$. Then  
each edge of $t(G^c)$ is a sheet-changing path between vertices in
adjacent stages or caps (here ``adjacent'' means ``intersecting in
a circle'').
One univalent vertex of $t(G^c)$ sits in the
bottom stage of $G^c$, each of the other univalent vertices is a
point in the interior of a cap of $G^c$, and each higher stage of
$G^c$ contains a single trivalent vertex of $t(G^c)$ 
(see e.g.~the right-hand side of Figure~\ref{bing-hopf-example-tower-and-grope-with-tree} below.)

In the case where the bottom stage of $G^c$ has genus $>1$, then
$t(G^c)$ is defined by cutting the bottom stage into genus one
pieces and taking the disjoint union of the unitrivalent trees just described.
Thus, each branch of $G^c$ contains a single tree in $t(G^c)$.

Note that each tree in $t(G^c)$ has exactly one univalent vertex which sits in the bottom
stage of $G^c$; these vertices can naturally be considered as roots, and it is customary
to associate rooted trees to gropes. Here we prefer to ignore this extra information, since we will be identifying $t(G^c)$ with the unrooted trees associated to Whitney towers.

The {\em class} of a capped grope $G^c$ is one more than the minimum of the orders
of the trees in $t(G^c)$. The body $G$ of $G^c$ inherits the same union of trees,
$t(G):=t(G^c)$, and the same notion of class. 

\textbf{Convention:} \emph{For the rest of this paper gropes may be assumed to be dyadic, even if not explicitly stated.}

\subsection{Intersection forests for capped gropes bounding links}\label{subsec:gropes-bounding-links}
The \emph{boundary} $\partial G$ of a grope $G$ is the boundary of its bottom stage.
An embedding $(G,\partial G)\hookrightarrow (B^4,S^3)$ is \emph{framed}
if a disjoint parallel push-off of the bottom stage of $G$ induces a given framing of $\partial G\subset S^3$ and extends to
a disjoint parallel push-off of $G$ in $B^4$. 

\begin{defn}\label{def:link-bounds-grope}
For a framed link $L\subset S^3$, the statement ``$L$ bounds a capped grope
$G^c$ in $B^4$'' means that the link components $L_i$ bound disjointly embedded framed gropes
$G_i\subset B^4$, such that the tips of the $G_i$ bound framed caps whose interiors are disjointly
embedded, with each cap having a single transverse interior intersection with the bottom stage 
of some $G_j$. Here a cap is \emph{framed} if the parallel push-off of its boundary in the grope extends to
a disjoint parallel copy of the entire cap. The union of the gropes is denoted $G:=\cup_i G_i$, and $G^c:=\cup_i G_i^c$ is the union of $G$ together with all the caps. 
\end{defn}

All previous grope notions carry over to this setting, even though the bottom stage of $G$
is not connected; e.g.~we refer to the grope $G$ as the body of the capped grope $G^c$.
In particular, the disjoint union of trees $t(G^c):=\amalg_i\, t(G^c_i)$ can now be considered as a subset of $B^4$.
This provides labels from $\{1,2,\ldots,m\}$ for all univalent vertices: The bottom-stage univalent vertex of each tree in 
$t(G^c_i)$ inherits the label $i$; and if a cap intersects the bottom stage of $G_j$, then the vertex
corresponding to that cap inherits the label $j$, as shown in 
the right-hand side of Figure~\ref{bing-hopf-example-tower-and-grope-with-tree}. 
Orientations on all stages of $G$ induce
orientations of the trivalent vertices in $t(G^c)$, and orientations on all caps determine signs
for each cap-bottom stage intersection. To each tree in $t(G^c)$ is associated a sign $\pm$ which is the product of the signs of its caps. We assume the convention that the orientations of the bottom stages of $G$ correspond to the link orientation.
Thus, when $G^c$ is oriented, meaning that all stages and caps are oriented, $t(G^c)$ is a disjoint union of signed oriented labeled trees
which we call the \emph{intersection forest} of $G^c$, in line with the terminology for Whitney towers.

\subsection{Intersection forests for twisted capped gropes bounding links}\label{subsec:twisted-gropes-bounding-links}
A \emph{twisted capped grope} $G^c$ in $B^4$ is the same as a capped grope as in Definition~\ref{def:link-bounds-grope} just above, 
except that at most one cap in each branch of $G^c$ is allowed to be arbitrarily \emph{twisted} 
as long as its interior is embedded and disjoint from all other caps and stages of $G^c$.
Here a cap $c$ is $k$-twisted, for $k\in\Z$, if the parallel push-off of its boundary in the grope 
determines a section of the normal bundle of $c\subset B^4$ with relative Euler number $k$.
(So a $0$-twisted cap is framed.)

\begin{defn}\label{def:link-bounds-twisted-grope-and-forest}
A link $L\subset S^3$ \emph{bounds a twisted capped grope}
if the link components $L_i$ bound disjointly embedded framed gropes
$G_i\subset B^4$ which extend to a twisted capped grope $G^c=\cup_i G_i^c\subset B^4$.

The \emph{intersection forest $t(G^c)$ of a twisted capped grope bounding a link} is defined as
the extension of
the framed definition
which labels each univalent vertex that corresponds to a non-trivially $k$-twisted cap with the twist symbol $\iinfty$,
and takes the twisting $k$ as a coefficient. 
\end{defn}

Recall from~\ref{subsec:gropes} above that for a capped grope $G^c$, 
if $n$ is the minimum of the orders of the trees in
$t(G^c)$, then the class of $G^c$ is $n+1$.

Motivated by the correspondence with twisted Whitney towers described below,
we define the \emph{class of a twisted capped grope} $G^c$ to be $n+1$
if $n$ is the minimum of the orders of the non-$\iinfty$ trees in
$t(G^c)$, and if the $\iinfty$-trees in $t(G^c)$ are of order at least $n/2$.

\subsection{From twisted Whitney towers to twisted capped
gropes}\label{subsec:w-tower-to-grope} 
\begin{figure}
\centerline{\includegraphics[scale=.55]{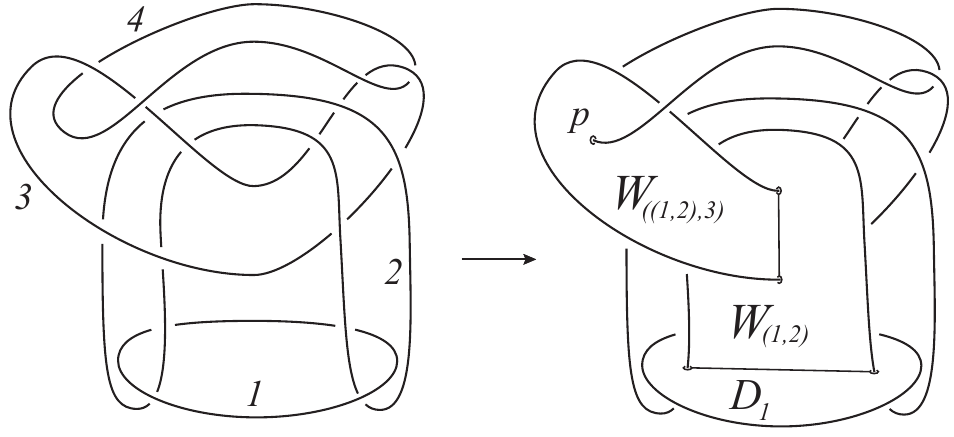}}
         \caption{Moving into $B^4$ from left to right, a Bing-doubled Hopf link $L\subset S^3$ bounds an order $2$ Whitney tower $\cW$:
         The order $0$ disk $D_1$ consists of a collar on $L_1$ together with the indicated embedded disk on the right. The other three order $0$ disks $D_2$, $D_3$ and $D_4$
         consist of collars on the other link components which extend further into $B^4$ and are capped off by disjointly embedded disks.
         The Whitney disk $W_{(1,2)}$ pairs $D_1\cap D_2$, and $W_{((1,2),3)}$
         pairs $W_{(1,2)}\cap D_3$, with $p=W_{((1,2),3)}\cap D_4$
         the only unpaired intersection point in $\cW$.}
         \label{bing-hopf-example-with-tower}
\end{figure}

\begin{figure}
\centerline{\includegraphics[scale=.65]{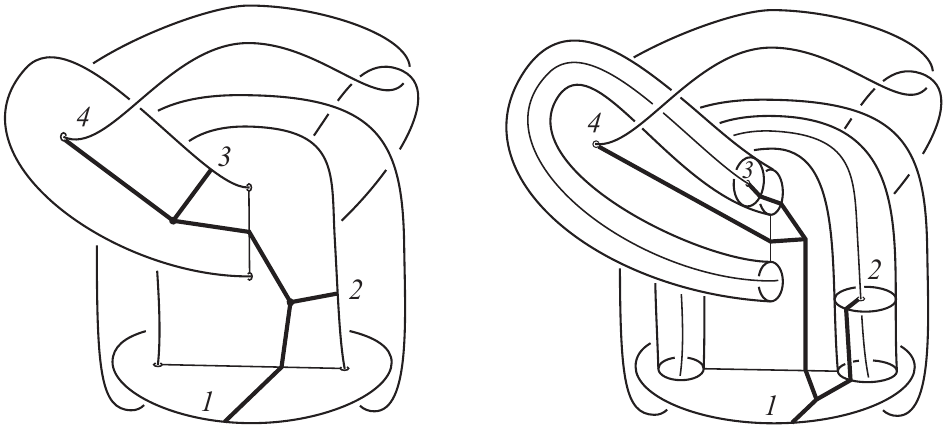}}
         \caption{Both sides of this figure correspond to the slice of $B^4$ shown in the right-hand side of Figure~\ref{bing-hopf-example-with-tower}.
         The tree $t_p=\langle ((1,2),3),4 \rangle$ is shown on the left as a subset of the order $2$ Whitney tower $\cW$. 
         Replacing this left picture by the picture on the right illustrates the tree-preserving construction
         of a class $3$ capped grope $G^c$ bounded by $L$. In this case, the component $G^c_1$ bounded by $L_1$ is the class $3$ capped grope
         shown (partly translucent) on the right (together with a collar on $L_1$) which is gotten by surgering $D_1$ and $W_{(1,2)}$. 
         The three other components of $G$ are just the disks $D_2$, $D_3$ and $D_4$ of $\cW$, each of which has a single intersection with a cap of $G^c_1$.}
         \label{bing-hopf-example-tower-and-grope-with-tree}
\end{figure}

In \cite{S1} a
``tree-preserving'' procedure for converting an order $n$ (framed) Whitney
tower $\cW$ into a class $n+1$ capped grope (and vice versa) is
described in detail. This construction will be extended to the twisted setting in Lemma~\ref{lem:twisted-tower-to-grope} just below, which will be used in the proof of Theorem~\ref{thm:Milnor invariant} given in the next section. The rough idea is that the ``subtower'' of Whitney disks containing
a tree in a split Whitney tower can be surgered to a dyadic branch of a capped grope containing the same tree, with the capped grope orientation inherited from that of the Whitney tower.
Orientation and sign conventions will be presented during the course of the proof.

\begin{lem}\label{lem:twisted-tower-to-grope} 
If $L$ bounds an order $n$ split twisted Whitney tower $\cW$, then 
$L$ bounds a dyadic class $n+1$ twisted capped grope $G^c$ such that:
\begin{enumerate}
\item $t(\cW)$ is isomorphic to $t(G^c)$.
\item Each framed cap of $G^c$ has intersection $+1$ with a bottom
  stage of $G$, except that one framed cap in each dyadic branch of $G^c$ with
  signed tree $\epsilon_p\cdot t_p$ has intersection $\epsilon_p$
  with a bottom stage.
  \item Each $\epsilon$-twisted cap of $G^c$ contains the corresponding $\iinfty$-labeled vertex of its $\iinfty$-tree in $t(G^c)$. 
\end{enumerate}
\end{lem}

\begin{proof}

\textbf{Outline:} A detailed inductive proof of the framed unoriented case is given in \cite[Thm.5]{S1}.
We will adapt the proof from \cite{S1} to the current twisted setting, sketching
the construction while introducing orientation and sign conventions. 
The basic idea of the procedure is to tube (0-surger)
along one boundary-arc of each Whitney disk; but in order to
maximize the class of the resulting grope, Whitney moves may need
to be performed when trees are not \emph{simple}, meaning right-
or left-normed (see Figure~17 of Section~7 in \cite{S1}). As mentioned at the start of this section,
an appreciation of the role played by these subtleties in the current proof depends largely on
having understood the proof of Theorem~5 of \cite{S1}.

A simple example of the construction (in the framed case) is illustrated in Figures~\ref{bing-hopf-example-with-tower} and \ref{bing-hopf-example-tower-and-grope-with-tree}, 
which show how an order $2$ Whitney tower
bounded by the Bing-double of the Hopf link can be converted to a class $3$ capped grope.

For each $t_p\in t(\cW)$, the construction works upward from a
chosen Whitney disk having a boundary arc on an order zero disk $D_i$, which corresponds to
the choice of
an $i$-labeled univalent vertex of $t_p$, creating caps out of
Whitney disks, then turning these caps into surface stages whose
caps are created from higher-order Whitney disks, and so on, until the process terminates
when each framed cap has a single intersection with a bottom stage surface. The
resulting dyadic branch of $G^c$ will inherit the tree $t_p$ as a subset.
Similarly, for each $\iinfty$-tree
$\pm J^\iinfty\in t(\cW)$ the construction will yield a dyadic branch
containing $J^\iinfty$ with the $\iinfty$-vertex sitting in a
$\pm 1$-twisted cap.

\textbf{The surgery step:} Figure~\ref{surgery-stepAandB-fig} illustrates a surgery step and
the corresponding modification of the embedded tree
near a trivalent vertex corresponding to a Whitney disk $W_{(I,J)}$ in $\cW$.
The sheet $c_I$ is a (temporary) cap which has already been created, or is just an order zero disk $D_i$ with $I=i$ 
in the first step of creating a dyadic branch of $G^c$.
Any interior intersections of $W_{(I,J)}$ are not shown.
After the surgery which turns the $c_I$ into a surface stage
$S_I$, the Whitney disk $W_{(I,J)}$ minus part of a collar becomes
one cap $c_{(I,J)}$, and a normal disk to the $J$-sheet becomes a dual
cap $c_J$.  
\begin{figure}
\centerline{\includegraphics[width=120mm]{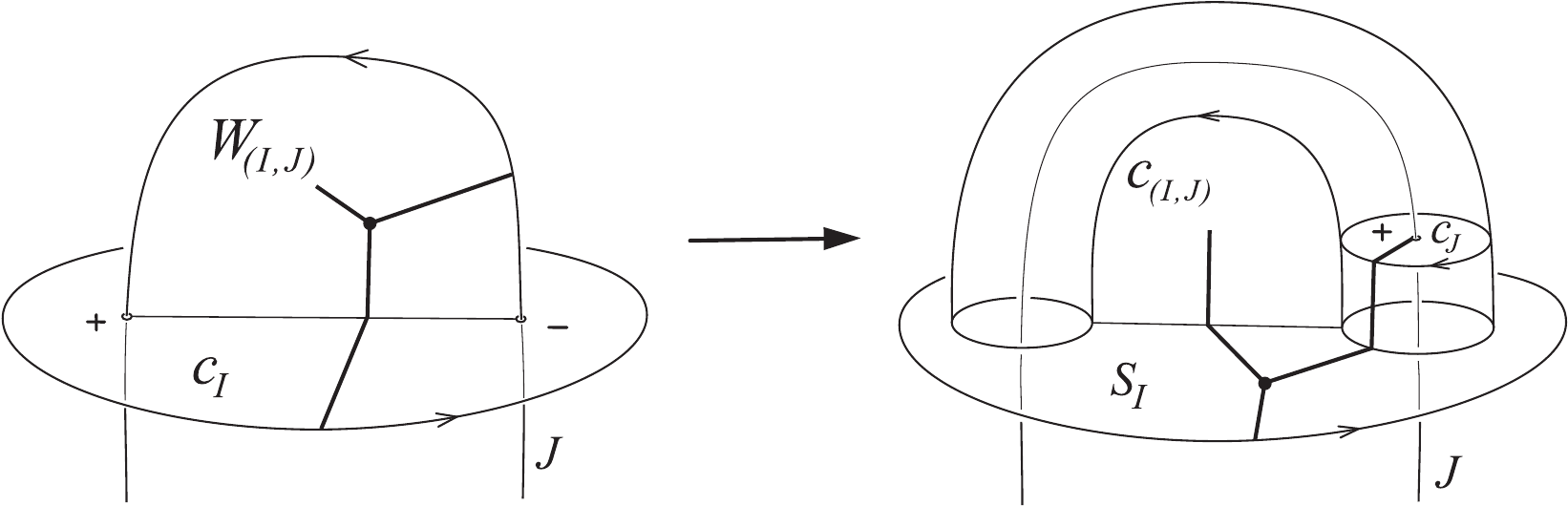}}
         \caption{A surgery step in the resolution of an order $n$ twisted Whitney tower to
         a class $n+1$ twisted capped grope. Any interior intersections in $W_{(I,J)}$ and $c_{(I,J)}$ are not shown.}
         \label{surgery-stepAandB-fig}
\end{figure}
The $S_I$ stage inherits the orientation of $c_I$, and the cap
$c_{(I,J)}$ inherits the orientation of $W_{(I,J)}$. As pictured in Figure~\ref{surgery-stepAandB-fig}, 
the effect of the surgery on the tree sends the trivalent
vertex in $W_{(I,J)}$ to the trivalent vertex in the $S_I$ sheet,
with the induced orientation. The cap $c_J$ is a parallel copy of what
used to be a neighborhood in $c_I$ around the negative
intersection point paired by $W_{(I,J)}$, but with the opposite
orientation, so that $c_J$ has a single positive intersection with
the $J$-sheet.

Here the $I$-subtree sits in the part of the grope branch which has already been constructed, 
while the $J$-subtree and any $K$-subtree corresponding to intersections with $c_{(I,J)}$ sit in
sub-towers of $\cW$ which have yet to be converted to grope stages.
The proof proceeds by considering the various cases depending on the orders of $I$, $J$, and $K$.

We call a Whitney disk or cap \emph{clean} if it is embedded and free of any interior intersections with any surface sheets.

\textbf{The surgery cases:}
If $W_{(I,J)}$ had a single interior intersection with an order zero disk $D_k$,
then so does the cap $c_{(I,J)}$; and we relabel this cap as $c_k$. If in this case
$J=j$ is also order zero, then there is no
further modification to $c_j$ and $c_k$, which remain as caps intersecting the bottom stages
of the gropes $G_j$ and $G_k$ when the construction is complete.

If $W_{(I,J)}$ was a clean $\epsilon$-twisted Whitney disk, 
then $c_{(I,J)}$ is a clean $\epsilon$-twisted cap of $G^c$ containing the $\iinfty$-label of the $\iinfty$-tree associated to the branch. In this
case there is no further modification of the cap, which will be
denoted $c^\iinfty_{(I,J)}$.

Note that surgering Whitney disks to caps preserves twistings: See
Figure~\ref{twistings-preserved-fig}.
\begin{figure}
\centerline{\includegraphics[width=110mm]{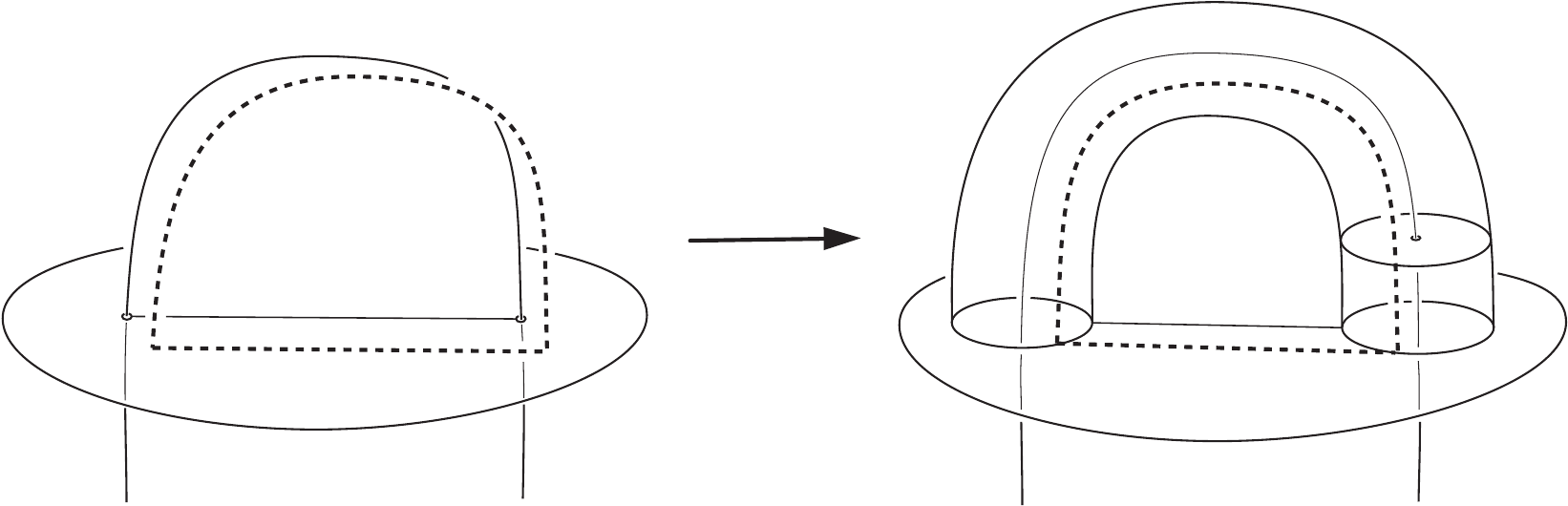}}
         \caption{The framing obstruction determined by
         the Whitney section over the boundary of a Whitney disk 
         is passed on to the framing obstruction on the cap resulting from surgery.}
         \label{twistings-preserved-fig}
\end{figure}

If $J=j$ is order zero and $W_{(I,j)}$ was a clean $\epsilon$-twisted
Whitney disk yielding $c^\iinfty_{(I,j)}$, then there is also no further
modification to the dual cap $c_j$. If $J=(J_1,J_2)$ has positive order, then 
the clean twisted cap $c^\iinfty_{(I,J)}$ remains ``as is'', but the dual cap $c_J$ is modified as described below
(in the next-to-last paragraph of the proof).

If the cap $c_{(I,J)}$ contains an arc of a Whitney disk boundary, then
the just-described surgery step for $c_I$ applies to $c_{(I,J)}$. Otherwise, the
grope construction requires a Whitney move as described next.

If the cap $c_{(I,J)}$ intersects some $W_K$ transversely in
the single point $p$, with $\sign (p)=\epsilon_p$, and
$K=(K_1,K_2)$ of positive order, then the grope construction
proceeds by doing a Whitney move guided by $W_K$ on either the $K_1$-sheet
or the $K_2$-sheet: The
effect of this $W_K$-Whitney move is to replace $p$ by a Whitney-disk boundary-arc in 
$c_{(I,J)}$ so that the surgery step can be
applied. Here $p$ could be the original unpaired intersection in
$t_p$, or an intersection created during the construction, and
Figure~\ref{fig:W-MOVE-trees} illustrates how the oriented
tree and the sign of the unpaired intersection are preserved in
the case $\epsilon_p=+1$; the case $\epsilon_p=-1$ can be checked
in the same way.
\begin{figure}
\centerline{\includegraphics[width=110mm]{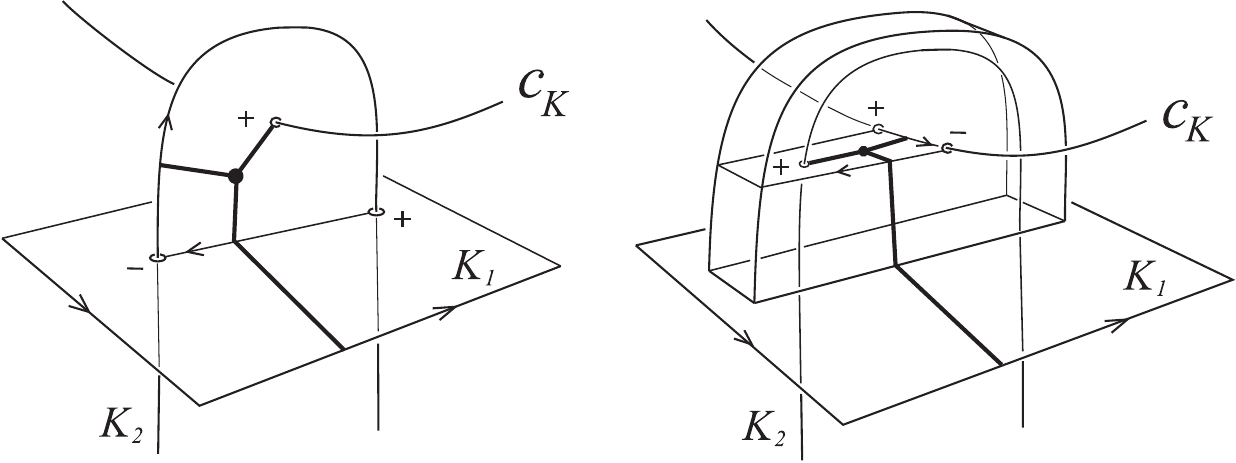}}
         \caption{A Whitney move preserves the sign and orientation at a trivalent vertex.}
         \label{fig:W-MOVE-trees}
\end{figure}

Similarly, if $J=(J_1,J_2)$ has positive order, then the grope construction proceeds by doing a
$W_J$-Whitney move to replace the positive intersection point
between $c_J$ and $W_J$ by a boundary arc of a Whitney disk, so
that the surgery step can be applied to $c_J$. That this preserves
the oriented tree and the $+1$ sign of the unpaired intersection
also follows from (a re-labeling of) Figure~\ref{fig:W-MOVE-trees}.

For each tree in $t(\cW)$ this procedure terminates when each
framed cap has a single intersection with a bottom stage, creating
a dyadic branch of the capped grope $G^c$; and applying the procedure to all trees in
$t(\cW)$ yields $G^c$, containing its intersection forest
$t(G^c)$, with all vertex orientations induced by the
orientation of $G^c$. Since conditions (ii) and (iii) of the lemma are satisfied, it follows that
$t(G^c)$ and $t(\cW)$ are isomorphic, since the coefficients of the trees are also preserved.
\end{proof}

\section{Proof of
Theorem~\ref{thm:Milnor invariant}}\label{sec:Milnor-thm-proof}
Recall the content of Theorem~\ref{thm:Milnor invariant}: For $L$ bounding an order~$n$ twisted Whitney tower $\cW$, the first non-vanishing order $n$ Milnor invariant $\mu_n(L)$ can be computed from $\cW$ as
$$
\mu_n(L)=\eta_n(\tau_n^\iinfty(\cW))
$$
where $\mu_n(L):=\sum_i X_i \otimes \mu_n^i(L) \in \sL_1 \otimes\sL_{n+1}$
collects the length $n+1$ iterated commutators determined by the link longitudes considered
as Lie brackets $\mu_n^i(L)$ in the free $\Z$-Lie algebra, and the map $\eta_n$ converts trees into rooted trees (Lie brackets) by summing over all choices of roots (the definition of $\eta_n$ is recalled below). The proof of this statement will also show that 
$\mu_k (L)$ vanishes for all $k<n$.

We first note that Theorem~\ref{thm:Milnor invariant} holds in two special cases:

\textsc{The order 0 case:}
It is easily checked from the definitions in Section~\ref{subsec:intro-Milnor-review} that for $i\neq j$ the coefficient of 
$X_i\otimes X_j$ in $\mu_0(L)$ is the linking number of $L_i$ and $L_j$, which via the well-known computation of linking numbers by counting signed intersections
between the properly immersed disks $D_i$ and $D_j$ bounded by $L_i$ and $L_j$ is also equal to the coefficient 
of $X_i\otimes X_j$ in $\eta_0(\tau_0^\iinfty(\cW))$.

Although Milnor invariants are not usually defined for knots, for framed links it is natural to consider the 
framing $f_i$ of $L_i$ as an order $0$ (length $2$) integer Milnor invariant, and the coefficient of $X_i\otimes X_i$ in $\mu_0(L)$ is exactly $f_i$ when this framing is used to determine the $i$th longitude. To see that the coefficient 
in $\eta_0(\tau_0^\iinfty(\cW))$ of
$X_i\otimes X_i$ is also equal to $f_i$, let $d_i$ denote the number of positive self
intersections of $D_i$ minus the number of negative self intersections  of $D_i$. Then
the relative Euler number of $D_i$ with respect to the framing $f_i$ on $L_i=\partial D_i$ is equal to $f_i-2d_i$
(see e.g.~Figure~19 of \cite{WTCCL} and accompanying discussion), and the terms of 
$\tau_0^\iinfty(\cW)$ which contribute via $\eta_0$ to the coefficient of $X_i\otimes X_i$
are exactly 
$
(d_i)\cdot \,i\,-\!\!\!-\!\!\!-\,i+(f_i-2d_i) \cdot\iinfty \,-\!\!\!-\!\!\!-\,i,
$
which get sent by $\eta_0$ to
$(f_i)\cdot X_i\otimes X_i$.

So we may assume in the rest of the proof  
that $\cW$ is of positive order.
Note that this means that all link component framings are zero,
since the self-intersections
of all order zero disks come in canceling pairs (paired by order $1$ Whitney disks).

\textsc{The case of slice links:} In the case that $\cW$ consists of disjointly embedded slice disks for $L$, then
Theorem~\ref{thm:Milnor invariant} is true for all $n>0$ since
$\mu_n(L)$ vanishes, and so does $\tau_n^\iinfty(\cW)$ since $\cW$ is an order $n$ twisted Whitney tower for all $n$.

So in the rest of the proof we may also assume that the intersection forest $t(\cW)$ (defined in Section~\ref{subsec:int-forest}) is non-empty.

\subsection{Outline of the proof}\label{subsec:Milnor-thm-proof-outline}
To prove Theorem~\ref{thm:Milnor invariant} we will first convert the order $n$
twisted Whitney tower $\cW$ bounded by $L$ to an  order $n+1$ twisted capped
grope $G^c$, as in Lemma~\ref{lem:twisted-tower-to-grope}. 
It will follow from an extension of \emph{grope duality} \cite{KT} to the setting of twisted capped gropes, together with Dwyer's theorem \cite{Dw}, that we can compute the link
longitudes in $\pi_1(B^4\setminus G^c)$ instead of $\pi_1(S^3\setminus L)$. Via the capped grope duality construction the
iterated commutators determined by the longitudes will be seen to correspond exactly to the image 
of $\tau_n^\iinfty(\cW)$ under the map $\eta_n$.
To preview the computation of the longitudes the reader can examine Figures~\ref{bing-hopf-example-with-tower} and \ref{bing-hopf-example-tower-and-grope-with-tree} which show
the Whitney tower-to-capped grope conversion for $L$ the Bing-double of the Hopf link.
It should be clear from the right-hand side of Figure~\ref{bing-hopf-example-tower-and-grope-with-tree} that the longitude for component $L_1$ is a triple commutator $[x_2,[x_3,x_4]]$ of meridians to the other components, as exhibited by the class $3$ capped grope $G_1^c$ bounded by $L_1$ and containing the order $2$ tree. As a consequence of grope duality, it will turn out that a the other longitudes
also bound class $3$ gropes which correspond to choosing roots at the other univalent vertices on the same order $2$ tree
(although these gropes are not so visible in the figure).

For the reader's convenience we recall from the introduction the definition of the map $\eta_n:\cT^\iinfty_n\to \sD_n$,
where $\sD_n=\sD_n(m)$ is the kernel
of the bracket map $\sL_1 \otimes \sL_{n+1}\rightarrow \sL_{n+2}$.

For $v$ a univalent vertex of an order $n$ tree
$t$, denote by $B_v(t)\in\sL_{n+1}$ the
Lie bracket of generators $X_1,X_2,\ldots,X_m$ determined by the formal bracketing from $\{1,2,\ldots,m\}$
which is gotten by considering $v$ to be a root of $t$.

Denoting the label of a univalent vertex $v$ by $\ell(v)\in\{1,2,\ldots,m\}$, the
map $\eta_n:\cT^\iinfty_n\rightarrow \sL_1 \otimes \sL_{n+1}$
is defined on generators by
$$
\eta_n(t):=\sum_{v\in t} X_{\ell(v)}\otimes B_v(t)
\quad \, \,
\mbox{and}
\quad \, \,
\eta_n(J^\iinfty):=\frac{1}{2}\eta_n(\langle J,J \rangle)
$$
where the first sum is over all univalent vertices $v$ of $t$, and the second expression is indeed in $\sL_1 \otimes\sL_{n+1}$ since the coefficient of $\eta_n(\langle J,J \rangle)$ is even.

The following lemma is proved in Section~\ref{subsec:lemma-eta-well-defined-proof}:
\begin{lem}\label{lem:eta-well-defined}
The homomorphism $\eta_n:\cT^\iinfty_n\rightarrow \sD_n$ is a well-defined surjection.
\end{lem}

We will also use:
\begin{lem}\label{lem:grope-duality}
If $L\subset S^3$ bounds a class $(n+1)$ twisted capped grope
$G^c\subset B^4$, then the inclusion $S^3\setminus
L\hookrightarrow B^4\setminus G^c$ induces an isomorphism
$$
\frac{\pi_1(S^3\setminus L)}{\pi_1(S^3\setminus L)_{n+2}} \cong \frac{\pi_1(B^4\setminus
G^c)}{\pi_1(B^4\setminus G^c)_{n+2}}.
$$
\end{lem}
The proof of Lemma~\ref{lem:grope-duality}
is given below in
Section~\ref{subsec:grope-duality-lemma-proof}.

\subsection{Computing the longitudes}\label{subsec:computing-the-longitudes} 
By
Lemma~\ref{lem:grope-duality} we can compute the iterated commutators determined by the link longitudes
in $\pi_1(B^4\setminus G^c)$ modulo $\pi_1(B^4\setminus G^c)_{n+2}$. The computation will show that the
longitudes lie in $\pi_1(B^4\setminus G^c)_{n+1}$, which implies that 
$\mu_k (L)$ vanishes for all $k<n$. 

\textsc{Terminology note:} Throughout this proof we will use the word \emph{meridian} to refer to
fundamental group elements represented by normal circles to deleted surfaces in $4$--space; and on occasion such
circles will themselves be referred to as ``meridians''.

\textsc{Conventions:}
Via the
isomorphisms of Lemma~\ref{lem:grope-duality} and
Section~\ref{subsec:intro-Milnor-review} we make the identifications
$$
\frac{\pi_1(B^4\setminus G^c)_{n+1}}{\pi_1(B^4\setminus
G^c)_{n+2}}\cong\frac{\pi_1(S^3\setminus L)_{n+1}}{\pi_1(S^3\setminus L)_{n+2}}\cong\frac{F_{n+1}}{F_{n+2}}
$$
where the generators $\{x_1,x_2,\ldots,x_m\}$ are meridians to the
bottom stages of $G^c$, with $x_i$ chosen to have linking number
$+1$ with the bottom stage of the grope component $G_i$ which is bounded
by $L_i$.

Orientations of surface sheets and their boundary circles are related by the usual ``outward
vector first'' convention.

We use the commutator notation $[g,h]:=ghg^{-1}h^{-1}$, and exponential notation $g^h:=hgh^{-1}$
for group
elements $g$ and $h$.

Since an element in $\frac{F_{n+1}}{F_{n+2}}$ determined by an
$(n+1)$-fold commutator of elements of $\frac{F}{F_{n+2}}$ only depends on the conjugacy classes of
the elements, we can and will suppress basings of meridians from
computations. This follows easily from the commutator relation $[xy,z]=[y,z]^x[x,z]$ which holds in any group.
The following relations in $\frac{F_{n+1}}{F_{n+2}}$ will be useful:

For any length $n+1$ commutator $[x_I,x_J]$, and $\epsilon=\pm 1$,
\begin{equation}
[x_I,x^\epsilon_J]=[x^\epsilon_I,x_J]=[x^{-\epsilon}_J,x_I]=[x^\epsilon_J,x_I]^{-1}=[x^\epsilon_J,x^{-1}_I].
\label{eqno1}
\end{equation}

\begin{figure}[h]
\centerline{\includegraphics[width=110mm]{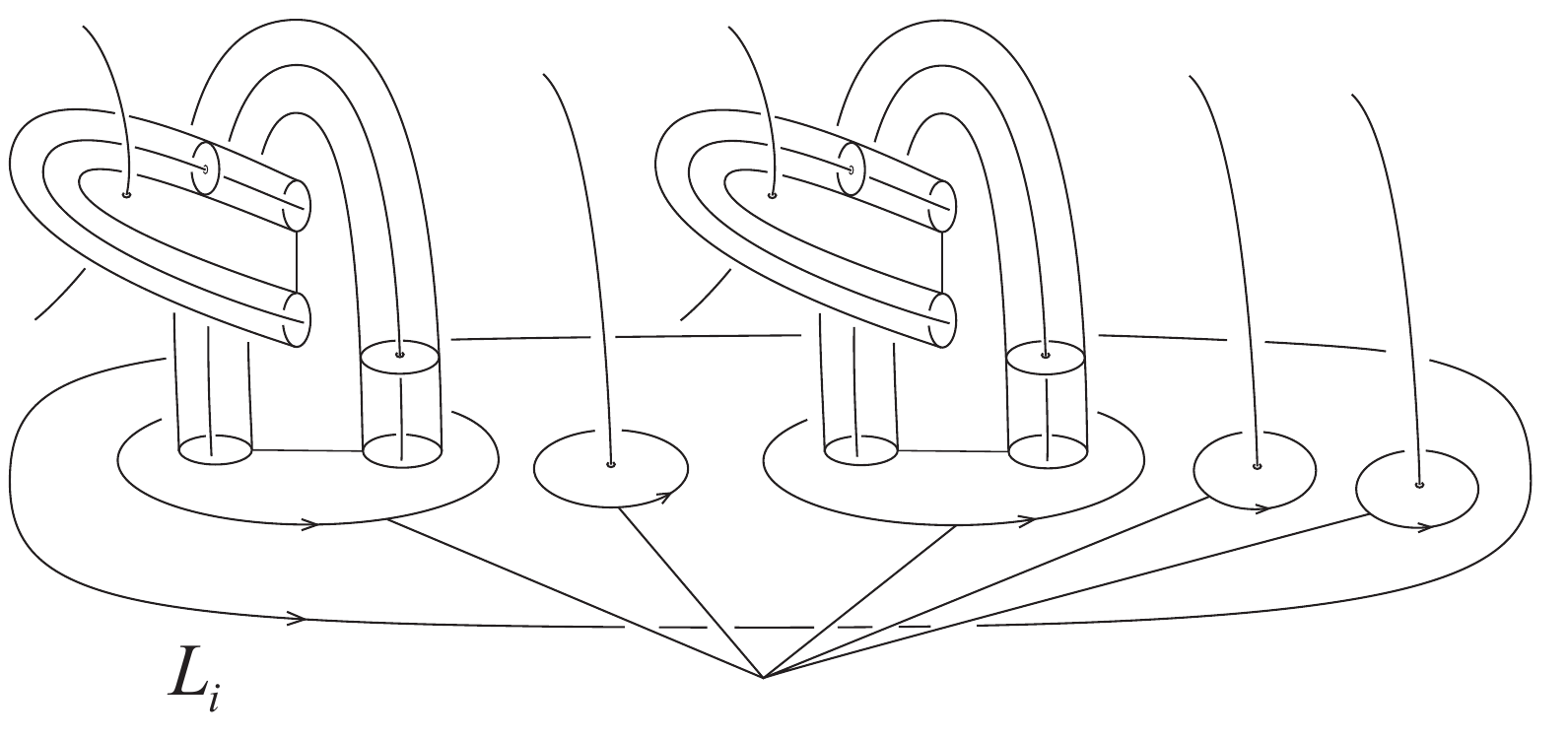}}
         \caption{A parallel push-off of $L_i$ is isotopic to a product of loops which are boundary circles of
         parallel push-offs of dyadic branches of $G_i$, or meridional circles to framed caps of $G^c$. So the corresponding factors
         $\gamma_{i_r}$ of the $i$th longitude $\gamma_i=\prod_r \gamma_{i_r}$ are in one-to-one correspondence with the $i$-labeled
         vertices of the trees in $t(G^c)$.}
         \label{longitude-Li-fig}
\end{figure}

For $L$ bounding $\cW$ of positive order, the longitudes $\gamma_i$ are represented by $0$-parallel
push-offs of the link components. As illustrated in
Figure~\ref{longitude-Li-fig}, each
longitude factors as $\gamma_i=\prod_r \gamma_{i_r}$, with each
$\gamma_{i_r}$ represented by either a parallel push-off of the
boundary of a dyadic branch of $G_i$, or a meridian to a
framed cap in $G^c$. (The ordering of the factors of $\gamma_i$ is
irrelevant since $\frac{F_{n+1}}{F_{n+2}}$ is abelian.)

For each $i$, the factors $\gamma_{i_r}$ are in
one to one correspondence with the set of $i$-labeled vertices $v_{i_r}$ on all the trees in
$t(G^c)$ (since each $i$-labeled univalent vertex on a tree corresponds either to an intersection between a framed cap and
the bottom stage of $G_i$, or to the $i$-labeled vertex sitting in the bottom stage of a dyadic branch of $G_i$). To finish the proof of Theorem~\ref{thm:Milnor invariant} it
suffices to check that each $\gamma_{i_r}$ is equal to the iterated
commutator $\beta_{v_{i_r}}(t)^\epsilon\in\frac{F_{n+1}}{F_{n+2}}$ determined by
putting a root at $v_{i_r}$ on the tree $\epsilon\cdot t\in t(G^c)$
containing $v_{i_r}$,
or equal to $\beta_{v_{i_r}}(\langle J,J \rangle)^\epsilon$
for $v_{i_r}$ in an $\iinfty$-tree $\epsilon\cdot  J^\iinfty\in t(G^c)$. 
Here the correspondence between rooted trees and iterated commutators is directly analogous to the correspondence with Lie brackets, and   
the isomorphism
 $\frac{F_{n+1}}{F_{n+2}}\cong \sL_{n+1}$ in the definition (\ref{subsec:intro-Milnor-review}) of $\mu_n(L)$ maps 
 $\beta_{v_{i_r}}(t)^\epsilon$  to the Lie bracket $\epsilon \cdot B_{v_{i_r}}(t)\in\sL_{n+1}$
as in the definition 
of the map $\eta_n$.
Similarly for $\iinfty$-trees $J^\iinfty$, $\beta_{v_{i_r}}(\langle J,J \rangle)^\epsilon$ maps to the correct 
Lie bracket $\epsilon \cdot B_{v_{i_r}}(\langle J,J \rangle)\in\sL_{n+1}$ if $n$ is even.
Then $\mu_n^i(L)$ is the sum of these Lie brackets over all the $v_{i_r}$.

\subsubsection{The order 1 case:}\label{subsec:main-proof-order-1}
As a warm-up and base case for the general proof we check that $\eta_1$ takes
$\tau^\iinfty_1(\cW)$ to $\mu_1(L)$ (the ``triple linking numbers'' of $L$) for any order $1$ twisted Whitney tower $\cW$ bounded by $L$. In this case the grope
construction yields a class $2$ twisted capped grope $G^c$
bounded by $L$, with intersection forest $t(G^c)$ a
disjoint union of signed order $1$ $\sY$-trees representing
$\tau^\iinfty_1(\cW)$.  The body $G$ is just a collection of disjointly embedded surfaces, and there are no twisted caps (since 
odd-order twisted Whitney towers do not contain twisted Whitney disks).

First consider the case where
$t(G^c)=\epsilon_p\cdot t_p=\epsilon_p\cdot \langle (i,j),k\rangle$ is a single
$\sY$-tree, with $i$, $j$ and $k$ distinct, and $G$ consists of a single genus one surface $G_i$ bounded by $L_i$, together with disjointly embedded disks $G_j$ and $G_k$ bounded by the link components $L_j$ and $L_k$ (Figure~\ref{order-one-longitude-A1-fig}).
\begin{figure}
\centerline{\includegraphics[width=60mm]{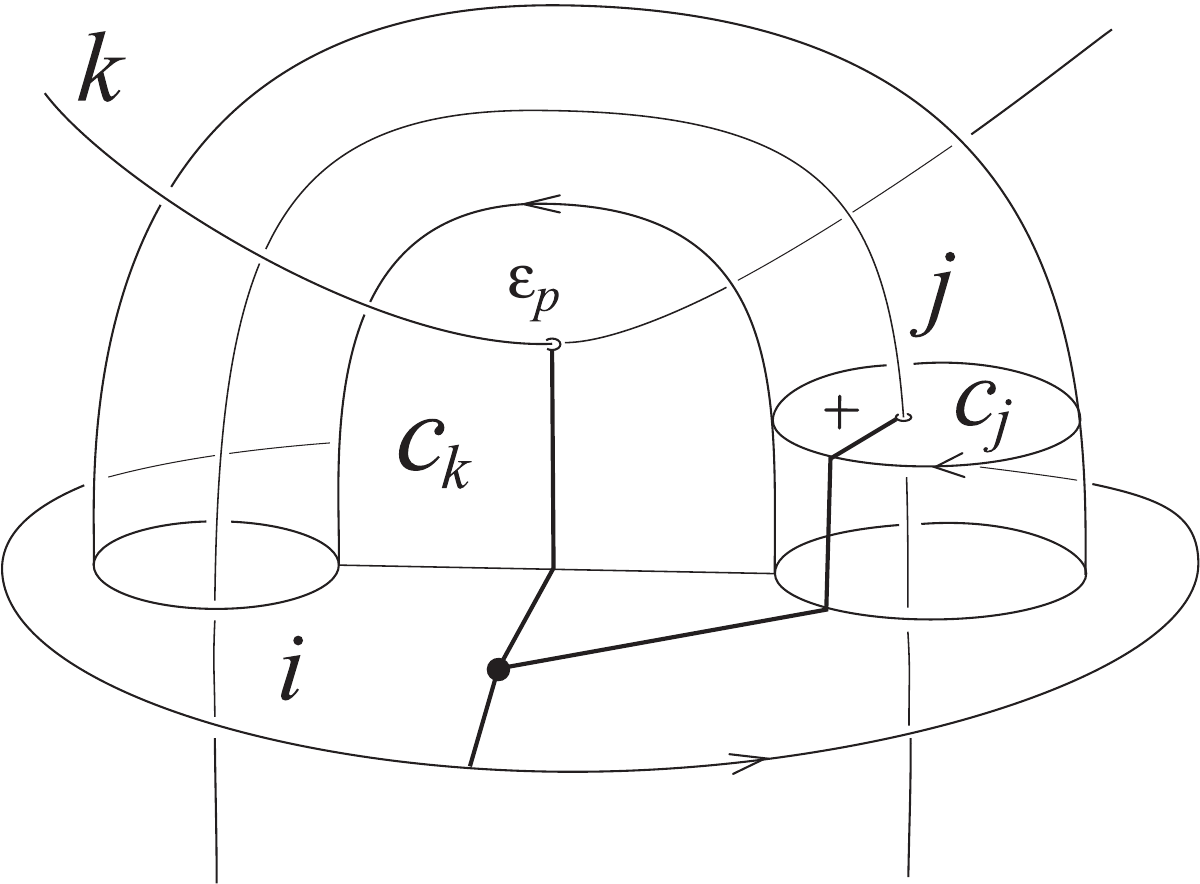}}
         \caption{Near the trivalent vertex of the signed $\sY$-tree $\epsilon_p\cdot t_p=\epsilon_p\cdot\langle (i,j),k\rangle$ in a dyadic class $2$ capped grope component (capped surface) bounded by $L_i$.}
         \label{order-one-longitude-A1-fig}
\end{figure}

We want to check that: 
$$
\mu_1(L)=\eta_1(\epsilon_p\cdot \langle (i,j),k\rangle)= 
\epsilon_p\cdot X_i\otimes\,-\!\!\!\!\!-\!\!\!<^{\,\,k}_{\,\,j}\,+ \,\, \epsilon_p\cdot X_j\otimes \,{\scriptstyle i}-\!\!\!\!\!-\!\!\!<^{\,\,k}_{\,\,}\,  + \,\, \epsilon_p\cdot X_k\otimes \, {\scriptstyle i}\,-\!\!\!\!\!-\!\!\!<^{\,\,}_{\,\,j}
$$

A parallel push-off of $L_i$ bounds a parallel push-off of $G_i$ in $B^4\setminus G^c$ and 
the longitude $\gamma_i$ can be computed from
Figure~\ref{order-one-longitude-A1-fig} (using the commutator relations
(\ref{eqno1}) above):
$$
\gamma_i=[x^{-1}_j,x^{-\epsilon_p}_k]=[x_j,x_k]^{\epsilon_p}
$$
This is the correct commutator $\beta_{v_i}( t_p)^{\epsilon_p}\in\frac{F_2}{F_3}$ 
corresponding to choosing a root for $t_p$ at the
$i$-labeled vertex $v_i$, confirming the first term in the right-hand side of the above expression for $\mu_1(L)$:
$$ 
X_i\otimes \,\epsilon_p\cdot B_{v_i}(t_p)=X_i\otimes\,\epsilon_p\cdot[X_j,X_k]=X_i\otimes\mu_1^i(L)
$$

\begin{figure}
\centerline{\includegraphics[width=60mm]{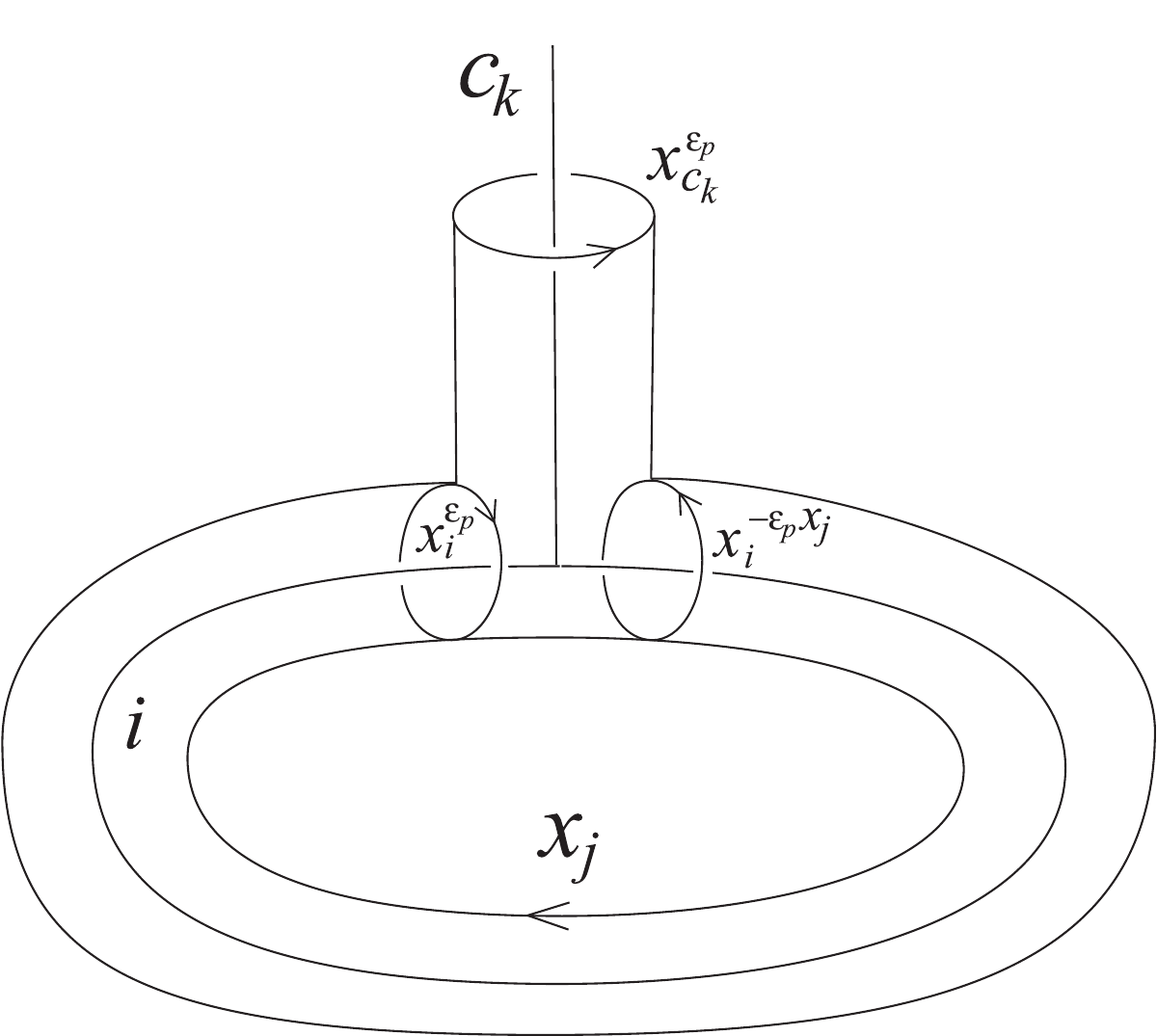}}
         \caption{A meridian to the cap $c_k$ in Figure~\ref{order-one-longitude-A1-fig} bounds a genus one surface which is a punctured normal torus to the surface stage containing the cap boundary. This normal torus consists of circle fibers in the normal circle bundle over a dual circle to the cap boundary in the surface stage. This dual circle is parallel to the boundary of the dual cap (which in Figure~\ref{order-one-longitude-A1-fig} represents the meridian $x_j$). Since the (closed) normal torus has a single intersection with the cap it is also called a ``dual torus'' for the cap.}
         \label{normal-torus-fig}
\end{figure}
A parallel push-off of $L_k$ bounds a parallel push-off of the embedded disk $G_k$ in $B^4\setminus G$,
with $G_k$ intersecting $G^c$ in the single point $p\in c_k$ with sign $\epsilon_p$.  Thus, the longitude $\gamma_k$
is equal to $x_{c_k}^{\epsilon_p}\in\frac{F}{F_{n+2}}$, the positive meridian $x_{c_k}$ to the
cap $c_k$ raised to the power $\epsilon_p$. This meridian can be expressed
in terms of the generators using the ``dual torus'' to $c_k$
illustrated in Figure~\ref{normal-torus-fig}, giving:
$$
\gamma_k=x_{c_k}^{\epsilon_p}=x^{\epsilon_p}_ix_i^{-\epsilon_px_j}=
x^{\epsilon_p}_ix_jx_i^{-\epsilon_p}x_j^{-1}=[x^{\epsilon_p}_i,
x_j]=[x_i,x_j]^{\epsilon_p}
$$
which is the correct commutator $\beta_{v_k}( t_p)^{\epsilon_p}$ 
when the root of $t_p$ is at the
$k$-labeled vertex $v_k$. (One way to check this expression for
$x_{c_k}^{\epsilon_p}$ directly from
Figure~\ref{order-one-longitude-A1-fig} is to push the $k$-sheet
down off $c_k$ into the $i$-sheet by a finger move (the vertical
tube in Figure~\ref{normal-torus-fig}) to get a cancelling pair of
intersection points which correspond to the factors
$x^{\epsilon_p}_i$ and $x_i^{-\epsilon_px_j}$.)
This confirms the third term in the right-hand side of the expression for $\mu_1(L)$:
$$ 
X_k\otimes \epsilon_p\cdot B_{v_k}(t_p)=X_k\otimes\epsilon_p\cdot[X_i,X_j]=X_k\otimes\mu_1^k(L)
$$

By similarly using a dual torus, one can also check that the contribution
to $\gamma_j$ coming from the intersection point in the cap $c_j$
is equal to $\beta_{v_j}(t_p)^{\epsilon_p}=[x_k,x_i]^{\epsilon_p}$,
confirming the second term in the right-hand side of the expression for $\mu_1(L)$:
$$ 
X_j\otimes \epsilon_p\cdot B_{v_j}(t_p)=X_j\otimes\epsilon_p\cdot[X_k,X_i]=X_j\otimes\mu_1^j(L)
$$

Since all other link components bound disjointly embedded disks,
this confirms Theorem~\ref{thm:Milnor invariant} in this case where 
$t(\cW)=t(G^c)=\epsilon_p\cdot t_p=\epsilon_p\cdot \langle (i,j),k\rangle$
with $i$, $j$ and $k$ distinct. If $i$, $j$ and $k$ are not distinct, then $t_p=0\in\cT^\iinfty_1$ by the boundary-twist relations
$\langle (i,j),j\rangle=0$ (and the just-described computation will show that $t_p$ contributes trivially to $\mu_1(L)$, since $[x_j,x_j]=0$ and $[x_j,x_i]+[x_i,x_j]=0$).
The general
order $1$ case follows by summing the above computation over all factors of each longitude.

\subsubsection{The higher-order framed case:}
\begin{figure}
\centerline{\includegraphics[width=65mm]{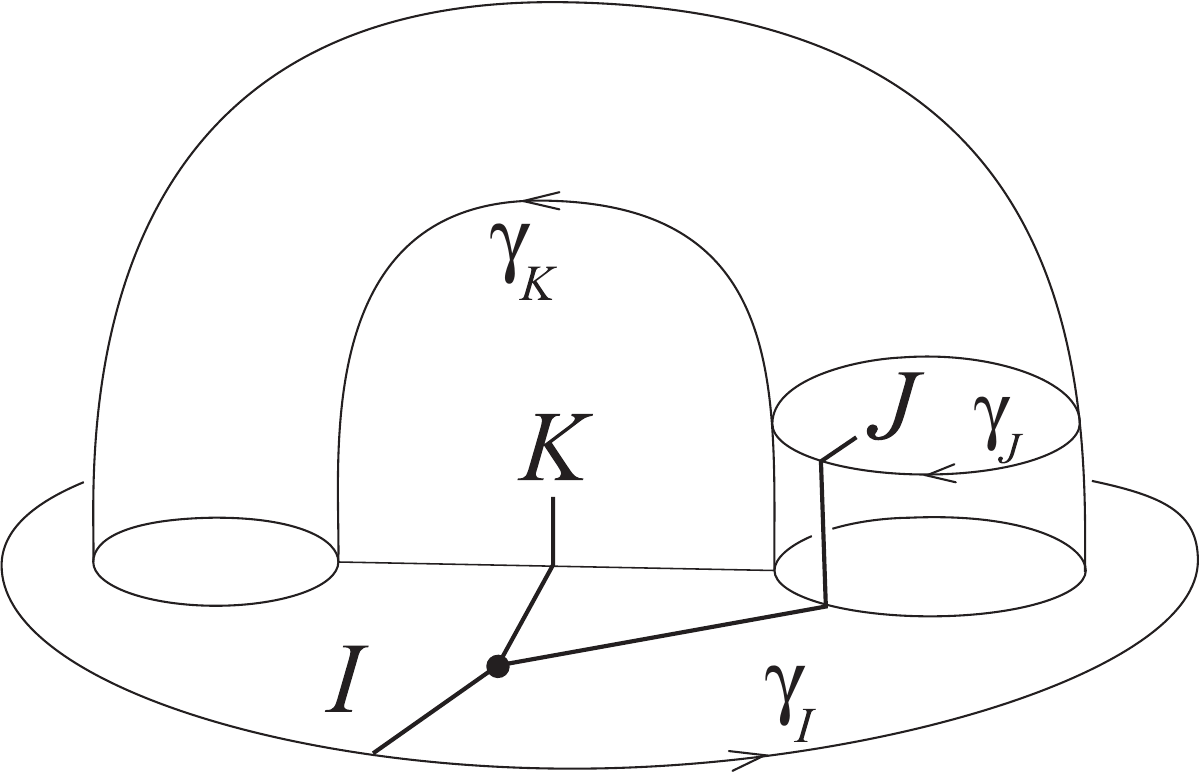}}
         \caption{Near a trivalent vertex in a dyadic branch of $G^c$.}
         \label{higher-order-longitude-vertex-fig}
\end{figure}
Now consider the general order $n$ case with the assumption that $\cW$ contains no twisted Whitney disks,
so that $G^c$ is a class $n+1$ capped grope with no twisted caps. 
That the longitude factors are equal to the iterated commutators
corresponding to putting roots at the univalent vertices of $t(G^c)$ for $n>1$
will follow by applying the computations for $n=1$ to recursively
express the relations between meridians and push-offs of
boundaries of surface stages of $G^c$ at an arbitrary trivalent
vertex of $t(G^c)$. As before we start by considering the case where $G^c$ consists of a single dyadic branch containing $t_p=t(G^c)$:

Figure~\ref{higher-order-longitude-vertex-fig}
shows three surface stages in $G^c$ around a trivalent
vertex which decomposes the (un-rooted) tree $t_p$
into three (rooted) subtrees $I$, $J$, and $K$ (whose roots are identified at the trivalent vertex), 
with the $I$-subtree
reaching down to the bottom stage of $G^c$, and where we assume for the moment that $J$ and $K$ are of positive order (so the $J$- and $K$-sheets are not caps).
Push-offs of the
boundaries of the stages represent fundamental group elements $\gamma_I$,
$\gamma_J$, and $\gamma_K$; and we denote by $x_I$, $x_J$, and
$x_K$ meridians to these stages.

The same computations as in the $n=1$ case now give the three
relations:
$$
\gamma_I=[\gamma_J,\gamma_K],\quad x_J=[\gamma_K,x_I],\quad
\mbox{and}\quad x_K=[x_I,\gamma_J]
$$

If either of $J$ or $K$ is order zero, say $K=k$, then the  
corresponding cap $c_k$ intersects the bottom stage $G_k$, and so
the cap boundary (labeled $\gamma_K$ in Figure~\ref{higher-order-longitude-vertex-fig}) will be a meridian $x_k$ to $G_k$, 
and the cap meridian $x_K$ will be denoted $x_{c_k}$; and the relations become: 
$$
\gamma_I=[\gamma_J,x_k],\quad x_J=[x_k,x_I],\quad
\mbox{and}\quad x_{c_k}=[x_I,\gamma_J]
$$

It follows recursively that, when $J$ and $K$ are of positive order, each of $\gamma_I$, $x_J$, and
$x_K$ are equal to the iterated commutators in the generators corresponding
to $I$, $J$ and $K$:
$$
\gamma_I=[J,K],\quad x_J=[K,I],\quad
\mbox{and}\quad x_K=[I,J].
$$
And if $K=k$ is order zero, then we have 
$$
\gamma_I=[J,x_k],\quad x_J=[x_k,I],\quad
\mbox{and}\quad x_{c_k}=[I,J]
$$
with similar relations for order zero $J=j$.

\begin{figure}
\centerline{\includegraphics[width=75mm]{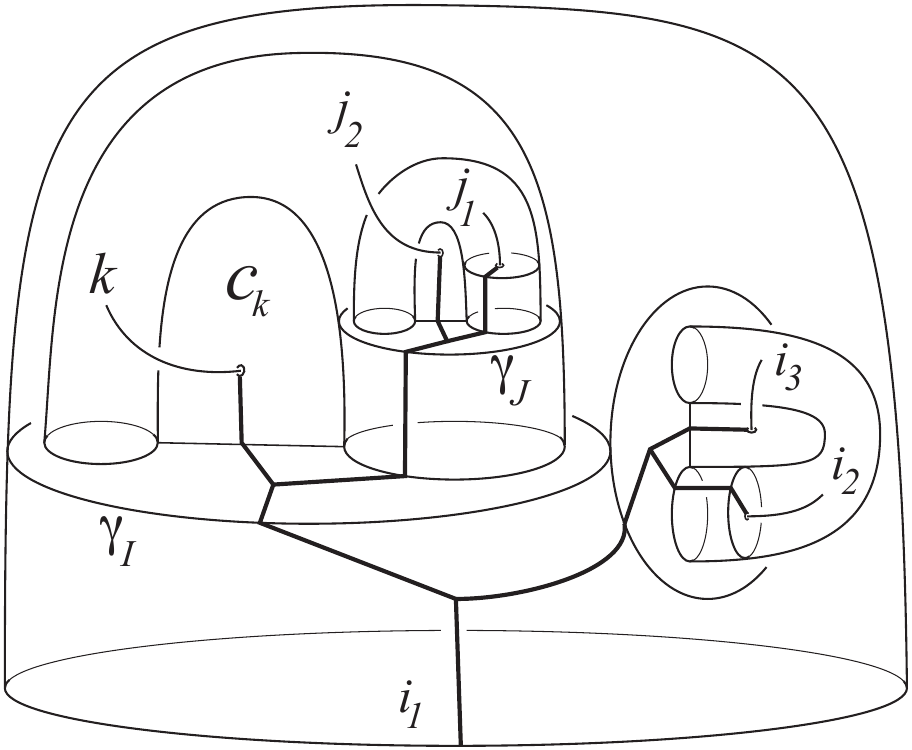}}
         \caption{An example of Figure~\ref{higher-order-longitude-vertex-fig} with $I=(i_1,(i_2,i_3))$ of order $2$,
         $J=(j_1,j_2)$ of order $1$, and 
         $K=k$ of order zero.}
         \label{higher-order-vertex-grope-example-fig}
\end{figure}

Theorem~\ref{thm:Milnor invariant} is confirmed in this case by taking
any of $I$, $J$, and $K$ to be order zero, which shows that the
corresponding factor contributed to the longitude is the iterated
commutator gotten by putting a root at that univalent vertex on
$t_p$. The general framed case follows by summing this computation over all dyadic branches.

For instance, referring to the example of Figure~\ref{higher-order-vertex-grope-example-fig} in the case $n=4$,
the contribution to the longitude $\gamma_k$ coming from the pictured intersection between $G_k$ and the cap $c_k$
is represented by the cap meridian:
$$
x_{c_k}=[x_I,\gamma_J]=[[x_{i_1},[x_{i_2},x_{i_3}]],[x_{j_1},x_{j_2}]]
$$
which is the iterated commutator $\beta_{v_k}(\langle (I,J),k\rangle)$ determined by putting a root at the $k$-labeled vertex $v_k$ of the tree
$\langle (I,J),k\rangle$.

\subsubsection{The general twisted case:}\label{subsubsec:general-twisted-case}

Now consider the general order $n$ case where $G^c$ may contain twisted caps (for even $n$)
corresponding to $\pm 1$-twisted Whitney disks (of order $n/2$) in $\cW$. Again, by additivity of the computation over the dyadic branches
it is enough to consider a single dyadic
branch of $G^c$ containing a $\pm 1$-twisted cap $c_J^\iinfty$, and check that the corresponding $\iinfty$-tree $J^\iinfty$
contributes $\eta_n(J^\iinfty)=\frac{1}{2}\eta_n(\langle J,J \rangle)$ to $\mu_n(L)$.

The key observation in this case is that because the cap
$c_J^\iinfty$ is $\pm 1$-twisted, the element $\gamma_\iinfty$ represented by a parallel push-off of the
(oriented) boundary of the cap is the ($\pm$)-meridian
$x_J^{\pm 1}$ to the cap. For $J=(J_1,J_2)$, referring to Figure~\ref{higher-order-longitude-twisted-vertex-J-fig}
and using the same dual torus as for a framed
cap (Figure~\ref{normal-torus-fig}) this element can be expressed
as the commutator: 
$$
\gamma_\iinfty=x_{(J_1,J_2)}^{\pm 1}=[x_{J_1},\gamma_{J_2}]^{\pm 1}
$$  
where if $J_2=j_2$ is order zero, then $\gamma_{J_2}$ is replaced by the meridian $x_{j_2}$ to $G_{j_2}$ 
(as in the notation for the previous untwisted case). 
\begin{figure}
\centerline{\includegraphics[width=65mm]{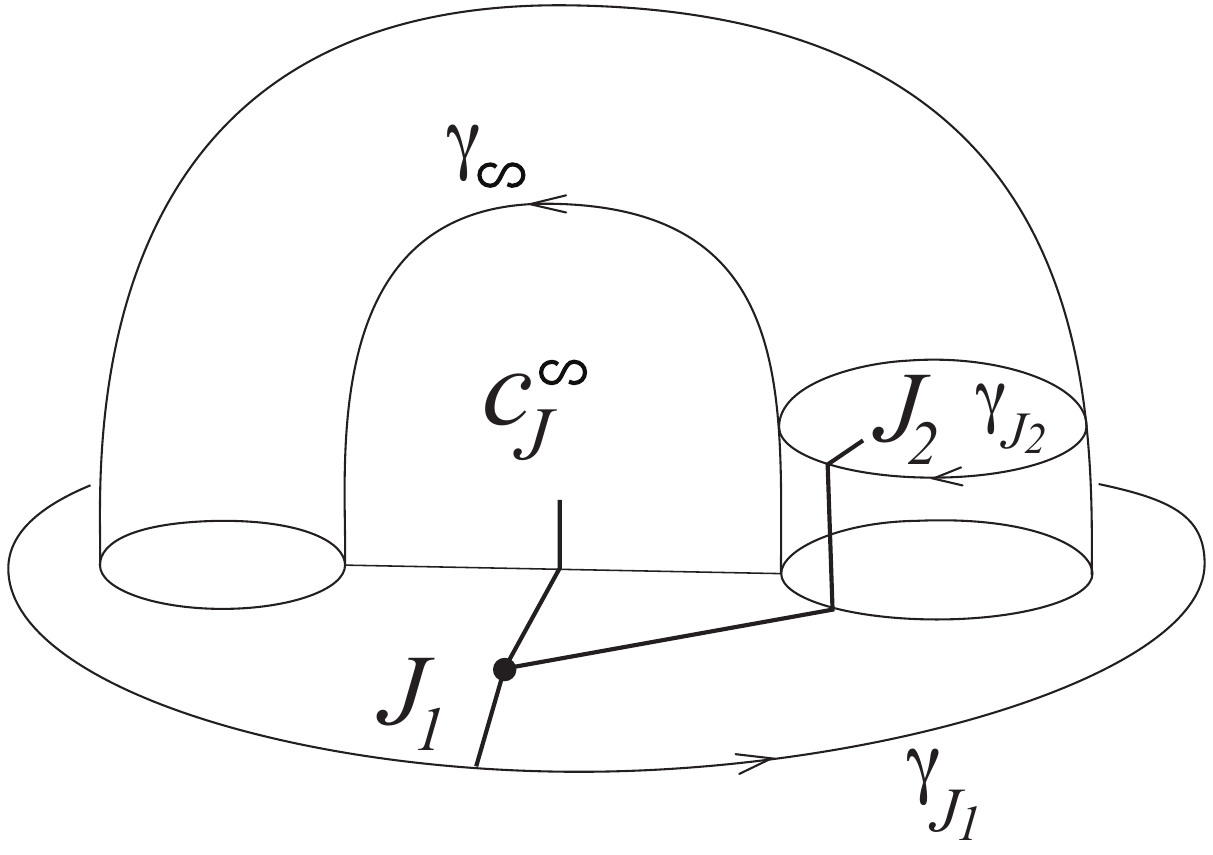}}
         \caption{Near a twisted cap in a dyadic branch of $G^c$.}
         \label{higher-order-longitude-twisted-vertex-J-fig}
\end{figure}

So the analogous computations as in Figure~\ref{higher-order-longitude-vertex-fig}
applied to the twisted setting of Figure~\ref{higher-order-longitude-twisted-vertex-J-fig}
give the relations:
$$
\gamma_{J_1}=[\gamma_{J_2},[x_{J_1},\gamma_{J_2}]]^{\pm 1}\;\;\;\;\;
\mbox{and}\;\;\;\;\;x_{J_2}=[[x_{J_1},\gamma_{J_2}],x_{J_1}]^{\pm 1}
$$
and recursively as in the framed case:
$$
\gamma_{J_1}=[J_2,[J_1,J_2]]^{\pm 1}\;\;\;\;\;
\mbox{and}\;\;\;\;\;x_{J_2}=[[J_1,J_2],J_1]^{\pm 1}
$$
with $J_1$ and $J_2$ denoting the corresponding iterated commutators in the meridional generators.  

To see that the contribution to $\gamma_{i_r}$ corresponding to any $i_r$-labeled vertex $v_{i_r}$ of $J^\iinfty$ is the iterated commutator 
$\beta_{v_{i_r}}(\langle J,J \rangle)$, 
observe that if $v_{i_r}$ is in $J_2$ then the contribution will be an iterated commutator containing $x_{J_2}$, 
and if $v_{i_r}$ is in $J_1$ then the contribution will be the iterated commutator containing
$\gamma_{J_1}$.  Thus, the effect of the twisted cap is to ``reflect'' the iterated commutator determined by $J$ at the $\iinfty$-labeled root.
For instance, in the 
example of Figure~\ref{twisted-vertex-grope-example-fig}
for the case $n=8$, the contribution to the longitude $\gamma_{j_1}$ corresponding to the boundary of the dyadic branch
is: 
$$
\begin{array}{ccl}
[[x_{j_2},x_{j_3}],\gamma_{J_1}] & = & [[x_{j_2},x_{j_3}],[\gamma_{J_2},\gamma_\iinfty]] \\
                                                      & = & [[x_{j_2},x_{j_3}],[[x_{j_4},x_{j_5}],x_{c_J}]]  \\
                                                       & = & [[x_{j_2},x_{j_3}],[[x_{j_4},x_{j_5}],[x_{J_1},\gamma_{J_2}]]] \\
                                                       & = & [[x_{j_2},x_{j_3}],[[x_{j_4},x_{j_5}],[[x_{j_1},[x_{j_2},x_{j_3}]],[x_{j_4},x_{j_5}]]]]\\
                                                       & = & \beta_{v_{j_1}}(\langle J,J \rangle)
\end{array}
$$
for $J=(J_1,J_2)=((j_1,(j_2,j_3)),(j_4,j_5))$ and assuming the twisting of $c_J^\iinfty$ is $+1$.

Since each univalent vertex of $J$ contributes one term to $\mu_n(L)$,
the total contribution of the branch is equal to $\eta_n(J^\iinfty)=\frac{1}{2}\eta_n(\langle J,J \rangle)$.

\begin{figure}
\centerline{\includegraphics[width=65mm]{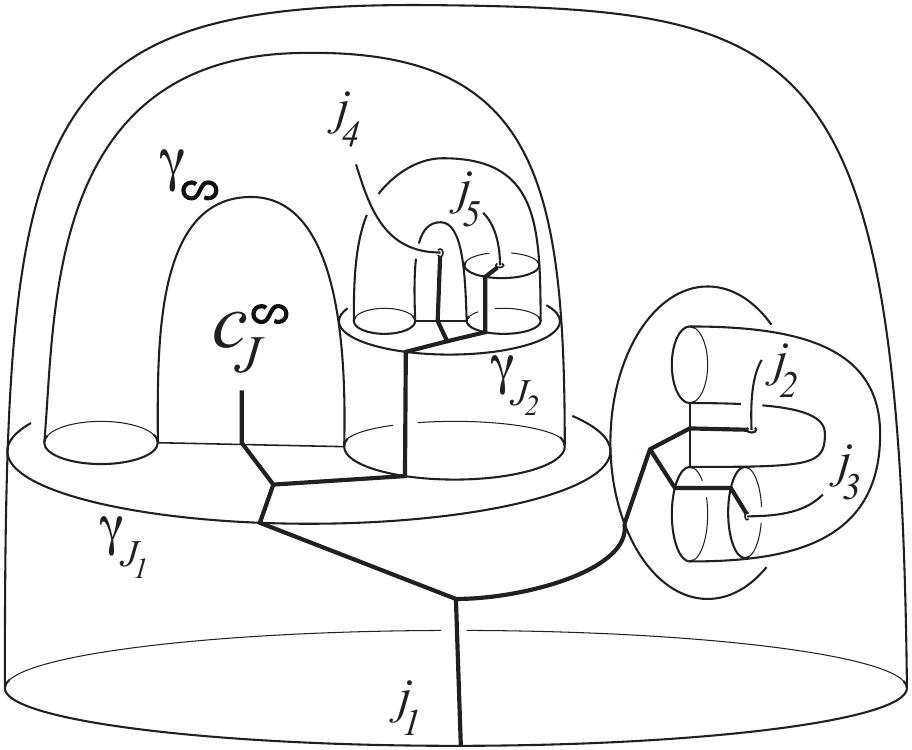}}
         \caption{An example of Figure~\ref{higher-order-longitude-twisted-vertex-J-fig} with $J_1=(j_1,(j_2,j_3))$ of order $2$,
         and $J_2=(j_4,j_5)$ of order $1$. }
         \label{twisted-vertex-grope-example-fig}
\end{figure}

This completes the proof of Theorem~\ref{thm:Milnor invariant}, modulo the proofs of Lemma~\ref{lem:eta-well-defined} and Lemma~\ref{lem:grope-duality} which follow.

\subsection{Proof of Lemma~\ref{lem:eta-well-defined}}\label{subsec:lemma-eta-well-defined-proof}
Lemma~\ref{lem:eta-well-defined} states that $\eta_n:\cT^\iinfty\to\sD_n$ is a well-defined surjection.
Levine showed in \cite{L2} that an analogous map $\eta'_n:\cT_n\to \sD'_n:=\Ker\{\sL'_1 \otimes \sL'_{n+1}\to\sL'_{n+2}\}$ is a well-defined surjection,
where $\eta_n'$ is defined on trees using the same ``sum over all choices of root'' formula as $\eta_n$,
and $\sL'=\oplus_{n\in\N}\sL_n'$ is the free \emph{quasi Lie algebra} on $\{X_1,X_2,\ldots,X_m\}$ gotten from the free $\Z$-Lie algebra $\sL$ by replacing the self-annihlation
relations $[X,X]=0\in\sL$ with the antisymmetry AS relations $[X,Y]+[Y,X]=0\in\sL'$. (See also \cite{CST3}, where we prove the \emph{Levine Conjecture}: that $\eta'_n$ is an isomorphism for all $n$.) It follows that $\eta_n$ vanishes on the usual IHX and AS relations, and maps onto $\sD_n$. So it suffices to check that $\eta_n$ respects the other 
relations in $\cT^\iinfty_n$.

First consider the odd order case. To see that $\eta_n$ vanishes on the boundary-twist relations, observe that 
$\eta_n(\langle (i,J),J\rangle)=0$, since placing a root at the \mbox{$i$-labeled} vertex determines a trivial symmetric bracket in $\sL_{n+1}$,
and all the other Lie brackets come in canceling pairs corresponding to putting roots on vertices in each of the isomorphic
$J$ sub-trees.

Now considering the even order case, we have
$$
\eta_n((-J)^\iinfty)=\frac{1}{2}\eta_n(\langle -J,-J\rangle)=\frac{1}{2}\eta_n(\langle J,J\rangle)=\eta_n(J^\iinfty),
$$
and
$$
\eta_n(2J^\iinfty)=2\cdot\eta_n(J^\iinfty)=2\cdot\frac{1}{2}\eta_n(\langle J,J\rangle)=\eta_n(\langle J,J\rangle).
$$
And for $I$, $H$, and $X$, the terms in a Jacobi relation $I=H-X$, we have
$$
\begin{array}{ccl}
\eta_n(I^\iinfty)  & =  & \frac{1}{2}\eta_n(\langle I,I\rangle)  \\
  & =  & \frac{1}{2}\eta_n(\langle H-X,H-X\rangle)  \\
  & =  &  \frac{1}{2}(\eta_n(\langle H,H\rangle)-2\cdot\eta_n(\langle H,X\rangle)+\eta_n(\langle X,X\rangle))\\
  & = & \eta_n(H^\iinfty+X^\iinfty-\langle H, X\rangle )
\end{array}
$$
where the second equality comes from applying the Jacobi relation to the sub-trees $I$ inside the inner product $\langle I,I\rangle$ and expanding.\hfill$\square$

\subsection{Proof of Lemma~\ref{lem:grope-duality}}\label{subsec:grope-duality-lemma-proof}
By Dwyer's theorem \cite{Dw}, it suffices to show that the inclusion $(S^3 \setminus L)\hookrightarrow (B^4\setminus G^c)$
induces an isomorphism on first homology, and that the relative (integral) second homology group $H_2(B^4\setminus G^c,S^3 \setminus L)$
is generated by maps of closed gropes of class at least $n+2$ (where a grope is \emph{closed} if its bottom stage is compact with empty boundary).

Observe first that $H_1(S^3 \setminus L)$ is Alexander dual to $H^1(L)$ and is hence generated by meridians. Similarly,  $H_1(B^4\setminus G^c)$ is generated by meridians to the bottom stages of the grope. It follows that the inclusion induces an isomorphism on $H_1$, since meridians of the link go to meridians of the bottom
stages.

By Alexander duality, the generators of $H_2(B^4\setminus G^c)$ which don't come from
the boundary are the Clifford tori (or ``linking tori'', see e.g.~1.1, 2.1 of \cite{FQ}) around the intersections
between the caps and the bottom stages of $G^c$. Each such Clifford torus contains a pair of dual circles,
one a meridian $x_k$ to the $k$th bottom stage of $G$ and the other a meridian $x_{c_k}$ to the cap $c_k$. 
Referring to
Figure~\ref{order-one-longitude-A1-fig} and Figure~\ref{normal-torus-fig} (with $i$ and $j$ replaced by $I$ and $J$,
respectively),
the cap meridian $x_{c_k}$ bounds a genus one surface $T_{c_k}$ containing a pair of dual circles, one a meridian 
$x_I$ to the
$I$-stage of $G^c$ containing $\partial c_k$, and the other a parallel push-off of the boundary of
the $J$-stage representing  $\gamma_J$ which is dual to $c_k$ (if $c_k$ is dual to another cap $c_j$, then $\gamma_J$ is just a meridian $x_j$
to the $j$th bottom stage of $G^c$).

Consider first the case where the dyadic branch of $G^c$ containing $c_k$ does not contain a twisted cap, and let
$t=\langle k,(I,J) \rangle\in t(G^c)$ be the corresponding order $n$ tree. Applying the grope duality construction of section 4 in
\cite{KT} to $T_{c_k}$ yields a class $n+1$ grope in $B^4\setminus G$ having $T_{c_k}$ as a bottom stage and associated tree $\langle k,(I,J) \rangle$ 
(with the $k$-labeled univalent vertex corresponding to $T_{c_k}$). 
Since this class $n+1$ grope consists of normal tori and parallel push-offs of higher stages
of $G$ it actually lies in the complement of the caps of $G^c$ (which only intersect the bottom stages of $G$).
The union of this class $n+1$ grope with the Clifford torus is a class $n+2$ closed grope 
with associated order $n+1$ rooted tree $(k,(I,J))$ (with the root corresponding to the Clifford torus), completing the proof in the case where $G^c$ has no twisted caps.

Now consider the case where the dyadic branch of $G^c$ containing $c_k$ does contain a twisted cap 
$c^\iinfty_J$,
with associated $\iinfty$-tree $J^\iinfty$. 
Recall the observation of Section~\ref{subsubsec:general-twisted-case} above that a normal push-off 
of the cap boundary $\partial c^\iinfty_J$ representing $\gamma_\iinfty\in\pi_1(B^4\setminus G^c)$ is a meridian $x_J$ to $c^\iinfty_J$. 
In this case, the grope duality construction of the previous paragraph which builds a grope on $T_{c_k}$
will at some step look for a subgrope bounded by a normal push-off 
of $\partial c^\iinfty_J$. Just as the computations 
in Section~\ref{subsec:computing-the-longitudes} show that $x_J$ represents the iterated commutator in $\pi_1(B^4\setminus G^c)$ corresponding to the rooted tree $J$, the punctured dual torus to $c^\iinfty_J$ bounded by $x_J$ extends to a grope in $B^4\setminus G^c$ with tree $J$.  
Thus the torus $T_{c_k}$ extends to (a map of) a grope in $B^4\setminus G^c$ whose associated tree is gotten by putting a root at 
the corresponding $k$-labeled univalent vertex of (either one of the sub-trees) $J$ in $\langle J,J \rangle$.
It follows that the  
Clifford torus near the cap $c_k$ extends to a grope whose corresponding tree is gotten by inserting
a (rooted) edge into the edge of $\langle J,J \rangle$ adjacent to the $k$-labeled univalent vertex.
Since the order of $\langle J,J \rangle$ is $n$, it follows that
the class of the grope containing the Clifford torus as a bottom stage is $n+2$.$\hfill\square$


\section{The order $2$ twisted intersection invariant and the classical Arf invariant}\label{sec:proof-lem-Arf}
This section contains a proof of Lemma~\ref{lem:Arf} from the introduction.

Recall the statements of Lemma~\ref{lem:Arf}: Any knot $K$ bounds a twisted Whitney tower $\cW$ of order $2$ and the classical Arf invariant of $K$ can be identified with the intersection invariant
$ \tau^\iinfty_2(\cW) \in \cT^\iinfty_2(1) \cong \Z_2$;
and more generally, the classical Arf invariants of the components of an $m$-component link give an isomorphism
$\Arf: \Ker(\mu_2:\W^\iinfty_2\sra \sD_2) \overset{\cong}{\to} (\Z_2\otimes \sL_1) \cong (\Z_2)^m$.

\begin{proof}
Starting with the first statement, observe that any knot $K\subset S^3$ bounds an immersed disk $D \looparrowright B^4$, and by performing cusp homotopies as needed it can be
arranged that all self-intersections of $D$ come in canceling pairs admitting order $1$ Whitney disks. These Whitney disks can be made to have disjointly embedded boundaries
by a regular homotopy applied to Whitney disk collars (Figure~3 in \cite{ST1}).
It is known that the sum modulo $2$ of the number of intersections between $D$ and the Whitney disk interiors together with the framing obstructions on all the Whitney disks
is equal to $\Arf (K)$ (see \cite{FK,FQ,Ma} and sketch just below).
By performing boundary-twists on the Whitney disks as in Figure~\ref{boundary-twist-and-section-fig} (each of which changes a framing obstruction by 
$\pm 1$), it can be arranged that all intersections between $D$ and the Whitney disk interiors 
come in canceling pairs. This means that $\Arf (K)$ is now equal to the sum modulo $2$ of the twistings on all the order $1$ Whitney disks, and that all order $1$ intersections can be paired by
order $2$ Whitney disks. So $K$ bounds an order $2$ twisted Whitney tower $\cW$ with $\Arf (K)=\tau^\iinfty_2(\cW)$
which counts the $(1,1)^\iinfty$ in
$\cT^\iinfty_2(1)\cong\Z_2$.
On the other hand, given an arbitrary order $2$ twisted Whitney tower $\cW$ bounded by $K$, one has  $\Arf (K)=\tau^\iinfty_2(\cW)\in\cT^\iinfty_2(1)\cong\Z_2$ determined again as the sum modulo $2$ of twistings on all order $1$ Whitney disks.

\begin{figure}
\centerline{\includegraphics[width=130mm]{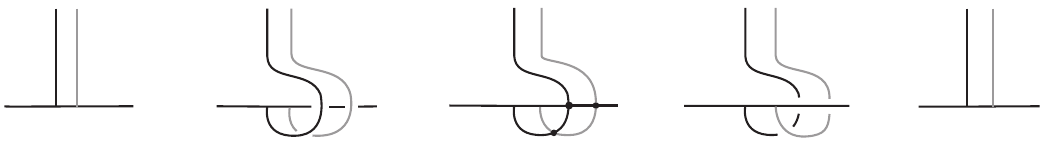}}
         \caption{Boundary-twisting a Whitney disk $W$ changes $\omega (W)$ by $\pm 1$ and creates an intersection point with one of the sheets paired by $W$. The horizontal arcs trace out part of the sheet, the dark non-horizontal arcs
         trace out the newly twisted part of a collar of $W$, and the grey arcs indicate part of the Whitney section over $W$. The bottom-most
         intersection in the middle picture corresponds to the $\pm 1$-twisting created  by the move.}
         \label{boundary-twist-and-section-fig}
\end{figure}

We sketch here a proof that $\Arf (K)$ is equal to the sum modulo $2$ of the order $1$ intersections plus framing obstructions in any  \emph{weak} order $1$ Whitney tower $\cW\subset B^4$ bounded by $K\subset S^3$. Here ``weak'' means that the Whitney disks are not necessarily framed. (We are assuming that the Whitney disk boundaries are disjointly embedded, although we could instead also count Whitney disks boundary singularities.)  Any $K$ bounds a Seifert surface $F\subset S^3$, and by definition
$\Arf (K)$ equals the sum modulo $2$ of the products of twistings on dual pairs of $1$-handles of $F$. Restricting to the case where $F$ is genus $1$, denote by $\gamma$ and $\gamma'$ core circles of the pair of dual $1$-handles of $F$, with respective twistings $a$ and $a'$, so that $\Arf(K)$ is the product $aa'$ modulo $2$. Let $D_\gamma$ be any immersed disk bounded by $\gamma$ into $B^4$, so that the interior of $D_\gamma$ is disjoint from $F$. After performing $|a|$ boundary-twists on $D_\gamma$, each of which creates a single intersection between $D_\gamma$ and $F$, it can be arranged that $D_\gamma$ is framed with respect to $F$, so that surgering $F$ along $D_\gamma$ creates only canceling pairs of self-intersections in the resulting disk $D$ bounded by $K$. Each self-intersection in $D_\gamma$ before the surgery contributes \emph{two} canceling pairs of self-intersections of $D$, since the surgery adds both $D_\gamma$ and an oppositely oriented parallel copy of $D_\gamma$ to create $D$. On the other hand, the $|a|$ intersections between $D_\gamma$ and $F$ before the surgery give rise to exactly $|a|$ canceling pairs of self-intersections of $D$, so the total number of canceling pairs of self-intersections of $D$ is equal to $a$ modulo $2$. Observe that all of these canceling pairs admit Whitney disks constructed from parallel copies of any immersed disk bounded by $\gamma'$ with interior in $B^4$.  The framing obstruction on each of these Whitney disks is equal to the twisting $a'$ along $\gamma'$, and the only order $1$ intersections between the Whitney disk interiors and $D$ come in canceling pairs, since they correspond to intersections with $D_\gamma$ and its parallel copy. Thus the sum of framing obstructions and order $1$ intersections is equal to the product $aa'$ modulo $2$. The higher genus case is similar. That this construction is independent of the choice of weak Whitney tower follows from the fact that the analogous homotopy invariant for $2$--spheres in $4$--manifolds vanishes on any immersed $2$--sphere in the $4$--sphere (e.g.~\cite{ST1}, or \cite{FQ} 10.8A and 10.8B).  

Considering now the second statement of Lemma~\ref{lem:Arf} regarding links, it follows from Corollary~\ref{cor:Milnor invariants} and Proposition~\ref{prop:kerEta4k-2} that if $L$ is any link in $\Ker(\mu_2)<\W_2^\iinfty$, and $\cW$ is any order $2$ 
twisted Whitney tower bounded by $L$; then $\tau^\iinfty_2(\cW)$ is contained in the subgroup of $\cT_2^\iinfty$  spanned by the symmetric twisted trees $(i,i)^\iinfty$, and this subgroup is isomorphic to $(\Z_2)^m$. By the first statement of Lemma~\ref{lem:Arf}, the desired isomorphism $\Arf(L)$ is given by $\tau^\iinfty_2(\cW)$.

So to finish the proof of Lemma~\ref{lem:Arf} it suffices to show that for any $L\in\Ker(\mu_2)<\W_2^\iinfty$, $L=0\in\W_2^\iinfty$ if and only if 
$\tau^\iinfty_2(\cW)=0$. But if $L=0\in\W_2^\iinfty$, then by definition $L$ bounds an order $3$ twisted Whitney tower, so
$\tau^\iinfty_2(\cW)=0$.  And if $\tau^\iinfty_2(\cW)=0$, then $L$ bounds an order $3$ twisted Whitney tower by Theorem~\ref{thm:twisted-order-raising-on-A}, hence $L=0\in\W_2^\iinfty$.
\end{proof}


\section{Boundary links and higher-order Arf invariants}\label{sec:Bing-Arf-proof}
This section contains proofs of Lemma~\ref{lem:Bing} and Proposition~\ref{prop:Arf-2-and-greater}
from the introduction.

\subsection{Proof of Lemma~\ref{lem:Bing}}\label{subsec:proof-lemma-Bing}
Recall the statement of Lemma~\ref{lem:Bing}: 
For any rooted tree $J$ of order $k-1$,
by performing iterated untwisted Bing-doublings and interior band sums on the figure-eight knot $K$ one can create a boundary link $K^J$ as the boundary of a twisted Whitney tower $\cW$ of order $4k-2$ with
$\tau^\iinfty_{4k-2}(\cW)= \iinfty\!\!-\!\!\!\!\!-\!\!\!<^{\,\,J}_{\,\,J}$.

\begin{figure}
\centerline{\includegraphics[width=135mm]{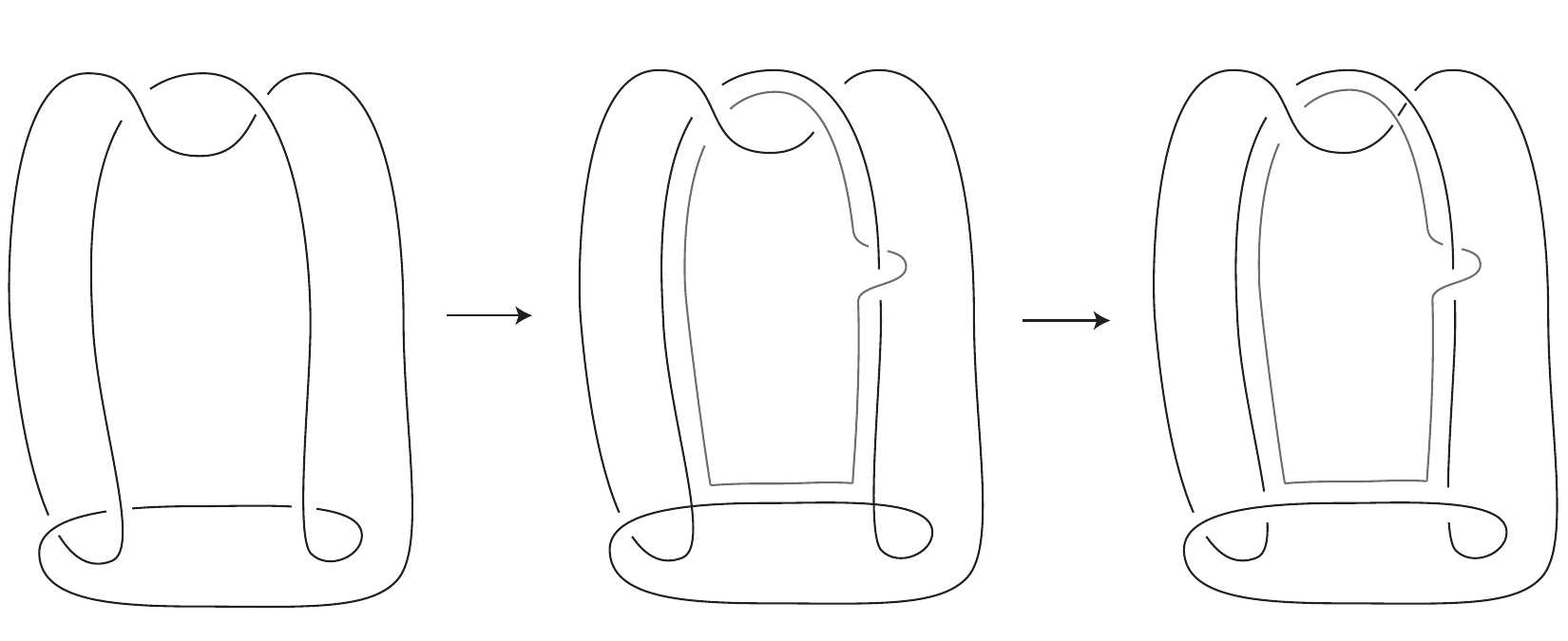}}
         \caption{From left to right: The trace of a null-homotopy of the figure-eight knot describes an order $0$ disk $D$ with a canceling pair of self-intersections that are paired by a clean $+1$-twisted Whitney disk $W$. A collar of $W$ is indicated by the grey loop in the middle diagram, and the unlink in the right hand diagram can be capped off by two embedded disks which form the rest of $D$ and $W$. The twisting $\omega (W)=1$ of $W$ corresponds to the twist in the collar annulus in the middle diagram, as explained in Figure~\ref{Fig8knot-Wdisk-Whitney-section-fig}.}
         \label{Fig8knot-and-twisted-Wdisk-fig}
\end{figure}
\begin{figure}
\centerline{\includegraphics[width=95mm]{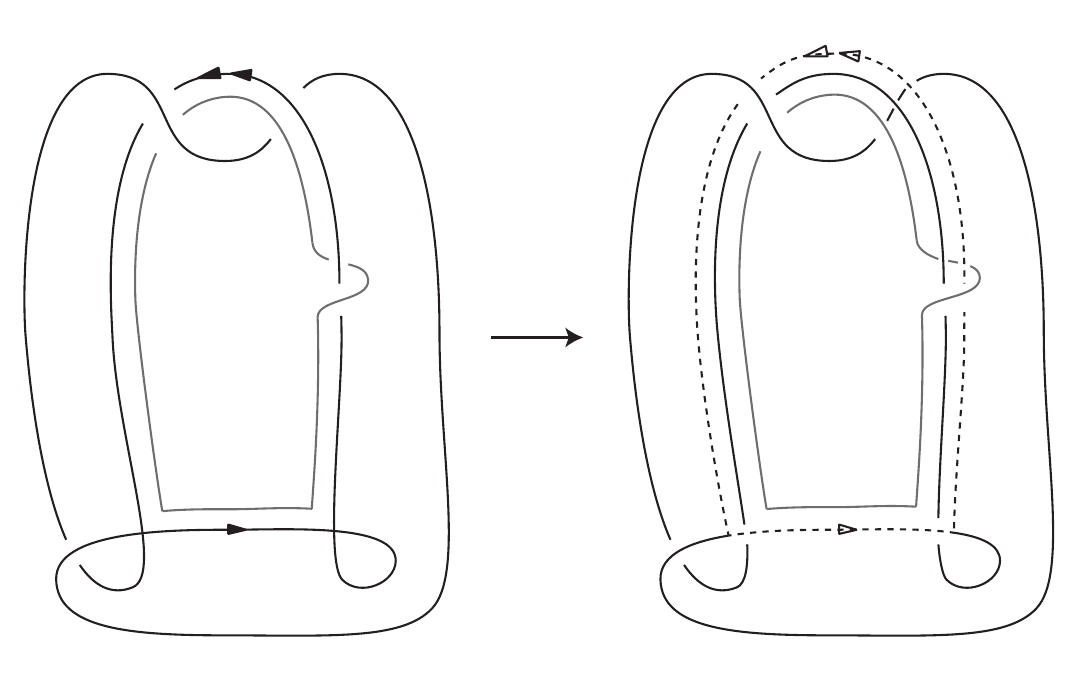}}
         \caption{Shown are the two right-most pictures from Figure~\ref{Fig8knot-and-twisted-Wdisk-fig} with the Whitney disk boundary arcs indicated by the arrows on the left, and the corresponding arcs of the Whitney section shown by the dotted arcs on the right. 
(Compare with the framed Whitney disk in Figure~\ref{Framing-of-Wdisk-fig}.) The $+1$-linking between the grey circle (boundary of a collar of $W$) and the dotted circle (Whitney section) corresponds to the twisting 
$\omega(W)=1$. }
         \label{Fig8knot-Wdisk-Whitney-section-fig}
\end{figure}

First of all, Figures~\ref{Fig8knot-and-twisted-Wdisk-fig} and \ref{Fig8knot-Wdisk-Whitney-section-fig} show that the figure-eight knot $K$ bounds an order $2$ twisted Whitney tower consisting of an order $0$ disk $D$ which contains a single canceling pair of self-intersections, and a clean $+1$-twisted Whitney disk $W$ pairing the self-intersections of $D$. (This exhibits the fact that $K$ has non-trivial classical Arf invariant, as in Lemma~\ref{lem:Arf}; and the same computation shows that the Whitney disk in Figure~\ref{fig:Whitehead-12infty-tower} is $+1$-twisted.)

The link $\Bing (K)$ pictured in the left hand side of Figure~\ref{BingFig8knot-and-banded-sheets-fig} is the untwisted Bing-double of the figure-eight knot, where ``untwisted'' refers to the untwisted band with core $K$ used to guide the construction of the two clasped unknotted components. The right-hand side of 
Figure~\ref{BingFig8knot-and-banded-sheets-fig} shows how
the untwisted Bing-double of any boundary link is again a boundary link, as disjoint Seifert surfaces for the new pair of components can be constructed by banding together four parallel copies of the Seifert surface for the original component which was doubled.

Moving into $B^4$, a null-homotopy of 
$\Bing(K)$ which pulls apart the clasps (and is supported near the untwisted band) describes embedded order $0$ disks $D_1$ and $D_2$ which have a single canceling pair of intersections (corresponding to the crossing-changes undoing the clasps).  The boundary of an order $1$ Whitney disk $W_{(1,2)}$ pairing $D_1\cap D_2$ sits in an $S^3$-slice of $B^4$ as a figure-eight knot, exactly as in the left-hand side of Figure~\ref{Fig8knot-and-twisted-Wdisk-fig}, with the rest of
$W_{(1,2)}$ described by the same null-homotopy as for $D$ of Figure~\ref{Fig8knot-and-twisted-Wdisk-fig}. Since $\Bing(K)$ is untwisted,
$W_{(1,2)}$ is framed \cite[Sec.3]{WTCCL}. The interior of $W_{(1,2)}$ is disjoint from $D_1$ and $D_2$, but contains a canceling pair of self-intersections corresponding to the self-intersections of $D$ in Figure~\ref{Fig8knot-and-twisted-Wdisk-fig}. These self-intersections can be paired by an order $3$ clean $+1$-twisted Whitney disk (the $W$ in Figure~\ref{Fig8knot-and-twisted-Wdisk-fig}) which has $\iinfty$-tree $((1,2),(1,2))^\iinfty$. Thus, $K^J:=\Bing(K)$ for $J=(1,2)$ bounds an order $6$ twisted Whitney tower $\cW$ with
$\tau_6^\iinfty(\cW)=((1,2),(1,2))^\iinfty\in\cT^\iinfty_6$. As mentioned in the introduction, Conjecture~\ref{conj:Arf-k} would imply that
the $\Bing(K)$ does \emph{not} bound any order $7$ (twisted) Whitney tower $\cW'$ on immersed disks in the $4$--ball.
Note that if $\Bing(K)$ did indeed bound such a $\cW'$, then taking the union of $\cW$ and $\cW'$ along $\Bing(K)$ would yield
a pair of immersed $2$--spheres in the $4$--sphere supporting the order $6$ twisted Whitney tower $\cV$ with $\tau_6(\cV)=((1,2),(1,2))^\iinfty$.
This would imply that $\Arf_2(L)=0$ for any link $L$ since by tubing these $2$--spheres as needed into any order $6$ twisted Whitney tower bounded by 
$L$ one could kill any $2$-torsion elements  $((i,j),(i,j))^\iinfty\in\cT^\iinfty_6$.

\begin{figure}
\centerline{\includegraphics[width=115mm]{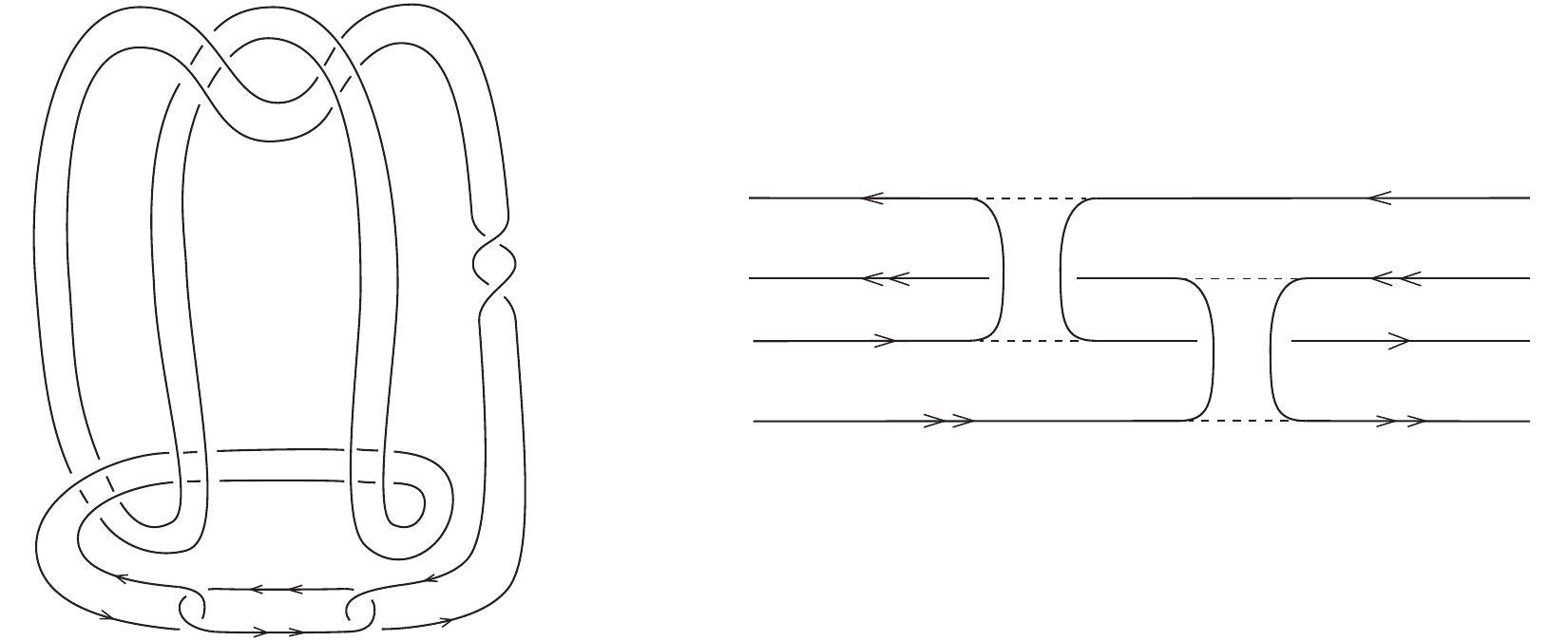}}
         \caption{On the left, the (untwisted) Bing-double of the figure-eight knot. On the right, disjoint Seifert surfaces for the two components of a Bing-doubled knot can be constructed from four copies of a Seifert surface bounded by the original knot.}
         \label{BingFig8knot-and-banded-sheets-fig}
\end{figure}

The construction for arbitrary $J$ follows inductively by observing that having realized
$(J,J)^\iinfty$ by $K^J$, with $J$ of order $r$, Bing-doubling a component of $K^J$ realizes  $(J',J')^\iinfty$,
with $J'$ of order $r+1$ gotten from $J$ by adding two new edges to a univalent vertex of $J$ (see Figure~\ref{Bing-W-disk-in-collar-fig}); and any $J'$ of order $r+1$ can be gotten from some such $J$, by banding together some link components as necessary to get repeated univalent labels. For instance, to realize $((1,(1,2)),(1,(1,2)))^\iinfty$ from $((1,2),(1,2))^\iinfty$ realized by
$\Bing(K)$ above, just Bing-double the second component of $\Bing(K)$ and then band one of the new components into the old first component. 
$\hfill\square$

\subsection{Proof of Proposition~\ref{prop:Arf-2-and-greater}} \label{subsec:proof-prop-Arf-2-and-greater}
As stated in Conjecture~\ref{conj:Arf-k}, we believe that $\Arf_k$ is non-trivial for all $k$; however interest in the first unknown ``test case''  $k=2$ is heightened by Proposition~\ref{prop:Arf-2-and-greater} from the introduction which states that if $\Arf_2=0$ then $\Arf_k$ is trivial for all $k\geq 2$.

\begin{figure}
\centerline{\includegraphics[scale=.375]{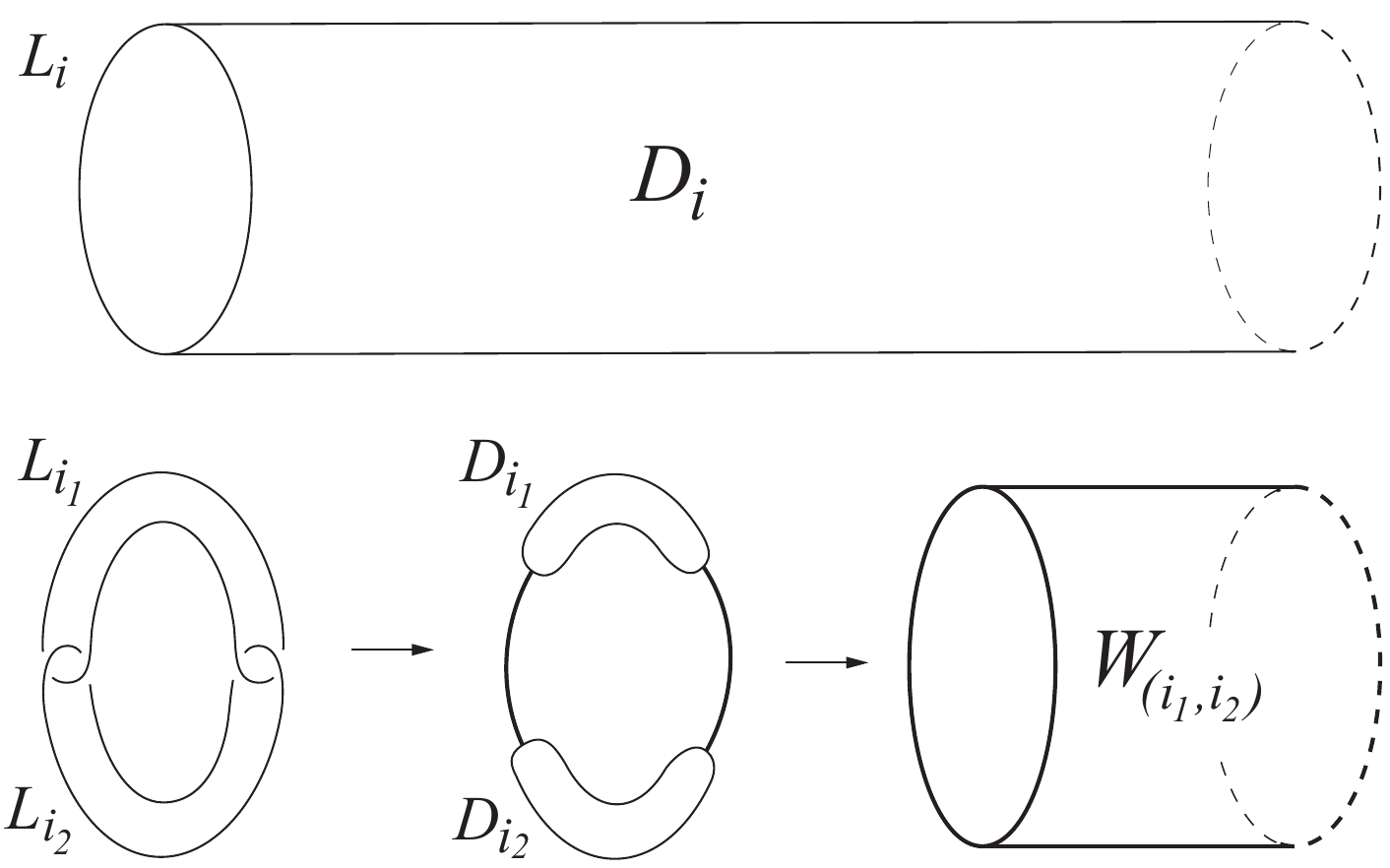}}
         \caption{Moving into $B^4$ from left to right: Above, a collar of $L_i$ in $D_i$.
         Below, $D_i$ yields a Whitney disk $W_{(i_1,i_2)}$ for the intersections between $D_{i_1}$ and $D_{i_2}$ bounded by $L_{i_1}$ and $L_{i_2}$, the untwisted Bing-double of $L_i$. $D_{i_1}$ and $D_{i_2}$ are traced out by null-homotopies of $L_{i_1}$ and $L_{i_2}$; and the curved vertical arcs are part of $W_{(i_1,i_2)}$.}
         \label{Bing-W-disk-in-collar-fig}
\end{figure}

By Proposition~\ref{prop:kerEta4k-2} it suffices to show that if the untwisted Bing-double $\Bing(K)=K^{((1,2),(1,2))}$ bounds an order $7$ twisted Whitney tower for some $K$ with non-trivial classical Arf invariant, then each link $K^J$ of Lemma~\ref{lem:Bing} with $J$ of order $k-1$ bounds an order $4k-1$ twisted Whitney tower, for $k>2$.

Note that the assumption that $\Bing(K)$ bounds an order $7$ twisted Whitney tower implies that
$\Bing(K)$ in fact bounds an order $10$ twisted Whitney tower $\cW$ by Theorem~\ref{thm:twisted-three-quarters-classification}, since boundary links have vanishing Milnor invariants in all orders. 
Now applying the Bing-doubling and banding construction of the proof of Lemma~\ref{lem:Bing} to get $K^{J'}$ from $\Bing(K)$, where $J'$ is any 
order $2$ tree gotten from the order $1$ tree $J=(1,2)$,  
yields $K^{J'}$ bounding an order $11$ twisted Whitney tower gotten from $\cW$
by converting an order $0$ disk of $\cW$ into an order $1$ Whitney disk (Figure~\ref{Bing-W-disk-in-collar-fig}). Inductively, if $K^J$, with $J$ of order $k-1$, bounds an order $4k-1$ twisted Whitney tower, then $K^J$ also bounds an order $4k+2$ twisted Whitney tower by Theorem~\ref{thm:twisted-three-quarters-classification}, and since Bing-doubling a component of $K^J$ raises the order by at least $1$ it follows that $K^{J'}$ bounds an order $4(k+1)-1$ twisted Whitney tower,
for any $J'$ gotten from $J$ by attaching at least one pair of new edges to a univalent vertex of $J$. 
As was observed in the proof of Lemma~\ref{lem:Bing}, 
all trees can be gotten by this process of adding new pairs of edges to univalent vertices of lower-order trees.
 $\hfill\square$


\section{Milnor invariants and geometric $k$-sliceness}\label{sec:Milnor-geo-k-slice}
This section gives proofs of Theorem~\ref{thm:geo-k-slice-equals-odd-W} and Theorem~\ref{thm:mega-k-slice} from the introduction. The proof of Theorem~\ref{thm:geo-k-slice-equals-odd-W} uses the classification of the twisted Whitney tower filtration from \cite{WTCCL}, together with the Whitney tower-to-grope techniques of \cite{S2} (as sketched in Section~\ref{section:twisted-towers-and-gropes} above). The proof of Theorem~\ref{thm:mega-k-slice} will use Theorem~\ref{thm:geo-k-slice-after-sums} of the introduction, together with a mild generalization of Theorem~\ref{thm:Milnor invariant} given by Proposition~\ref{prop:mu=tau-on-immersed-surfaces} below.
 
 \subsection{Proof of Theorem~\ref{thm:geo-k-slice-equals-odd-W}}
Recall from Section~\ref{subsec:intro-k-slice} that Theorem~\ref{thm:geo-k-slice-equals-odd-W} states: a link $L$ is geometrically $k$-slice if and only if $\mu_n(L)=0$ for all $n\leq 2k-2$ and $\Arf_n(L)=0$ for all  $n\leq \frac{1}{2}k$; where $L\subset S^3$ is \emph{geometrically $k$-slice} if the components $L_i$ bound disjointly embedded (oriented) surfaces $\Sigma_i\subset B^4$ such that a symplectic basis of curves on each $\Sigma_i$
bound disjointly embedded framed class $k$ gropes in the complement of $\Sigma := \cup_i\Sigma_i$.

\begin{proof}
By the classification of the twisted Whitney tower filtration
\cite[Cor.1.16]{WTCCL}, the stated vanishing of $\mu_n(L)$ and $\Arf_n(L)$ 
is equivalent to $L$ bounding an order $2k-1$ twisted Whitney tower in $B^4$ (see Theorem~\ref{thm:twisted-three-quarters-classification}, Proposition~\ref{prop:kerEta4k-2} and Definition~\ref{def:Arf-k} above).
So it will suffice to show that $L$ is in $\mathbb W^\iinfty_{2k-1}$ if and only if $L$ is geometrically $k$-slice.

Recall that $\mathbb W^\iinfty_{2k-1}=\mathbb W_{2k-1}$ by definition, so we may assume that $L$ bounds a framed Whitney tower $\cW$ of order $2k-1$. By applications of the Whitney-move IHX construction
 (Section~7 of \cite{S1}) it can be arranged that all trees in the intersection forest $t(\cW)$ are \emph{simple}, meaning that
 every trivalent vertex is adjacent to at least one univalent edge. Since all these simple trees are of order (at least) $2k-1$ we can choose a preferred univalent vertex on each tree which is (at least) $k-1$ trivalent vertices away from both ends of its tree.
Now converting the order $2k-1$ Whitney tower $\cW$ to a class $2k$ embedded grope $G$ via (the framed part of) the above construction in the proof of Lemma~\ref{lem:twisted-tower-to-grope} (as described in detail in \cite{S1}), with the preferred univalent vertices corresponding to the bottom stages of the connected components of $G$, yields dyadic branches having bottom stages with a symplectic basis of circles bounding gropes of class (at least) $k$ (each tree yields one hyperbolic pair of such circles).

Note that the construction of \cite{S1} used here, as in the proof of Lemma~\ref{lem:twisted-tower-to-grope}, yields a \emph{capped} grope $G^c$ which is contained in any small neighborhood of $\cW$. In this argument we only need the body $G$.

On the other hand, being geometrically $k$-slice is the same as bounding a particular kind of embedded class $2k$ grope $G\subset B^4$. Since $B^4$ is simply connected, caps can be found, and can be framed by twisting as necessary. All intersections in the caps can be pushed down into the bottom grope stages using finger moves, yielding a capped grope $G^c$ bounded by the link, which can be converted to an order $2k-1$ Whitney tower via the inverse operation to that used in the proof of Lemma~\ref{lem:twisted-tower-to-grope} above (see Theorem~6 of \cite{S1}).
\end{proof}

 \subsection{Proof of Theorem~\ref{thm:mega-k-slice}}\label{subsec:proof-cor-mega-k-slice}
Recall the statement of Theorem~\ref{thm:mega-k-slice}:
A link $L=\cup_i L_i$ has $\mu_n(L)=0$ for all $n\leq 2k-2$
if and only if the link components $L_i$ bound
disjointly embedded surfaces $\Sigma_i$ in the $4$--ball, with each surface a
connected sum of two surfaces $\Sigma'_i$ and $\Sigma''_i$ such that
 a symplectic basis of curves on $\Sigma'_i$
 bound disjointly embedded framed gropes $G_{i,j}$ of class $k$ in the complement of 
 $\Sigma := \cup_i\Sigma_i$, and 
 a symplectic basis of curves on $\Sigma''_i$ bound immersed disks in the
 complement of 
 $\Sigma\cup G$, where $G$ is the union of all $G_{i,j}$.

 \begin{proof}
Given $L$ with vanishing Milnor invariants of all orders $\leq 2k-2$, by Theorem~\ref{thm:geo-k-slice-after-sums} there exist finitely many boundary
links as in Lemma~\ref{lem:Bing} such that taking band sums of $L$ with all these boundary links yields a geometrically $k$-slice link $L'\subset S^3$. Consider each of these boundary links to be contained in a $3$--ball, and embed these $3$--balls disjointly in a single $3$--sphere, so the union of the boundary links forms a single boundary link denoted $U$. Decompose the $3$--sphere $S^3=B^3_L\cup_{S_0^2}B^3_U$ containing $L'$ into two $3$--balls exhibiting the band sum of 
$L'=L\#U$, with $L\subset B^3_L$, and $U\subset B^3_U$. Each band in the sum intersects the separating $2$--sphere $S_0^2$ in a single transverse arc. Since $L'$ is geometrically $k$-slice, $L'$ bounds $\Sigma'\subset B^4$ which satisfies the conditions in the first item of Theorem~\ref{thm:mega-k-slice}.

Now consider $S^3=B^3_L\cup_{S_0^2}B^3_U$ as the equator of a $4$--sphere $S^4$, with the interior of $\Sigma'\subset B^4$ contained in the `southern hemisphere' $B^4\subset S^4$. The components of the boundary link $U$ bound disjoint
Seifert surfaces which are contained in $B^3_U$, and symplectic bases of these Seifert surfaces bound immersed disks into the `northern hemisphere' $4$--ball in $S^4$. We may assume that the interiors of these immersed disks are contained in a `northern quadrant' of $S^4$, which is a $4$--ball $B^4_+$ bounded by
a $3$--sphere consisting of $B^3_+\cup_{S_0^2}B^3_U$, where $B^3_+$ is a $3$--ball bounded by $S^2_0$ whose interior cuts the northern hemisphere into two $4$--balls.
Gluing $B^4_+$ to the southern hemisphere $B^4$ along $B^3_U$, with the boundaries of the Seifert surfaces glued along $U$, has the effect of eliminating $U$ from the band sum with $L$ ($U$ gets replaced by the unlink). This leaves $L$ in the $3$--sphere $B^3_+\cup_{S_0^2}B^3_+$ which bounds the `other' northern quadrant of $S^4$, and the union $\Sigma''$ of the Seifert surfaces       
together with $\Sigma'$ form the surfaces $\Sigma$ as desired.

Conversely, suppose the components $L_i$ of $L$ bound disjointly embedded surfaces $\Sigma_i\subset B^4$ as in the statement. The class $k$ gropes $G_{i,j}$ attached to dual circles in $\Sigma_i'$ can be thought of as grope branches of class $2k$ by subdividing each $\Sigma_i'$ into genus one pieces. 
Caps can be chosen for all tips of these branches, and by pushing down intersections (using finger moves) it can be arranged that the caps only intersect the bottom stages $\Sigma_i'$. This means that the caps are disjointly embedded, and disjoint from the immersed disks bounded by the symplectic bases in $\Sigma''$. Now applying the capped grope-to-Whitney tower construction of \cite[Thm.6]{S1} to these capped branches yields an order $2k-1$ Whitney tower $\cW$ on immersed surfaces $S_i$ each bounded by $L_i$ such that all Whitney disks and singularities of $\cW$ are contained
in a neighborhood of the capped branches, and with $\Sigma_i''\subset S_i$ for each $i$.
In particular, a symplectic basis on each $S_i$ bounds immersed disks whose interiors are contained in the complement of $\cW$. 

The proof is completed by the following proposition, which mildly generalizes Theorem~\ref{thm:Milnor invariant} and in particular implies that $L$ as above has vanishing Milnor invariants of all orders $\leq 2k-2$. 
\end{proof}

\begin{prop}\label{prop:mu=tau-on-immersed-surfaces}
Theorem~\ref{thm:Milnor invariant} holds for
an order $n$ twisted Whitney tower $\cW\subset B^4$ on order zero immersed surfaces $S_i$ bounded by
$L$ such that a symplectic basis of curves
on each $S_i$ bounds immersed disks in the complement of $\cW$: The Milnor invariants $\mu_k(L)$ vanish for $k<n$, and $\mu_n(L)=\eta_n\circ\tau^\iinfty_n(\cW)$.
\end{prop}
\begin{proof}
We work through each step of the proof of Theorem~\ref{thm:Milnor invariant} given in Section~\ref{sec:Milnor-thm-proof}, checking that the assertions still hold when the order $0$ disks bounded by the link components are replaced by the surfaces $S_i$:

The twisted Whitney tower $\cW$ is resolved to a twisted capped grope $G^c$ just as in the proof of Lemma~\ref{lem:twisted-tower-to-grope}
except that the bases of curves from the $S_i$ are left uncapped. Note that $G^c$ does not really have class $n+1$ because no higher grope stages are attached to these basis curves;
however, we will see that the proof still goes through since
these curves bound immersed disks in $B^4\setminus G^c$.

To see that Lemma~\ref{lem:grope-duality} still holds, the only new point that needs to be checked in the proof given in Section~\ref{subsec:grope-duality-lemma-proof}
is that the new generators of $H_2(B^4\setminus G^c)$ which are Alexander dual to the basis curves on the $S_i$ are represented by maps of gropes of class at least $n+2$. These new generators are in fact represented by maps of $2$--spheres (which are gropes of arbitrary high class): A torus consisting of the union of circle fibers in the normal circle bundle over a basis curve contains a dual pair of circles, one of which is a meridian to $S_i$ (and bounds a normal disk to the basis curve, exhibiting Alexander duality), while the other circle (which is parallel to the basis curve) bounds by assumption an immersed disk in the complement of $G^c$. Therefore, each such torus can be surgered to an immersed $2$--sphere in $B^4\setminus G^c$.

It only remains to check that the computation of the link longitudes in Section~\ref{subsec:computing-the-longitudes}
still corresponds to the composition $\eta_n(\tau^\iinfty_n(\cW))$. 
But this is clear since all the basis curves from the $S_i$ represent trivial elements in 
$\pi_1(B^4\setminus G^c)$. 
\end{proof}




   \end{document}